\documentclass[11pt]{article}

\usepackage[sort,numbers]{natbib}
\usepackage{algorithm}
\usepackage{algpseudocode}

\usepackage[bookmarks=true,bookmarksnumbered=true,colorlinks=true,allcolors=black,driverfallback=dvipdfmx]{hyperref}
\usepackage{booktabs}

\usepackage[dvipdfmx]{graphicx}
\usepackage{amsmath,amssymb,amsfonts,amsthm}
\usepackage{lscape}
\usepackage{arydshln}

\usepackage[capitalize]{cleveref}

\usepackage{geometry}
\usepackage{indentfirst}

\usepackage{enumitem,color}

\newtheorem{theorem}{Theorem}[section]
\newtheorem{lemma}{Lemma}[section]
\newtheorem{proposition}{Proposition}[section]

\theoremstyle{definition}

\newtheorem*{rmk*}{Remark}
\newtheorem{rmk}{Remark}[section]

\DeclareMathOperator{\Var}{Var}
\DeclareMathOperator{\Cov}{Cov}

\DeclareMathOperator*{\argmax}{arg\,max}

\DeclareMathOperator{\E}{E}
\DeclareMathOperator{\pr}{P}

\DeclareMathOperator{\leb}{Leb}

\allowdisplaybreaks

\numberwithin{equation}{section}

\makeatletter
    \renewcommand*{\section}{\@startsection{section}{1}{\z@}%
    {6pt}{3pt}{\reset@font\normalsize\bfseries}}
\makeatother
    
\makeatletter
    \renewcommand*{\subsection}{\@startsection{subsection}{2}{\z@}%
    {3pt}{3pt}{\reset@font\normalsize\mdseries\itshape}}
\makeatother

\makeatletter
    \renewcommand*{\subsubsection}{\@startsection{subsubsection}{3}{\z@}%
    {3pt}{3pt}{\reset@font\normalsize\mdseries\itshape}}
\makeatother

\makeatletter
\def\@listi{\leftmargin\leftmargini
  \topsep=.5\baselineskip 
  \partopsep=0pt \parsep=0pt \itemsep=0pt}
\let\@listI\@listi
\@listi
\def\@listii{\leftmargin\leftmarginii
  \labelwidth\leftmarginii \advance\labelwidth-\labelsep
  \topsep=0pt \partopsep=0pt \parsep=0pt \itemsep=0pt}
\def\@listiii{\leftmargin\leftmarginiii
  \labelwidth\leftmarginiii \advance\labelwidth-\labelsep
  \topsep=0pt \partopsep=0pt \parsep=0pt \itemsep=0pt}
\def\@listiv{\leftmargin\leftmarginiv
  \labelwidth\leftmarginiv \advance\labelwidth-\labelsep
  \topsep=0pt \partopsep=0pt \parsep=0pt \itemsep=0pt}
\makeatother

\makeatletter
\newcommand{\opnorm}{\@ifstar\@opnorms\@opnorm}
\newcommand{\@opnorms}[1]{%
  \left|\mkern-1.5mu\left|\mkern-1.5mu\left|
   #1
  \right|\mkern-1.5mu\right|\mkern-1.5mu\right|
}
\newcommand{\@opnorm}[2][]{%
  \mathopen{#1|\mkern-1.5mu#1|\mkern-1.5mu#1|}
  #2
  \mathclose{#1|\mkern-1.5mu#1|\mkern-1.5mu#1|}
}
\makeatother

\def\be#1{\begin{equation*}#1\end{equation*}}
\def\ben#1{\begin{equation}#1\end{equation}}
\def\bes#1{\begin{equation*}\begin{split}#1\end{split}\end{equation*}}
\def\besn#1{\begin{equation}\begin{split}#1\end{split}\end{equation}}

\def\bm#1{\begin{multline*}#1\end{multline*}}

\def\ba#1{\begin{align*}#1\end{align*}}
\def\ban#1{\begin{align}#1\end{align}}

\newcommand{\eps}{\varepsilon}
\newcommand{\ds}{\mathcal{X}^\mathrm{raw}}
\newcommand{\dsr}{\mathcal{X}^\mathrm{rel}}
\newcommand{\dsn}{\mathcal{X}^0}

\newcommand{\tri}{\mathrm{tri}}

\newcommand{\ul}[1]{\underline{#1}}

\newcommand{\wt}[1]{\widetilde{#1}}
\newcommand{\mf}[1]{\mathfrak{#1}}
\newcommand{\mcl}[1]{\mathcal{#1}}

\newcommand{\norm}[1]{\left\|#1\right\|}
\newcommand{\bra}[1]{\left(#1\right)}
\newcommand{\cbra}[1]{\left\{#1\right\}}
\newcommand{\sbra}[1]{\left[#1\right]}
\newcommand{\abs}[1]{\left|#1\right|}

\makeatletter
\newcommand{\pushright}[1]{\ifmeasuring@#1\else\omit\hfill$\displaystyle#1$\fi\ignorespaces}
\newcommand{\pushleft}[1]{\ifmeasuring@#1\else\omit$\displaystyle#1$\hfill\fi\ignorespaces}
\makeatother

\title{
On lead-lag estimation of non-synchronously observed point processes
} 
\author{
Takaaki Shiotani\thanks{Graduate School of Mathematical Sciences, The University of Tokyo, 3-8-1 Komaba, Meguro-ku, Tokyo 153-8914 Japan}
\and
Takaki Hayashi\thanks{Graduate School of Business Administration, Keio University, 4-1-1 Hiyoshi, Yokohama 223-8526, Japan}
\and
Yuta Koike\footnotemark[1]
}

\begin{document}

\maketitle

\begin{abstract}

This paper introduces a new theoretical framework for analyzing lead-lag relationships between point processes, with a special focus on applications to high-frequency financial data. 
In particular, we are interested in lead-lag relationships between two sequences of order arrival timestamps. 
The seminal work of Dobrev and Schaumburg proposed model-free measures of cross-market trading activity based on cross-counts of timestamps. 
While their method is known to yield reliable results, it faces limitations because its original formulation inherently relies on discrete-time observations, an issue we address in this study. 
Specifically, we formulate the problem of estimating lead-lag relationships in two point processes as that of estimating the shape of the cross-pair correlation function (CPCF) of a bivariate stationary point process, a quantity well-studied in the neuroscience and spatial statistics literature. 
Within this framework, the prevailing lead-lag time is defined as the location of the CPCF's sharpest peak. 
Under this interpretation, the peak location in Dobrev and Schaumburg's cross-market activity measure can be viewed as an estimator of the lead-lag time in the aforementioned sense. 
We further propose an alternative lead-lag time estimator based on kernel density estimation and show that it possesses desirable theoretical properties and delivers superior numerical performance. 
Empirical evidence from high-frequency financial data demonstrates the effectiveness of our proposed method. 
\smallskip


\end{abstract}

\section{Introduction}

Empirical research on lead-lag relationships between two financial time series has long been an active area of study in finance. 
%
%
Their identification is fundamental to understanding price discovery and may provide practitioners with opportunities for excess profits.
In modern financial markets, such relationships can persist only over very short horizons, even on the order of one millisecond or less. 
Therefore, lower-frequency or coarsely aggregated data inevitably fail to find the fine structure of these relationships.
This motivates the use of \emph{tick data}, i.e., raw high-frequency data that records all transactions as they arrive randomly and \emph{non-synchronously}. In particular, handling non-synchronicity is a central issue when estimating lead--lag relationships from such data.

Most existing studies have examined high-frequency lead-lag dynamics using price series. 
Prominent approaches include methods based on estimating the cross-covariance function \cite{hoffmann2013estimation,de1997high,huth2014high}, wavelet analysis \cite{hayashi2017multi,hayashi2018wavelet}, local spectral estimation \cite{koike2021inference}, Hawkes process-based multi-asset models \cite{bacry2013some,da2017correlation} and the multi-asset lagged adjustment model of \cite{buccheri2021high}. 
Among these, \citet{hoffmann2013estimation} introduced a simple cross-covariance estimator that can be computed directly from non-synchronously observed returns and proposed estimating the prevailing lead-lag time by locating its maximizer. 
Although their method yields sensible empirical implications due to its intuitive interpretation \cite{huth2014high,bollen2017tail,bangsgaard2024lead,dao2018ultra,alsayed2014ultra,poutre2024profitability}, the resulting lead-lag time estimates are often unstable and unreliable \cite{huth2014high,hayashi2017jpn,bangsgaard2024lead}, presumably because high-frequency price series are affected by market microstructure noise. 
As an alternative method, \citet{dobrev2017high} proposed model-free measurements of the lead-lag relationship between two assets based on cross-counts of their order arrivals. 
Their estimator of lead–lag time has been shown to produce highly stable and reliable estimates in practice; see \cite{dobrev2017high,hayashi2017jpn}. 

However, the Dobrev--Schaumburg method is essentially descriptive, and it is not immediately clear what underlying quantity the method actually estimates.\footnote{Although \cite{dobrev2023high} discuss some asymptotic properties of their measurements when the two timestamp series are independent, they do not clearly specify the underlying estimands.}
Indeed, as we show in \cref{sec:ds}, there exist situations in which their method performs poorly in practice, particularly when the data contain relatively few observations. 
Moreover, implementing their method requires partitioning the observation period into equi-spaced buckets, and the choice of bucket size has a substantial impact on the results. 
Yet, because the method is ``model-free,'' it does not offer a statistical explanation for why such sensitivity arises. 

To address these issues, we reformulate the Dobrev--Schaumburg method from a point process perspective. 
This viewpoint reveals that their measurements essentially estimate shape characteristics of the \emph{cross-pair correlation function} (CPCF) of a bivariate point process generated by order arrivals; see \cref{sec:pp} for definitions. 
Accordingly, the Dobrev--Schaumburg estimator of lead-lag time can be interpreted as an estimator of the CPCF's sharpest peak location. 
This interpretation also clarifies that the instability observed in their method arises when the bucket size is chosen too small relative to a range that is permissible given the properties of the underlying data. 
At the same time, because their estimator can only take values that are integer multiples of the bucket size, using a larger bucket size results in excessively coarse estimates. 
To overcome these limitations, we propose a nonparametric, kernel-based estimator of the lead–lag time, together with a data-driven bandwidth selection procedure. We show both theoretically and empirically that this new estimator produces stable and accurate results even in settings where the Dobrev–Schaumburg method fails.

The remainder of the paper is organized as follows. 
\cref{sec:ds} provides a detailed explanation of the Dobrev--Schaumburg method. 
\cref{sec:pp} introduces a point process framework that clarifies the theoretical meaning of the Dobrev--Schaumburg method. 
\cref{sec:kde} proposes an alternative estimator of the lead-lag time within this framework and develops its theoretical properties. 
\cref{sec:simulation} demonstrates its superior numerical performance through a comprehensive Monte Carlo study. 
\cref{sec:empirical} presents an empirical application that illustrates the effectiveness of our proposed estimator using real data. 
\cref{sec:conclusion} concludes by summarizing our main contributions and discussing directions for future research. 
The appendix contains mathematical proofs and additional implementation details. 

\paragraph{Notation} 
The cardinality of a finite set $S$ is denoted by $|S|$. 
$\leb$ denotes the Lebesgue measure. 
The Borel $\sigma$-algebra of a topological space $S$ is denoted by $\mcl B(S)$. 
For a real-valued function $f$ defined on a set $S$, we set $\|f\|_\infty=\sup_{x\in S}|f(x)|$. 
Also, we denote by $\argmax_{x\in S}f(x)$ the set of maximizers of $f$ on $S$. 
For a random variable $X$ and $p\geq1$, we set $\|X\|_p:=(\E[|X|^p])^{1/p}$. 
The underlying probability space is denoted by $(\Omega,\mcl F,\pr)$.
We interpret $1/0=\infty$ by convention. 

\section{The Dobrev--Schaumburg method}\label{sec:ds}

Suppose that we have tick data for two financial assets, with timestamps given by $0\leq t^1_1<\cdots<t^1_{n_1}\leq T$ and $0\leq t^2_1<\cdots<t^2_{n_2}\leq T$. 
Throughout the paper, we assume that $T\geq1$ is an integer and timestamps are expressed in seconds when working with real data.; hence we may regard $T$ as a large integer. 
\citet{dobrev2017high} proposed measuring the lead-lag relationship between two assets using the following procedure. 

First, divide the observation interval $[0,T]$ into equi-spaced time buckets $I^h_k:=(kh,(k+1)h]$, $k=0,1,\dots,T/h-1$, where $h>0$ is chosen so that $h^{-1}\in\mathbb N$. 
We interpret a situation where an event for asset 2 occurs with a lag $\ell\in\mathbb Z$ after an event for asset 1 as the existence of timestamps $t^1_i,t^2_j$ and a bucket index $k\in\{|\ell|, |\ell|+1,\dots,T/h-1-|\ell|\}$ such that $t^1_i\in I^h_k$ and $t^2_j\in I^h_{k+\ell}$.
By calculating the number of such bucket indices, we obtain a measure of the lead-lag effect of asset 1 on asset 2 with lag $\ell$. 
Based on this idea, \citet{dobrev2017high} introduced the \emph{raw cross-market activity} at offset $\ell$ as
\[
\ds_h(\ell):=\sum_{k=|\ell|}^{T/h-1-|\ell|}1_{\{\exists t^1_i\in I^h_k,\,\exists t^2_j\in I^h_{k+\ell}\}}.
\] 
To adjust the degree of freedom, they also defined the \emph{relative cross-market activity} as
\[
\dsr_h(\ell):=\frac{\ds_h(\ell)}{\min\cbra{\sum_{k=|\ell|}^{T/h-1-|\ell|}1_{\{\exists t^1_i\in I^h_k\}},\ \sum_{k=|\ell|}^{ T/h-1-|\ell|}1_{\{\exists t^2_{i+\ell}\in I^h_k\}}}}.
\]
Computing $\ds_h(\ell)$ and $\dsr_h(\ell)$ requires choosing a value of $h$. 
In \cite{dobrev2017high}, $h$ is set to 1 millisecond, which is chosen in an ad-hoc manner by considering the time granularity of the data. 
Here, 1 millisecond is the minimum time unit available in their dataset for the S\&P 500 cash market. 

By definition, the larger value of $\dsr_h(\ell)$ indicates a stronger lead-lag effect in which asset 2 follows asset 1 after a delay of $\ell h$. 
Motivated by this, \citet{dobrev2017high} proposed identifying the prevailing lead-lag time by locating the peak of the map $\ell\mapsto\dsr_h(\ell)$. 
Specifically, given a search grid 
$\mcl G_h$, the lead-lag time is estimated by 
\[
\hat\theta_h^{DS}=\hat \ell h,\quad\text{where }\hat \ell\in\argmax_{\ell\in\mcl G_h}\dsr_h(\ell).
\]
We refer to $\hat\theta_h^{DS}$ as the \emph{DS estimator}. 
In what follows, we assume the true lead-lag time lies within the interval $(-r,r)$ for some known positive constant $r$, and define $\mcl G_h:=\{\ell\in\mathbb Z:|\ell h|\leq r\}$. 

As mentioned in the introduction, several empirical studies have reported that the DS estimator yields stable and interpretable estimates of lead-lag time. 
However, there are cases where the estimator performs poorly, particularly when the dataset contains relatively few observations. 
\cref{fig:ds-vs-cpcf} illustrates this issue, showing the relative cross-market activity measure computed from best-quote updates on two U.S.~stock exchanges, NASDAQ and BATS, for the MNST stock on August 12, 2015. 
In this example, the values of the cross-market activity measure fluctuate heavily, making it difficult to identify the peak reliability. 
Yet, due to the ``model-free'' nature of the Dobrev--Schaumburg method, no statistical interpretation is provided for the origin of such instability. 
One goal of this study is to establish a theoretical foundation for their approach and fill this gap. 

\begin{figure}[ht]
\centering
\includegraphics[scale=0.8]{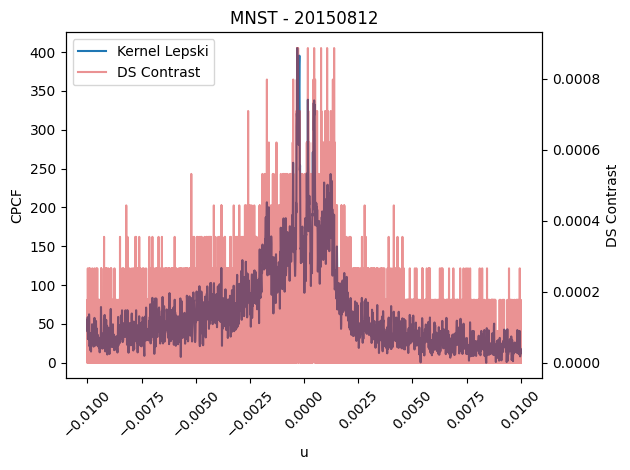}
%
\caption{Contrast functions: DS vs ours. MNST, NASDAQ vs BATS (quote, Aug. 12, 2015). The cross-market activity measure $\dsr_h(\ell)$ is indicated by the red line, while the kernel density estimator $g_{\hat h}(u)$ with the Lepski-selected bandwith (see \cref{sec:kde}) is indicated by the blue line. The unit of the horizontal axis is seconds.}
\label{fig:ds-vs-cpcf}
\end{figure}

\section{Proposed framework}\label{sec:pp}

To clarify the statistical meaning of the Dobrev--Schaumburg method, we model the observed timestamps $(t^1_i)_{i=1}^{n_1}$ and $(t^2_j)_{j=1}^{n_2}$ as realizations of a bivariate point process on the real line. 
Specifically, for each $a=1,2$, we consider the timestamps $(t^a_i)$ to be a result of observing a point process $N_a$ on $\mathbb R$ over the interval $[0,T]$. 
That is, $N_a(A)=|\{i:t^a_i\in A\}|$ for a Borel set $A\subset[0,T]$. 
Here and below, we mainly follow the mathematical formulation of point processes described in \cite{daley2006introduction,daley2007introduction} and refer to these monographs for unexplained concepts and notation (see also \cite[Section 2.2]{shiotani2024statistical} for a summary). 
With this formulation, we can rewrite $\dsr_h(\ell)$ as
\[
\dsr_h(\ell)=\frac{\sum_{k=|\ell|}^{T/h-1-|\ell|}1_{\{N_1(I_{k}^h)>0,\ N_2(I_{k+\ell}^h)>0\}}}{\min\cbra{\sum_{k=|\ell|}^{T/h-1-|\ell|}1_{\{N_1(I_{k}^h)>0\}},\ \sum_{k=|\ell|}^{T/h-1-|\ell|}1_{\{N_2(I_{k+\ell}^h)>0\}}}}.
\]
Using this expression, we can relate $\dsr_h(\ell)$ to the \emph{cross-pair correlation function (CPCF)} of $N$. 

To state the theoretical result formally, we introduce several assumptions and notation. 
We assume that $N=(N_1,N_2)$ is a simple stationary bivariate point process on $\mathbb R$ with intensities $\lambda_1,\lambda_2\in(0,\infty)$. 
Note that we use the term ``intensity'' in the same sense as in \cite{daley2006introduction} (see page 47 ibidem). 
By \cite[Proposition 3.3.IV]{daley2006introduction}, we have $\lambda_a=\E[N_a((0,1])]$ for $a=1,2$. 
We also assume that there exists a locally integrable function $\lambda_{12}:\mathbb R\to[0,\infty]$ such that
\ben{\label{def:cross-intensity}
\E[N_1(A_1)N_2(A_2)]
=\int_{A_1\times A_2}\lambda_{12}(y-x)dxdy
}
for any bounded $A_1,A_2\in\mcl B(\mathbb R)$. 
We refer to $\lambda_{12}$ as the \emph{cross-intensity function}\footnote{In neuroscience, the term ``cross-intensity function'' usually refers to the functions $\lambda_{12}(u)/\lambda_1$ or $\lambda_{12}(u)/\lambda_2$ (see e.g.~\cite{bryant1973correlations}). 
We follow the terminology used in spatial statistics \cite{hessellund2022semiparametric,hessellund2022second}. 
In the terminology of point process theory, $\lambda_{12}$ is a density of the reduced cross-moment measure of $N$ (cf.~\cite[Section 8.3]{daley2006introduction}).} of $N$.  
The CPCF of $N$ is the function $g:\mathbb R\to[0,\infty]$ defined as
\[
g(u)=\frac{\lambda_{12}(u)}{\lambda_1\lambda_2}\qquad(u\in\mathbb R).
\]
The cross-intensity function can be related to the cross-covariance function between the infinitesimal increments of $N_1$ and $N_2$ in the following sense (cf.~\eqref{ds-kernel}):
\[
\nu_{12}(u):=\lim_{h\downarrow0}\frac{\Cov(N_1(0,h],N_2(u,u+h])}{h^2}=\lambda_{12}(u)-\lambda_1\lambda_2\quad\text{a.e. }u.
\]
The function $\nu_{12}$ is called the \emph{covariance density} of $N$. 
In this sense, $g(u)$ measures the cross-covariation between $N_1(\cdot)$ and $N_2(\cdot+u)$. 
%

We also need the notion of $\alpha$-mixing (or strong mixing) for point processes. 
Recall that the $\alpha$-mixing coefficient of two sub-$\sigma$-algebras $\mcl G$ and $\mcl H$ of $\mcl F$ is defined as 
\[
\alpha(\mcl G, \mcl H) := \sup \{|\pr(C \cap D) - \pr(C)\pr(D)|: C \in \mcl G, D \in \mcl H \}.
\]
For $E\in\mcl B(\mathbb R)$, we denote by $N\cap E=(N_i\cap E)_{i=1}^2$ the restriction of $N$ to $E$, i.e.~$(N_i\cap E)(A)=N_i(A\cap E)$ for $i=1,2$ and $A\in\mcl B(\mathbb R)$. 
Also, $E\oplus r:=\{x\in\mathbb R:|y-x|<r\text{ for some }y\in E\}$ denotes the $r$-enlargement of $E$. 
Moreover, given a bivariate point process $M=(M_1,M_2)$ on $\mathbb R$, $\sigma(M)$ denotes the $\sigma$-algebra generated by $\bigcup_{i=1}^2\{M_i(A):A\in\mcl B(\mathbb R)\}$. 
As $\alpha$-mixing coefficients of $N$, we adopt the following definition:
\besn{\label{def:alpha-mixing}
  \alpha_{c_1, c_2}^N(m; r) =
  \sup \Bigl\{
   & \alpha(\sigma(N\cap E_1), \sigma(N\cap E_2)):   
    E_1 = \bigcup_{j\in J_1}I_j\oplus r,\, E_2 = \bigcup_{j\in J_2}I_j\oplus r, \\
   & |J_1|\leq c_1,\, |J_2|\leq c_2,\,
  d(J_1, J_2)\geq m,\, J_1, J_2\subset \mathbb Z
  \Bigr\}, \quad m, c_1, c_2, r \geq 0,
}
where $I_j:=I_j^1=(j,j+1]$ and $d(J_1,J_2):=\inf\{|j_1-j_2|:j_1\in J_1,\,j_2\in J_2\}$. 
This definition is a minor variant of the one used in \cite{shiotani2024statistical}, where they use $(j-\frac{1}{2},j+\frac{1}{2}]$ instead of $I_j$. 
This difference is inessential, and we just adopt the present definition to (slightly) simplify our technical arguments. 

\begin{rmk}[Mixing coefficients as a continuous-time process]
Since $N$ can be regarded as a stochastic process indexed by $\mathbb R$, we can also define the $\alpha$-mixing coefficients in this sense \cite{brillinger1976estimation}:
\[
\alpha^N_{\text{proc}}(\tau):=\sup_{t\in\mathbb R}\alpha\bra{\sigma(N\cap(-\infty,t)),\,\sigma(N\cap(t+\tau,\infty))},\qquad\tau\geq0.
\]
Our definition is weaker than this version in the sense that we have
\ben{\label{mixing-rel}
\alpha^N_{c_1,c_2}(m;r)\leq\alpha^N_{\text{proc}}(m-2r)
}
for all $c_1,c_2,r\geq0$ and $m\geq2r$. 
For the case of point processes, it is important to work with the weaker version because it is sometimes difficult to bound the left hand side of \eqref{mixing-rel} uniformly in $c_1,c_2\geq0$; see \cite[Proposition 2.8]{poinas2019mixing} and \cite[Lemma 10.13]{shiotani2024statistical} for example. 
\end{rmk}

We impose the following regularity conditions on $N$. 
\begin{enumerate}[label={\normalfont[A1]}]




\item\label{ass:pp} (i) For every $p\geq1$, there exists a constant $B_p>0$ such that $\max_{i=1,2}\lambda_i^{-1}\|N_i((0,1])\|_p\leq B_p$. 

(ii) For any $p,q\geq1$, there exists a constant $B_{p,q}>0$ such that 
$
\alpha_{p,p}^N(m;r_1)\leq B_{p,q}m^{-q}
$
for all $m\in\mathbb N$, where $r_1:=r+1$.  

\end{enumerate}
This assumption is fairly reasonable in the literature and is satisfied by many standard point process models such as the Hawkes process and Neyman-Scott process as long as their kernels satisfy some regularity conditions; see \cref{sec:sim-models} for details. 

Under \ref{ass:pp} and additional technical assumptions, we have the following asymptotic representation of Dobrev--Schaumburg's cross-market activity measure:
\begin{proposition}\label{prop:ds-kernel}
Assume \ref{ass:pp}. 
Assume also that $g$ is bounded and
\besn{\label{ds-extra}
\max_{\ell\in\mcl G_h}\E[N_1(I_0^h)\{N_1(I_0^h)-1\}N_2( I_\ell^h)]&=o(h^{1+\varpi}),\\
\max_{\ell\in\mcl G_h}\E[N_1(I_0^h)N_2(I_\ell^h)\{N_2(I_\ell^h)-1\}]&=o(h^{1+\varpi})
}
as $h\to0$ for $\varpi=1$. 
Moreover, assume $h=h_T\asymp T^{-\gamma}$ as $T\to\infty$ for some $0<\gamma<1$. 
Then
\ba{
\max_{\ell\in\mcl G_h}\abs{\frac{\dsr_h(\ell)}{h}-(\lambda_1\vee\lambda_2)\int_{\mathbb R}\frac{1}{h}K^\tri\bra{\frac{u-\ell h}{h}}g(u)du}\to^p0,
}    
where $K^\tri(x)=(1-|x|)1_{[-1,1]}(x)$.
\end{proposition}

\begin{rmk}[On condition \eqref{ds-extra}]
The quantities on the left hand sides of \eqref{ds-extra} can be related to the factorial moment measures of orders (2,1) and (1,2) of $N$ (see Section 2.3 of \cite{shiotani2024statistical} for the definition). 
In particular, \eqref{ds-extra} holds for $\varpi=1$ if these measures have bounded densities with respect to the Lebesgue measure. 
Since each factorial moment measure can be expressed as a sum of factorial cumulant measures through \cite[Eq.(2.5)]{shiotani2024statistical} (see also \cite[Eq.(3.21)]{brillinger1972spectral}), its density can be computed for the Hawkes process via \cite[Eq.(39)]{jovanovic2015cumulants} and the Neyman-Scott process via \cite[Eq.(5.2)]{shiotani2024statistical}, respectively; hence, one can in principle verify \eqref{ds-extra} under appropriate assumptions on their kernels. 
We do not pursue this point further because this condition is unnecessary for the theoretical development of our new estimator proposed in the next section.  
\end{rmk}

Under the assumptions of \cref{prop:ds-kernel}, we have $\max_{\ell\in\mcl G_h}|\dsr_h(\ell)/h-(\lambda_1\vee\lambda_2)g(\ell h)|\to^p0$ if $g$ is continuous. 
Hence, the cross-market activity measure can be interpreted as an estimator for the CPCF up to a multiplicative constant. 
Moreover, if $g$ has a unique maximizer $\theta^*$ in $(-r,r)$, the above result implies that $\hat\theta_h^{DS}$ is a consistent estimator of $\theta^*$ as $T\to\infty$. 
For further understanding of theoretical properties of $\hat\theta_h^{DS}$, we study its rate of convergence. 
For this purpose, we introduce the following assumption on $g$:
\begin{enumerate}[label={\normalfont[A2]}]

\item\label{ass:cpcf} 
There exist constants $\theta^*\in(-r,r)$, $\alpha\in(0,1)\cup(1,\infty)$, $b>1$ and $\delta>0$ such that the following conditions hold: 
\begin{enumerate}[label={\normalfont(\roman*)}]

\item\label{type-I} 
If $\alpha>1$, $\sup_{u\in\mathbb R}g(u)\leq b$ and 
\[
\min\cbra{\sup_{0<u-\theta^*<\delta}\frac{g(\theta^*)-g(u)}{|u-\theta^*|^{\alpha-1}},\sup_{0<\theta^*-u<\delta}\frac{g(\theta^*)-g(u)}{|u-\theta^*|^{\alpha-1}}}\leq b
\]
and
\[
\inf_{0<|u-\theta^*|<\delta}\frac{g(\theta^*)-g(u)}{|u-\theta^*|^{\alpha-1}}\geq \frac{1}{b}
\]
and
\[
\sup_{|u-\theta^*|\geq\delta}g(u)\leq g(\theta^*)-\frac{1}{b}.
\]

\item\label{type-II} 
If $\alpha<1$, 
\[
\max\cbra{\inf_{0<u-\theta^*<\delta}\frac{g(u)}{|u-\theta^*|^{\alpha-1}},\inf_{-\delta<u-\theta^*<0}\frac{g(u)}{|u-\theta^*|^{\alpha-1}}}\geq \frac{1}{b}.
\]
Moreover, there exist a constant $\alpha<\alpha_0\leq1$ and a measurable function $g_0:\mathbb R\to[0,\infty]$ such that $\|g_0\|_{L^{1/(1-\alpha_0)}([-r_1,r_1])}\leq b$ and
\[
g(u)\leq b\bra{g_0(u)+|u-\theta^*|^{\alpha-1}}\quad\text{for all }u\in[-r_1,r_1].
\]

\end{enumerate}

\end{enumerate}

Under \ref{ass:cpcf}\ref{type-I}, $g$ is bounded and $\theta^*$ is the unique maximizer of $g$. 
By contrast, under \ref{ass:cpcf}\ref{type-II}, $g$ diverges at $\theta^*$ and is therefore unbounded. 
We allow $g$ to have poles other than $\theta^*$ through an auxiliary function $g_0$; however, the condition $\|g_0\|_{L^{1/(1-\alpha_0)}([-r_1,r_1])}<\infty$ ensures that $\theta^*$ remains the ``sharpest'' peak location of $g$. 
We call $\theta^*$ the \emph{lead-lag (time) parameter}. 
We allow $g$ to be unbounded motivated by several empirical observations: (i) Dobrev–Schaumburg's cross-market activity measure often exhibits extremely sharp peaks (see e.g.~\cite[Fig.~7]{dobrev2017high}); (ii) \citet{shiotani2024statistical} found that their semiparametric model fits Japanese stock market data better when the CPCF is unbounded; (iii) \citet{rambaldi2018detection} report that intensity burst occurrences across different foreign exchange rates exhibit lead-lag relationships. 

\begin{rmk}[Relation to mode estimation]\label{rmk:mode}
Our formulation of the lead-lag parameter estimation problem is naturally connected to mode estimation for a probability density function, once the CPCF is replaced by the density. 
\ref{ass:cpcf} is motivated by this observation. 
In fact, \citet{wegman1971note} classifies the problem of mode estimation for unimodal distributions into three types, labeling cases with bounded density as Type I, cases with unbounded density as Type II, and cases without density as Type III. 
In this classification, \ref{ass:cpcf}\ref{type-I} corresponds to Type I, and \ref{ass:cpcf}\ref{type-II} corresponds to Type II. 
For the purpose of analyzing the convergence rate, we assume that the CPCF behaves like a power function $|u-\theta^*|^{\alpha-1}$ in a neighborhood of $\theta^*$. 
When $\alpha>1$, this assumptions is analogous to the conditions studied in \cite{arias2022estimation} (see Section 1.1 ibidem). 
When $\alpha<1$, it corresponds to the analogue of condition [\textbf{H-III}] in \cite{bercu2002estimation}. 
Note that the case $\alpha=1$ is excluded simply because the power function $|u-\theta^*|^{\alpha-1}$ becomes constant in that case. 
\end{rmk}
%
%
Under \ref{ass:cpcf}, we set 
\[
\beta_\alpha:=\alpha\vee(2\alpha-1)=\begin{cases}
\alpha & \text{if }\alpha<1,\\
2\alpha-1 & \text{if }\alpha>1.
\end{cases}
\]

\begin{theorem}\label{thm:ds}
Assume \ref{ass:pp}, \ref{ass:cpcf} and $h=h_T\asymp T^{-\gamma}$ as $T\to\infty$ for some $0<\gamma<1/\beta_\alpha$. 
Moreover, assume \eqref{ds-extra} for $\varpi=\alpha$. 
Then, $\hat\theta_h^{DS}=\theta^*+O_p(h)$ as $T\to\infty$.
\end{theorem}

Since $\hat\theta_h^{DS}$ is, by construction, an integer multiple of $h$, \cref{thm:ds} implies that its convergence rate is exactly of order $O(h)$. 
In particular, the accuracy of $\hat\theta_h^{DS}$ improves as $h$ becomes smaller. 
Given the definition of $\dsr_h(\ell)$, however, it is meaningless to choose $h$ smaller than the minimum time unit in the data. 
In this sense, it is reasonable that \citet{dobrev2017high} set $h$ equal to the minimum time resolution. 

Nevertheless, \cref{thm:ds} also imposes a constraint on the order of $h$, requiring that $h$ converge to 0 more slowly than $T^{-1/\beta_\alpha}$. 
Moreover, this constraint is not a technical artifact of the proof; it is essential, because we show that $T^{-1/\beta_\alpha}$ gives a minimax lower bound on the convergence rate of any estimator of $\theta^*$ under the assumptions of \cref{thm:ds}, at least when $\alpha<2$ (see \cref{rmk:ds-extra-bp}). 
Therefore, it is theoretically impossible for the conclusion of \cref{thm:ds} to hold when $h\ll T^{-1/\beta_\alpha}$. 

Note that $n_a/T \to \lambda_a > 0$ as $T \to \infty$ a.s.~for each $a=1,2$ under \ref{ass:pp}. 
Hence, for $\hat\theta_h^{DS}$ to perform properly, $h$ must be chosen based on $n_1,n_2$ (the sample sizes) and $\alpha$ (the sharpness of the CPCF peak). 
Specifically, $h$ should be taken larger when $n_1$ and $n_2$ are smaller and/or when $\alpha$ is larger. 
This provides one explanation for the poor performance of the Dobrev--Schaumburg method in the example shown in \cref{fig:ds-vs-cpcf}. 
That is, in that dataset, setting $h=1~\mu\text{s}$ was likely far too small, given the relatively small numbers of observations $(n_1,n_2)=(28048,11287)$. 

These observations indicate that the choice of $h$ plays a crucial role in the implementation of the DS estimator. 
However, the DS estimator should be interpreted as an estimator of an interval containing the lead-lag parameter, rather than the parameter itself, which makes data-driven selection of $h$ difficult. 
For these reasons, in the next section, we propose a new lead-lag time estimator based on kernel density estimation and demonstrate that it can overcome this issue. 


\section{New estimator}\label{sec:kde}

\cref{prop:ds-kernel} suggests that $\dsr_h(\ell)$ would be asymptotically equivalent to a discretized version of a kernel density estimator for $g$ based on the triangular kernel. 
This naturally motivates us to consider kernel density estimators for $g$ directly. 

Formally, let $K:\mathbb R\to\mathbb R$ be a kernel function and $h=h_T>0$ a bandwidth parameter such that $h\to0$ as $T\to\infty$. 
We then consider the following statistic:
\[
\hat g_h(u)=\frac{T}{n_1n_2}\int_{(0,T]^2}K_h(y-x-u)N_1(dx)N_2(dy),\qquad u\in\mathbb R,
\]
where $K_h(t)=h^{-1}K(t/h)$ for $t\in\mathbb R$. 
We estimate $\theta^*$ by taking a maximizer of $\hat g_h$. 
That is, we define a random variable $\hat\theta_h$ satisfying
\[
\hat\theta_h\in\argmax_{u\in[-r,r]}\hat g_h(u).
\]
The practical procedure for computing $\hat\theta_h$ and its computational complexity are described in Appendix \ref{app:computation}.

When the uniform kernel is used as the kernel function, $\hat g_h$ is essentially equivalent to the so-called \textit{cross-correlation histogram} in neuroscience and has long been applied to investigate relationships between neuronal spikes (see e.g.~\cite{bryant1973correlations}). 
This line of work has motivated statisticians to investigate the theoretical properties of $\hat g_h$ \citep{cox1972multivariate,brillinger1975statistical,brillinger1976estimation,ellis1991density}. 
Nevertheless, to the best of our knowledge, no prior research has addressed the case where $g$ is unbounded, nor the asymptotic behavior of the maximizer of $\hat g_h$. 

We impose the following assumption on the kernel. 
\begin{enumerate}[label={\normalfont[K]}]

\item\label{kernel} $K$ is non-negative, continuous at 0, of bounded variation and supported on $[-1,1]$ such that $K(0)>0$, $\int_{-\infty}^\infty K(t)dt=1$ and $\argmax_{u\in[-r,r]}\hat g_h(u)\neq\emptyset$ a.s.

\end{enumerate}
Assumption \ref{kernel} holds for the uniform kernel $K=\frac{1}{2}1_{[-1,1]}$ and the triangular kernel $K=K^\tri$, for example. 

\begin{theorem}\label{thm:cpcf}
Assume \ref{ass:pp}, \ref{ass:cpcf} and \ref{kernel}. 
Let $\eta>0$ be a constant. 
Then, there exist constants $A>1$ and $0<h_0<1$ depending only on $\alpha,\alpha_0,\delta,b$ and $K$ such that
\ben{\label{eq:cpcf}
\pr(|\hat\theta_h-\theta^*|>Ah)\leq\frac{C}{\sqrt{Th^{\beta_\alpha+\eps}}}
}
for all $h\leq h_0\wedge T^{-\eta}$ and $\eps>0$, where $C>0$ is a constant depending only on $r,\alpha,\delta,b$, $(B_p)_{p\geq1}$, $(B_{p,q})_{p,q\geq1}$, $\eps,\eta$ and $\|K\|_\infty$. 
In particular, $\hat\theta_h=\theta^*+O_p(h)$ as $T\to\infty$ if $h=h_T\asymp T^{-\gamma}$ for some $0<\gamma<1/\beta_\alpha$. 
\end{theorem}

By \cref{thm:cpcf}, $\hat\theta_h$ estimates $\theta^*$ at the same convergence rate as $\hat\theta_h^{DS}$. 
Moreover, unlike the DS estimator, we do not require assumption \eqref{ds-extra}. 
In particular, by selecting the bandwidth appropriately, our estimator nearly attains the minimax optimal convergence rate $T^{-1/\beta_\alpha}$ (see \cref{thm:minimax}). 
The next subsection discusses methods for selecting the bandwidth in a data-driven way. 

\begin{rmk}[Relation to kernel mode estimation]
In light of \cref{rmk:mode}, our estimator is closely related to a mode estimator obtained by maximizing a kernel density estimator. 
Following \cite{chacon2020modal}, we refer to this type of estimator as the kernel mode estimator. 

For i.i.d.~data with density $f$ and unique population mode $\theta^*$, 
\citet{parzen1962estimation} established the asymptotic normality of the kernel mode estimator when $f$ is of class $C^2$ and $f''(\theta^*)<0$. This setting is a particular case of \ref{ass:cpcf} with $\alpha=3$. 
Notably, with a suitable bandwidth choice, the kernel mode estimator nearly achieves the convergence rate $n^{-1/5}$, where $n$ is the sample size. Since $\beta_3=5$, our estimator enjoys an analogous property. 
\citet{hasminskii1979lower} proved that the rate $n^{-1/5}$ is minimax optimal in this setting, but it can be improved under additional smoothness assumptions on the density and the kernel; see e.g.~\cite{eddy1980optimum,klemela2005adaptive,vieu1996note}. 

Without smoothness assumptions on the density, \citet{abraham2003simple} and \citet{herrmann2004rates} obtained convergence rates for the kernel mode estimator and its variant under conditions analogous to \ref{ass:cpcf}\ref{type-I}. 
In this scenario, \citet{arias2022estimation} established the minimax optimal rate and developed an adaptive estimation procedure. 
The convergence rate in this setting is $n^{-1/(2\alpha-1)}$, which again analogous to ours. 
To the best of our knowledge, apart from the work of \cite{bercu2002estimation}, no results exist for mode estimation under assumptions analogous to \ref{ass:cpcf}\ref{type-II}.
\citet{bercu2002estimation} investigated a histogram-type estimator, which can be viewed as a discretized kernel mode estimator, and showed that its convergence rate can be made arbitrarily close to $n^{-1/\alpha}$ by selecting the bin width appropriately. 
This behavior is also analogous to that of our estimator. 

Finally, to our knowledge, the asymptotic distribution of the kernel mode estimator remains unknown for non-smooth densities. 
We conjecture that it may be non-Gaussian when $\alpha<3/2$, drawing a parallel to location parameter estimation in the presence of density singularities (cf.~\cite[Chapter VI]{ibragimov2013statistical}). 
\end{rmk}

\subsection{Bandwidth selection by Lepski's method}\label{sec:lepski}

In the classical setting of i.i.d.\ observations and kernel density estimation, bandwidth selection is typically based on minimizing the mean integrated squared error (MISE) of the kernel estimator; see, for example, \cite[Chapter 1]{tsybakov2008nonparametric}.
In kernel estimation of moment density functions for point processes, analogous MISE-type criteria are also standard \cite{guan2007least, jalilian2018fast}. 

However, for our purposes, a global loss criterion such as MISE is not entirely satisfactory, because the object of interest is not the function $g$ itself but the location $\theta^\ast$ of its peak.
Intuitively, a global criterion aims to fit the entire curve, including regions where $g$ is only moderately large or small, whereas the estimation error of $\theta^\ast$ may be governed almost exclusively by the local behaviour of $g$ in a small neighbourhood of the peak, which may even be singular under condition \ref{ass:cpcf}(ii). 
A closely related issue has been recognized in the literature on nonparametric modal regression with i.i.d.\ observation; see, for example, Section 4.2 in \citet{chen2018modal}.
Motivated by recent works in this context \cite{chen2016nonparametric, zhou2019bandwidth}, we investigate loss-minimization approaches based on cross-validation in Appendix~\ref{sec:cv}.

Here, we instead adopt an adaptive estimation strategy based on a Lepski-type method, i.e., a pairwise comparison of estimators with different bandwidths. 
Our approach was particularly inspired by \citet{klemela2005adaptive}, who has developed a Lepski-type method for mode estimation from i.i.d.~data. 
For textbook treatments of Lepski's method, we refer to \cite[Section 8.2]{gine2016mathematical}.

Fix constants $a>1$, $\gamma_{\max}>0$ and $j_{\min}\in\mathbb N$. 
We consider the following set as candidates for bandwidths:
\[
\mcl H_T:=\{a^{-j}:j_{\min}\leq j\leq \lceil\log_a(T^{\gamma_{\max}})\rceil,\,j\in\mathbb Z\}.
\]
For every $h\in\mcl H_T$, set
\[
\mcl M_h:=\argmax_{u\in[-r,r]}\hat g_h(u).
\]
We define
\[
\hat h:=\min\cbra{h\in\mcl H_T:\bar d(\mcl M_h,\mcl M_{h'})\leq A_Th'\text{ for all }h'\in\mcl H_T\text{ with }h'\geq h},
\]
where $A_T$ is a positive constant and
\[
\bar d(\mcl M_{h},\mcl M_{h'}):=\sup\cbra{|x-y|:x\in \mcl M_{h},\,y\in \mcl M_{h'}}.
\]

\begin{theorem}\label{thm:adaptive}
Assume \ref{ass:pp}, \ref{ass:cpcf} and \ref{kernel}. 
Also, assume $1/\beta_\alpha\leq\gamma_{\max}$. 
Further, assume $A_T\to\infty$ and $A_T=o(T^c)$ for any $c>0$ as $T\to\infty$. 
Then, $\hat\theta_{\hat h}=\theta^*+O_p(T^{-\gamma})$ for any $0<\gamma<1/\beta_\alpha$. 
\end{theorem}

\cref{thm:adaptive} shows that the estimator $\hat\theta_{\hat h}$ nearly achieves the optimal convergence rate $T^{-1/\beta_\alpha}$ without requiring the precise value of $\alpha$. 
This result is numerically validated in \cref{sec:sim-lepski}, where we also assess the robustness of the estimator with respect to the choice of the tuning parameter $A_T$. 
We will see that setting $A_T$ to a constant multiple of $\log\log T$ performs reasonably well.

\subsection{Minimax lower bound for the convergence rate}

In this subsection, we show that when \ref{ass:pp} and \ref{ass:cpcf} hold, $T^{-1/\beta_\alpha}$ gives a minimax lower bound for the convergence rate of any estimator for $\theta^*$. 
For this purpose, we consider a subclass of models for $N$ satisfying \ref{ass:pp} and \ref{ass:cpcf}, which is specified as follows. 
Given a probability density function $g$ on $\mathbb R$, we consider a probability measure $\pr_g$ on $(\Omega,\mcl F)$ having the following properties:
\begin{enumerate}[label=(\roman*)]

\item 
$N_1=\sum_{i=1}^\infty\delta_{t_i}$ is a Poisson process on $\mathbb R$ with unit intensity under $\pr_g$. 

\item 
$N_2$ is of the form $N_2=\sum_{i=1}^\infty\delta_{t_i+\gamma_i}$, where $(\gamma_i)_{i=1}^\infty$ is a sequence of i.i.d.~random variables independent of $N_1$ such that the law of $\gamma_1$ has density $g$ under $\pr_g$. 

\end{enumerate}
%
Under $\pr_g$, $N$ is a bivariate Poisson process on $\mathbb R$ in the sense of \cite[Example 6.3(e)]{daley2006introduction}, where we have $Q_1=Q_2=0$ and $Q_3(dxdy)=g(y-x)dxdy$ in their notation. 
In particular, under $\pr_g$, the CPCF of $N$ is given by $g$. 

We write $\mcl G(\theta^*,\alpha,\delta,b)$ for the class of probability density functions $g:\mathbb R\to[0,\infty]$ supported on $[-1,1]$ and satisfying the conditions of \ref{ass:cpcf}. 
Note that if $g\in\mcl G(\theta^*,\alpha,\delta,b)$, then $N$ satisfies \ref{ass:pp} and \ref{ass:cpcf} for some family of constants $(B_p)_{p\geq1}$. 
In fact, \ref{ass:cpcf} is evident, while \ref{ass:pp}(i) follows from the fact that both $N_1$ and $N_2$ are Poisson processes on $\mathbb R$ with unit intensity. 
Finally, since $g$ is supported on $[-1,1]$, $\alpha_{c_1,c_2}^N(m;r_1)=0$ for $m\geq m_0$, where $m_0$ depends only on $r$. Hence, \ref{ass:pp}(ii) is also satisfied. 
Given this consideration, the following theorem shows that $T^{-1/\beta_\alpha}$ gives a minimax lower bound for the convergence rate of any estimator for $\theta^*$ under \ref{ass:pp} and \ref{ass:cpcf}:

\begin{theorem}\label{thm:minimax}
For any $\alpha\in(0,1)\cup(1,\infty)$, there exists a constant $b>0$ such that
\[
\liminf_{T\to\infty}\inf_{\hat\theta_T}\sup_{|\theta|\leq 2\rho_T}\sup_{g\in\mcl G(\theta,\alpha,1/2,b)}\pr_g\bra{|\hat\theta_T-\theta|\geq \rho_T}>0,
\]
where $\rho_T:=T^{-1/\beta_\alpha}$ and the infimum is taken over all estimators based on $N\cap [0,T]$, i.e.~all $\sigma(N\cap [0,T])$-measurable random variables. 
\end{theorem}


\begin{rmk}\label{rmk:ds-extra-bp}
Since the cumulant measure of order $(p,q)$ of $N$ vanishes if $p\wedge q>1$, condition \eqref{ds-extra} holds whenever $\varpi=1$, and also holds for $\varpi<2$ when $g$ is bounded. 
Therefore, $T^{-1/\beta_\alpha}$ is also a minimax lower bound for the convergence rate of any estimator for $\theta^*$ under the assumptions of \cref{thm:ds} when $\alpha<2$.
\end{rmk}


\section{Simulation study}\label{sec:simulation}
Bivariate Hawkes and Neyman--Scott processes equipped with gamma kernels are used to model the relationship between two series of timestamps in high-frequency financial data; for instance, see \cite{potiron2025mutually} and \cite{shiotani2024statistical}, respectively.
In our experiments, we employ ``lagged'' variants of these models as the data-generating processes.
Using simulated data, we numerically investigate the accuracy, convergence rate, and tuning parameter sensitivity of the DS estimator and our proposed estimator.
We adopt the triangular kernel $K_{\mathrm{tri}}$ for the kernel method in all numerical experiments in this paper.

\subsection{Models}\label{sec:sim-models}
\subsubsection{Lagged bivariate Hawkes process with gamma kernels}

First, we introduce a bivariate Hawkes process with conditional intensity functions
\[
  \lambda_i(t)
  = \mu_i
    + \sum_{t_{k,1} < t} \phi_{i1}(t - t_{k,1})
    + \sum_{t_{k,2} < t} \phi_{i2}(t - t_{k,2}),
  \qquad i = 1,2,
\]
where $\{t_{k,i}\}$ denotes the $k$-th event time in the $i$-th component.
We parameterize the kernel functions as
$\phi_{ij}(t) = \alpha_{ij} h_{ij}(t)$ for $t > 0$, where $\alpha_{ij} > 0$
is the branching ratio and $h_{ij}(t)$ is a probability density function
on $(0,\infty)$.

In this study, we adopt gamma kernels. Specifically, we assume that $h_{ij}$ follows a gamma density $\Gamma(D_{ij}, \beta_{ij})$:
\[
  h_{ij}(t)
  = \frac{\beta_{ij}^{D_{ij}}}{\Gamma(D_{ij})}
    t^{D_{ij}-1} e^{-\beta_{ij} t}, \qquad t>0,
\]
where $D_{ij} > 0$ is the shape parameter and $\beta_{ij} > 0$ is the rate
parameter. When $D_{ij} = 1$, this specification reduces to the classical
exponential kernel, so the exponential Hawkes model is included as a special case.
We write $\boldsymbol{\mu} = (\mu_1, \mu_2)^{\top}$ for the baseline
intensity vector, and
$\boldsymbol{\alpha} = (\alpha_{ij})_{1 \le i,j \le 2}$,
$\boldsymbol{\beta} = (\beta_{ij})_{1 \le i,j \le 2}$,
and $\boldsymbol{D} = (D_{ij})_{1 \le i,j \le 2}$
for the matrices of branching ratios, rate parameters, and shape parameters,
respectively. We then collect all kernel and baseline parameters into
\[
  \eta = (\boldsymbol{\mu}, \boldsymbol{\alpha}, \boldsymbol{\beta}, \boldsymbol{D})\in (0, \infty)^2\times (0, \infty)^{2\times 2}\times (0, \infty)^{2\times 2}\times (0, \infty)^{2\times 2}.
\]
We assume the spectral radius of $\boldsymbol{\alpha}$ is smaller than 1 to ensure stationarity.
Following \citet{bacry2012non}, the intensity is
\[
  \Lambda
  = (\lambda_1(\eta),\lambda_2(\eta))^{\top}
  = (I_2 - \boldsymbol{\alpha})^{-1}\boldsymbol{\mu}
\]
(Eq.(3) in \cite{bacry2012non}), and the CPCF is
\[
  g(u;\eta,\theta)
  = 1
    + \frac{\nu_{12}(u;\eta)}
           {\lambda_1(\eta)\,\lambda_2(\eta)},
  \qquad u\in\mathbb{R},
\]
where $\nu_{12}$ is the (1, 2) component of the infinitesimal covariance matrix $\nu$ of the bivariate Hawkes process (Eq.(8) in \cite{bacry2012non}):
\begin{equation}\label{eq:hawkes-cov-density}
  \nu_{12}(u)
  = \Bigl(
  \Psi(u)\Sigma + \Sigma \Psi^{\top}(-u) + \widetilde{\Psi}*\Sigma \Psi^{\top} (u)
  \Bigr)_{1, 2}
  \qquad u\in\mathbb{R},
\end{equation}
where $\Sigma=\mathrm{diag}\{\lambda_1,\lambda_2\}$,
$\Psi(u) = (\Psi_{ij}(u))_{1\le i,j\le2}$ is defined by
\[
  \Psi(u)
  = \sum_{m=1}^\infty \Phi^{(*m)}(u),\qquad
  \Phi^{(*m)} = \underbrace{\Phi * \cdots * \Phi}_{m\ \text{times}},
\]
with kernel matrix
\(
  \Phi(u) = \bigl(\phi_{ij}(u)\bigr)_{1\le i,j\le2}
\)
and $*$ denoting the matrix convolution, 
and $\widetilde{\Psi}(u) = \Psi(-u), u\in\mathbb{R}$. 

In addition,
let $\theta\in \mathbb{R}$ denote the lead-lag parameter.
Given a realization of the bivariate Hawkes process $N = (N_1^0, N_2^0)$, we call the shifted process $(N_1, N_2) = (N_1^0,\, N_2^0(\cdot-\theta))$ a \emph{lagged bivariate Hawkes process with gamma kernels (LBHPG)}.
Its distribution is denoted by
\[
  \mathrm{LBHPG}(\eta, \theta).
\]
For the shifted process $\mathrm{LBHPG}(\eta, \theta)$, the intensity is still $\Lambda$, while the cross–covariance density is simply shifted:
\[
  \nu_{12}(u-\theta;\eta),\qquad u\in\mathbb{R}.
\]
Also, we have the cross-pair correlation function (CPCF) of $\mathrm{LBHPG}(\eta,\theta)$ as
\begin{equation}\label{eq:cpcf_LBHPG}
  g(u;\eta,\theta)
  = 1
    + \frac{\nu_{12}(u-\theta;\eta)}
           {\lambda_1(\eta)\,\lambda_2(\eta)},
  \qquad u\in\mathbb{R}.
\end{equation}

In the simulation study, we restrict to the case of common rate parameters, that is, $\beta_{ij}\equiv\beta$ for all $i,j\in\{1,2\}$. In such cases,  
$\mathrm{LBHPG}(\eta,\theta)$ satisfies \ref{ass:cpcf}(ii) with $\alpha = \min\{D_{12}, D_{21}\}$ if $\min\{D_{12},  D_{21}\} < 1$ (diverging gamma kernel(s)) and \ref{ass:cpcf}(i) with  $\alpha=2$ if $D_{11} = D_{21}=D_{12}=D_{22} = 1$ (exponential kernels).
The former can be obtained by the reproducibility of gamma densities and the local behavior of bilateral gamma densities at the origin \cite[Thm.~6.1]{kuchler2008shapes}. For details, see Appendix~\ref{app:lb_hawkes_gamma_a2}. 
The latter follows from the explicit formula for the CPCF for bivariate Hawkes processes with exponential kernels, which can be obtained by \cite[Example 3]{bacry2015hawkes}.
The moment condition \ref{ass:pp}(i) is guaranteed thanks to  Theorem 1 in \cite{leblanc2024exponential}.
We also have a bound on the strong mixing rate \ref{ass:pp}(ii) from Theorem 3.1 in \cite{boly2023mixing}, since the gamma kernels decay geometrically.
Therefore, $\mathrm{LBHPG}(\eta,\theta)$ satisfies the assumptions in Theorem \ref{thm:cpcf} with a (smoothing) kernel satisfying Assumption [K] in such special cases.

\subsubsection{Lagged bivariate Neyman-Scott process with gamma kernels}

First, we recall the construction of a bivariate Neyman-Scott process on $\mathbb{R}$
following Section 5.1 in \citet{shiotani2024statistical}. Let $\mathcal{C}$ be a homogeneous
Poisson (parent) process on $\mathbb{R}$ with intensity $\lambda>0$. For $i=1,2$ and each
$c\in \mathcal{C}$, let $M_i(c)$ be the number of offspring in the component $i$ from the parent $c$.
We assume that $\{M_i(c)\}_{c\in \mathcal{C}}$ are i.i.d.\ copies of a $\mathbb{Z}_{\ge0}$-valued
random variable $M_i$ with finite mean
\(
  \sigma_i = \mathbb{E}[M_i]\in(0,\infty), i=1,2.
\)
For simplicity, we assume that $M_i$ follows a Poisson distribution $\mathrm{Poi}(\sigma_i), i=1, 2$.
Conditional on $M_i(c)$, the temporal offsets $\{d_i(c,m)\}_{m=1}^{M_i(c)}$ of the
offspring are i.i.d.\ with density $f_i$, independent over $i,c,m$ and independent of
$\mathcal{C}$ and $\{M_i(c)\}$. 
Then, let
\[
  N_i^0
  = \sum_{c\in \mathcal{C}}\sum_{m=1}^{M_i(c)}\delta_{c+d_i(c,m)},
  \qquad i=1,2.
\]
We call $N^0=(N_1^0,N_2^0)$ a bivariate Neyman-Scott process.

$N^0$ is stationary with intensities
\[
  \lambda_i(\xi) = \lambda\,\sigma_i,\qquad i=1,2.
\]
Moreover, the cross–intensity
of $(N_1^0,N_2^0)$ is
\[
  \lambda_{12}(u;\xi)
  = \lambda^2\sigma_1\sigma_2
    + \lambda\sigma_1\sigma_2 \int_{\mathbb{R}} f_1(s;\tau_1)f_2(u+s;\tau_2)\,ds,
  \qquad u\in\mathbb{R},
\]
see Eq.(5.2) and the subsequent calculation in \cite{shiotani2024statistical}.
Therefore, the cross–pair correlation function (CPCF) of the bivariate
Neyman-Scott process is
\[
  g(u;\xi)
  = \frac{ \lambda_{12}(u;\xi)}{\lambda_1(\xi)\lambda_2(\xi)}
  = 1 + \frac{1}{\lambda}\int_{\mathbb{R}} f_1(s;\tau_1)f_2(u+s;\tau_2)\,ds,
  \qquad u\in\mathbb{R},
\]
where $\xi$ collects all parameters introduced above.

In this simulation study, we adopt gamma dispersal kernels as in the NBNSP-G model
of Shiotani and Yoshida \cite{shiotani2024statistical}. Specifically, for $i=1,2$ we assume that
\[
  f_i(u;\tau_i)
  = \frac{l_i^{\alpha_i}}{\Gamma(\alpha_i)}u^{\alpha_i-1}e^{-l_i u}1_{(0,\infty)}(u),
  \qquad u\in\mathbb{R},
\]
where $\tau_i=(\alpha_i,l_i)$, $\alpha_i>0$ is the shape parameter and $l_i>0$ is the
rate parameter of the gamma law. We then collect the parameters as 
\[
  \xi = (\lambda,\sigma_1, \sigma_2, \alpha_1, \alpha_2, l_1, l_2)\in (0, \infty)^7.
\]
Let $\theta\in\mathbb{R}$ denote the lead–lag parameter. Given a realization
$N^0=(N_1^0,N_2^0)$ of the bivariate Neyman-Scott process with gamma kernels described above, we define the lagged process by shifting the second component:
\[
  (N_1,N_2) = \bigl(N_1^0,\,N_2^0(\,\cdot\,-\theta)\bigr).
\]
We refer to $N=(N_1,N_2)$ as the \emph{lagged bivariate Neyman-Scott process with gamma kernels (LBNSPG)} and denote its distribution by
\[
  \mathrm{LBNSPG}(\xi,\theta).
\]
The intensities of $\mathrm{LBNSPG}(\xi,\theta)$ are still given by
$\lambda_1(\xi)$ and $\lambda_2(\xi)$, while the cross–intensity is simply shifted:
\[
  \lambda_{12}(u-\theta;\xi),\qquad u\in\mathbb{R}.
\]
Consequently, the CPCF of $\mathrm{LBNSPG}(\xi,\theta)$ is
\begin{gather*}
    g(u;\xi,\theta)
  = \frac{\lambda_{12}(u-\theta;\xi)}{\lambda_1(\xi)\lambda_2(\xi)}
  = 1 + \frac{1}{\lambda}p(u - \theta; \tau_1, \tau_2), \\
  p(u; \tau_1, \tau_2) = \int_{\mathbb{R}}f_1(s;\tau_1)f_2(u+s;\tau_2)\,ds,
  \qquad u\in\mathbb{R}.
\end{gather*}

The function $p(\cdot;\alpha_1,\alpha_2,l_1,l_2)$ is the density of a bilateral gamma distribution.
Küchler and Tappe \cite{kuchler2008shapes} provide a detailed analysis of the shapes of bilateral gamma densities, including unimodality and the local behavior near zero (see Theorem 6.1 therein).
Combining their results with the above representation shows that $g(\cdot;\xi,\theta)$ is strictly unimodal with the peak at $\theta$  whenever the parameters are symmetric, i.e., $\alpha_1=\alpha_2$ and $l_1=l_2$.
We restrict attention to such symmetric cases in our simulation experiments.
Under this restriction, $\mathrm{LBNSPG}(\xi,\theta)$ satisfies \ref{ass:cpcf}(ii) if $\alpha_1 + \alpha_2 < 1$ and \ref{ass:cpcf}(i) if $\alpha_1 + \alpha_2 > 1$, with $\alpha=\min\{\alpha_1+\alpha_2, 3\}$ in both cases.

Note that $\mathrm{LBNSPG}(\xi,0)$ is the special case (setting the noise process to zero and the distribution of $M_i$ to Poisson) of the NBNSP-G model introduced in Section 6.1 in \cite{shiotani2024statistical}.
Therefore, $\mathrm{LBNSPG}(\xi,0)$ satisfies condition [NS] in \cite{shiotani2024statistical}, so that Lemma~10.13 there bounds the $\alpha$–mixing coefficients of $\mathrm{LBNSPG}(\xi,\theta)$ in terms of the tail probabilities of the dispersal kernels:
for all $c_1\ge0$ and $m\ge2r+2$,
\[
  \alpha^N_{c_1,\infty}(m;r):=\sup_{c_2\geq0}\alpha^N_{c_1,c_2}(m;r)
  \le8m\lambda \;
  (m+1+2r)\sum_{i=1}^2\sigma_i\int_{|z|\ge m/2-2r} f_i(z;\tau_i)\,dz.
\]
Since the gamma kernels have geometrically decaying tails, this implies that
$\alpha_{c_1,\infty}(m;r_1)$ decreases faster than any power of $m$, so Assumption~\ref{ass:pp}(ii) holds for $\mathrm{LBNSPG}(\xi,\theta)$.
Moreover, since the Poisson-distributed $M_i$ possesses moments of all orders, Lemma~10.14 in \cite{shiotani2024statistical} yields finiteness of moments of $N_i((0,1])$  of all orders, so that Assumption~\ref{ass:pp}(i) is satisfied.
Therefore, $\mathrm{LBNSPG}(\xi,\theta)$ fulfills Assumptions
\ref{ass:cpcf} and \ref{ass:pp} and thus is in the scope of Theorem~\ref{thm:cpcf} if  $\alpha_1=\alpha_2$,  $l_1=l_2$, and $\alpha_1+\alpha_2\neq 1$.

\subsubsection{Model and parameter specifications in the simulation studies}
We consider two types of bivariate stationary point process models, each with three sets of parameter values.
The models are labeled as
\texttt{hawkes\_gamma\_sym}, 
\texttt{hawkes\_gamma\_asym}, \texttt{hawkes\_exp}, \texttt{ns\_gamma\_1}, \texttt{ns\_gamma\_2}, and \texttt{ns\_gamma\_3}.

\begin{table}
    \centering
    \begin{tabular}{ccccc}\toprule
         Name&Family&  $\alpha$&  $\beta_{\alpha}$ &The convergence rate $T^{-1/\beta_{\alpha}}$\\\midrule
         \texttt{hawkes\_gamma\_sym} &LBHPG&  0.4&  0.4 &$T^{-5/2}$ \\
 \texttt{hawkes\_gamma\_asym} &LBHPG& 0.4& 0.4 &$T^{-5/2}$ \\
         \texttt{hawkes\_exp} &LBHPG&  2&  3 &$T^{-1/2}$ \\
         \texttt{ns\_gamma\_1} &LBNSPG&  0.8&  0.8 &$T^{-5/4}$ \\
         \texttt{ns\_gamma\_2} &LBNSPG&  1.6&  2.2 &$T^{-5/11}$ \\ 
 \texttt{ns\_gamma\_3}& LBNSPG& 3&5 &$T^{-1/5}$ \\ \bottomrule
    \end{tabular}
    \caption{Values of $\alpha$ (the ``sharpness'' parameter of the CPCF $g$ in \ref{ass:cpcf}; smaller values imply sharper functions), $\beta_\alpha=\alpha\vee(2\alpha-1)$, and the minimax lower bound on the convergence rate $T^{-1/\beta_\alpha}$ in Theorem \ref{thm:minimax} for each model.}
    \label{tab:models}
\end{table}

All LBHPG models share
\[
  \boldsymbol{\mu} = (0.2, 0.2)^{\top},\qquad
  \boldsymbol{\alpha} =
  \begin{pmatrix}
    0.1 & 0.1 \\
    0.1 & 0.1
  \end{pmatrix},\qquad
  \boldsymbol{\beta} =
  \begin{pmatrix}
    10 & 10 \\
    10 & 10
  \end{pmatrix}.
\]
We use three specifications for the shape matrix $\boldsymbol{D}$ to cover different degrees of regularity and asymmetry:
\begin{align*}
  \texttt{hawkes\_gamma\_sym}:\quad
  & \boldsymbol{D} =
  \begin{pmatrix}
    0.4 & 0.4 \\
    0.4 & 0.4
  \end{pmatrix}; \\
  \texttt{hawkes\_gamma\_asym}:\quad
  & \boldsymbol{D} =
  \begin{pmatrix}
    0.4 & 0.4 \\
    0.8 & 0.4
  \end{pmatrix}; \\
  \texttt{hawkes\_exp}:\quad
  & \boldsymbol{D} =
  \begin{pmatrix}
    1 & 1 \\
    1 & 1
  \end{pmatrix}.
\end{align*}

For the LBNSPG models, we fix
\[
  \lambda = 0.1,\qquad
  \sigma_1 = \sigma_2 = 4
\]
and use symmetric dispersal shape and rate parameters.
We select three settings to vary the smoothness of the CPCF around the peak:
\begin{align*}
  \texttt{ns\_gamma\_1}:&\quad (\alpha_1, \alpha_2) = (0.4, 0.4), \quad (l_1, l_2) = (10, 10), \\
  \texttt{ns\_gamma\_2}:&\quad (\alpha_1, \alpha_2) = (0.8, 0.8), \quad (l_1, l_2) = (10, 10), \\
   \texttt{ns\_gamma\_3}:&\quad (\alpha_1, \alpha_2) = (2.0, 2.0), \quad (l_1, l_2) = (100, 100). \\
\end{align*}

\subsection{Accuracy}
In this experiment, we compare the accuracy of the lead-lag time estimators using the root mean squared error (RMSE) across the six scenarios.

For each pair of observation time interval length  \(T \in \{1000, 2000, 4000, 8000\}\) and the estimator,  we generate \(5000\) Monte Carlo replicates of sample paths.
In every replicate, the true lead-lag time is drawn as \(\theta^* \sim \mathcal{U}(-0.1, 0.1)\). 
For a given estimator \(\hat\theta\), we report
\[
  \mathrm{RMSE}
  =
  \left(\frac{1}{5000}\sum_{b=1}^{5000}\bigl(\hat\theta^{(b)}-\theta^{*(b)}\bigr)^2\right)^{1/2},
\]
where \(\hat\theta^{(b)}\) and \(\theta^{*(b)}\) denote the estimate and the true lead-lag time in replicate \(b\).
Randomizing \(\theta^*\) across replicates summarizes performance averaged over a range of lead-lag values rather than at a single fixed \(\theta^*\).
The lead-lag time parameter is supposed to be in $(-1, 1)$, i.e., we set $r=1$.
The bandwidth grid is \(\{10^{-1}, 10^{-2}, 10^{-3}, 10^{-4}, 10^{-5}, 10^{-6}\}\) for both the Lepski method and DS estimators.
For the Lepski method, we set $A_T = \log\log T$.
When the contrast function has multiple maximizers on $(-r, r)$, we select the minimum as a deterministic tie-breaking rule.
\begin{figure}[t]
    \centering
    \includegraphics[width=1\linewidth]{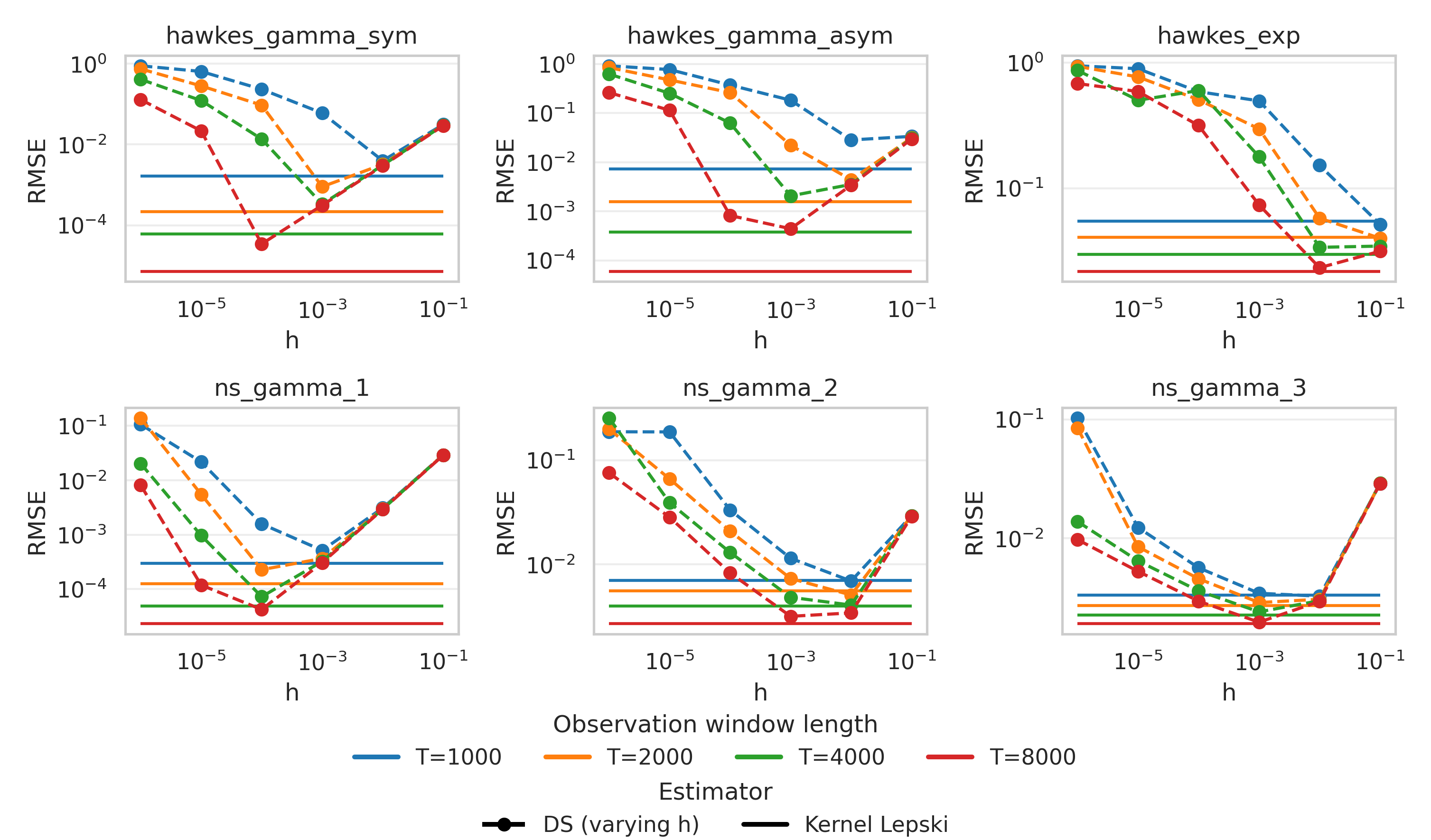}
    \caption{RMSE of the estimators for the lead-lag time across six scenarios.
Each panel uses log--log axes with bandwidth $h$ on the x-axis and RMSE on the y-axis.
Dashed lines with markers trace the DS estimator over the bandwidth grid, while solid horizontal lines show the kernel estimator with the Lepski-selected bandwidth.
Colors indicate the observation-window length $T \in \{1000, 2000, 4000, 8000\}$ and are shared across panels.}
    \label{fig:sim-accuracy}
\end{figure}

Figure \ref{fig:sim-accuracy} illustrates two systematic patterns.
First, the performance of the DS estimator is sensitive to the bucket width $h$. For each $T$, the RMSE curves as a function of $h$ are typically U-shaped: small buckets lead to high variance due to the scarcity of joint activations, whereas larger buckets introduce a discretization bias because the lead-lag time is constrained to a coarse grid. Moreover, the value of $h$ that minimizes the RMSE depends on both the underlying model and the observation window length $T$.
This confirms that, in practice, the DS estimator requires model-specific tuning of $h$.

Second, the kernel estimator with Lepski’s bandwidth selection achieves lower RMSE than the DS estimator for almost all combinations of $T$ and data-generating process.
In Figure \ref{fig:sim-accuracy}, the horizontal solid lines corresponding to the Lepski estimator lie close to (and often below) the best DS RMSE over the grid in each panel. Importantly, this improvement is obtained without any manual tuning of the bandwidth: once the grid $\mathcal{H_T}$ and the slowly diverging constant $A_T=\log\log T$ are fixed, the procedure automatically adapts the smoothing level to the data. This demonstrates the main practical advantage of the proposed approach over the DS method.



\subsection{Convergence rate and dependence on hyperparameters of Lepski's method}\label{sec:sim-lepski}
In this experiment, we investigate the convergence rate of the Lepski estimator, as shown in Theorem~\ref{thm:adaptive}.
We also examine how the estimator behaves as the tuning parameter $A_T$ varies.

As in the previous experiment, for each observation-window length $T \in \{1000, 2000, 4000, 8000\}$ and each estimator, we generate 5000 Monte Carlo replicates of sample paths.
In every replicate, the true lead-lag time is drawn as \(\theta^* \sim \mathcal{U}(-0.1, 0.1)\). 
The lead-lag time parameter is supposed to be in $(-1, 1)$, i.e., we set $r=1$.
The bandwidth grid is \(\{10^{-1}, 10^{-2}, 10^{-3}, 10^{-4}, 10^{-5}, 10^{-6}\}\).

\begin{figure}[t]
  \centering
  \includegraphics[width=\textwidth]{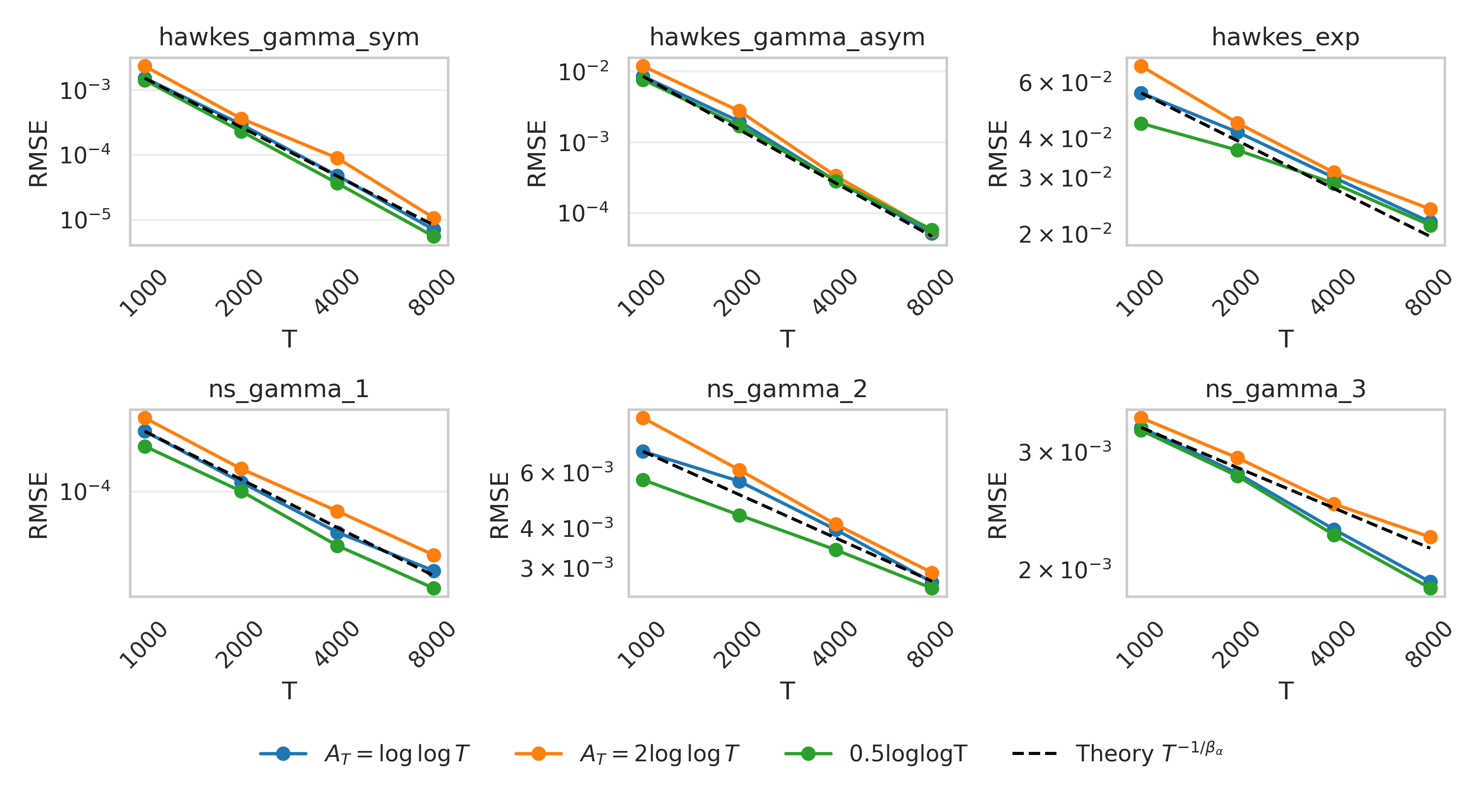}
  \caption{RMSE versus observation window $T$ (log--log scale) across all simulation settings and $A_T$ schedules. Theoretical $T^{-1/\beta_\alpha}$ slopes are shown as dashed lines.}
  \label{fig:sensitivity_lepski_all_scenarios}
\end{figure}

Figure \ref{fig:sensitivity_lepski_all_scenarios} investigates the convergence rate of the Lepski estimator. In each scenario, the RMSE is plotted against \(T\) on a log--log scale together with the theoretical slope $T^{-1/\beta_{\alpha}}$ derived from the minimax lower bound in Theorem \ref{thm:minimax}.
Across all six models, the empirical curves for the Lepski estimator are nearly parallel to the guideline.
Thus, the proposed estimator may attain the optimal convergence rate given by our theory.

The figure also compares three schedules for the tuning parameter $A_T$: $0.5\log\log T$, $\log\log T$, and $2\log\log T$.
Across all six models, the three RMSE curves are close and exhibit similar slopes, indicating that the convergence rate is robust to the choice of $A_T$ within this range.
The differences across schedules are mostly level shifts, so the precise constant in front of $\log\log T$ is not critical for achieving the minimax rate.


\section{Empirical illustration}\label{sec:empirical}




An outstanding feature of the method depicted in \citet{dobrev2017high} is that the DS estimator can effectively capture the fastest speed of information transmission between two geographically separated markets, which is typically close to the speed of light. 
Specifically, they analyzed the lead-lag relationships between the cash and futures markets for the 10-Year U.S. Treasury Note and for the S\&P 500 index. 
The DS estimator stably detected sharp peaks of the cross-market activity measures at 5 milliseconds, which is consistent with the optical propagation time between the futures exchange (in Aurora, Illinois) and the cash market platform (in Secaucus, New Jersey). 
In this section, we investigate whether this finding continues to hold even for geographically closer markets, for which sub-millisecond estimates are required to detect such relationships. 


More precisely, we examine the lead-lag relationships between the quotes of a single stock on two different exchanges: the NASDAQ (located in Carteret, New Jersey) and BATS (located in Secaucus, New Jersey). 
According to \cite[Table 2]{Tiv2020}, the optical propagation time between the NASDAQ and BATS exchanges is approximately 0.1 milliseconds. 
We obtain the timestamps of updates of the best quotes on each exchange in August 2015 from the Daily TAQ dataset, which provides every quote reported to the consolidated tape by all Consolidated Trade Association (CTA) and Unlisted Trading Privileges (UTP) participants. 
According to \cite{BM2019}, the microsecond timestamps were fully implemented in this dataset on August 6, 2015. 
For this reason, we focus on the sample period beginning on August 6, 2015, comprising 18 trading days. 
Our analysis covers the component stocks of the NASDAQ-100 in 2015, totally 108 stocks. 
We restrict attention to transactions occurring between 9:45 and 15:45, discarding the first and last 15 minutes of the trading day in order to avoid non-stationarities commonly observed at the open and close.

In the Daily TAQ data, each quote record contains two timestamps: \texttt{Time} and \texttt{Participant Timestamp}, which refer to the timestamps published by Securities Information Processors (SIPs) and exchange matching engines, respectively. 
Following \cite{BM2019}, we refer to the former as the \emph{SIP timestamp} and the latter as the \emph{participant timestamp}. 
See \cite[Section 2]{BM2019} and \cite[Section 4]{hasbrouck2021price} for the institutional background. 
In this study, we use the participant timestamp because our preliminary analysis suggests that the SIP timestamp is heavily contaminated by \emph{reporting latencies} in the terminology of \cite{BM2019}. 
Specifically, even when multiple market events occur simultaneously and share the same participant timestamp, they often receive different SIP timestamps, causing artificial misalignment across market events \cite{schwenkparticipant,wu2025latency,tivnan2018price}. 
Moreover, reporting latencies fluctuate dynamically due to various latency sources, and their distribution is heavy-tailed (cf.~\cite{holden2023blink,wu2025latency} and \cite[Appendix]{BM2019}), making it difficult to disentangle their effects from genuine lead–lag behavior. 
For these reasons, we rely on participant timestamps. 
Developing a lead–lag estimator that is robust to such timestamp contamination would be an interesting direction for future research. 

For each trading day and for each stock, we compute both our kernel estimator $\hat\theta_{h}$ and the DS estimator $\hat\theta^{DS}_h$. 
For the kernel estimator, we use the triangular kernel and select the bandwidth via the Lepski type method proposed in \cref{sec:lepski}, with $\mcl H_T=\{1~\mu\text{s},10~\mu\text{s},100~\mu\text{s},1000~\mu\text{s}\}$, $A_T=\log\log T$ and $T=21600$ (the number of seconds in 6 trading hours). 
For the DS estimator, we consider two bucket sizes: $h=1~\mu\text{s}$ and $h=100~\mu\text{s}$. 
The former corresponds to the minimum time unit, as suggested in \cite{dobrev2017high}, whereas the latter is motivated by the physical transmission time of approximately 100 $\mu$s between the two exchanges (cf.~\cite[Table 2]{Tiv2020}). 
We set $r=10~\text{ms}$ for the search range for lead-lag parameters. 

\cref{fig:violin} presents violin plots of the lead-lag time estimates for the three estimators. 
The estimates of $\hat\theta_{\hat h}$ cluster around several values, most notably around 95 $\mu$s and 130 $\mu$s. 
These values have clear physical interpretations: According to \cite[Table 2]{Tiv2020}, the transmission time between NASDAQ and BATS is approximately 95 $\mu$s via hybrid laser link and 128 $\mu$s via fiber optic cable. 
This also suggests that, in our sample, NASDAQ generally leads BATS in updating best quotes, which is an intuitive result given that all the stocks analyzed are listed on NASDAQ, and NASDAQ tends to have greater market participation. 

We observe a similar clustering pattern for the DS estimator with $h=1~\mu\text{s}$, although these estimates appear to cluster slightly more around 130 $\mu$s than 95 $\mu$s. 
However, they also exhibit a few negative ``outliers'', an issue not present in the kernel estimator. 
Using a larger bucket size $h=100~\mu\text{s}$ eliminates such outliers, but the coarse discretization inherent in the DS estimator significantly distorts the estimated lead–lag parameters.

To highlight the differences between our estimator and the DS estimator with $h=1~\mu\text{s}$, \cref{fig:scatter} presents a scatter plot of the two sets of estimates, color-coded by $\sqrt{n_1n_2}$, the geometric mean of the sample sizes. 
The two estimators yield similar values in many cases, but their estimates diverge as one or both of $n_1$ and $n_2$ become small. This observation aligns with our theoretical finding that the bucket size in the DS estimator should be increased when the sample size decreases (cf.~the discussion following \cref{thm:ds}).

\begin{figure}[ht]
    \centering
    \includegraphics[scale=0.9]{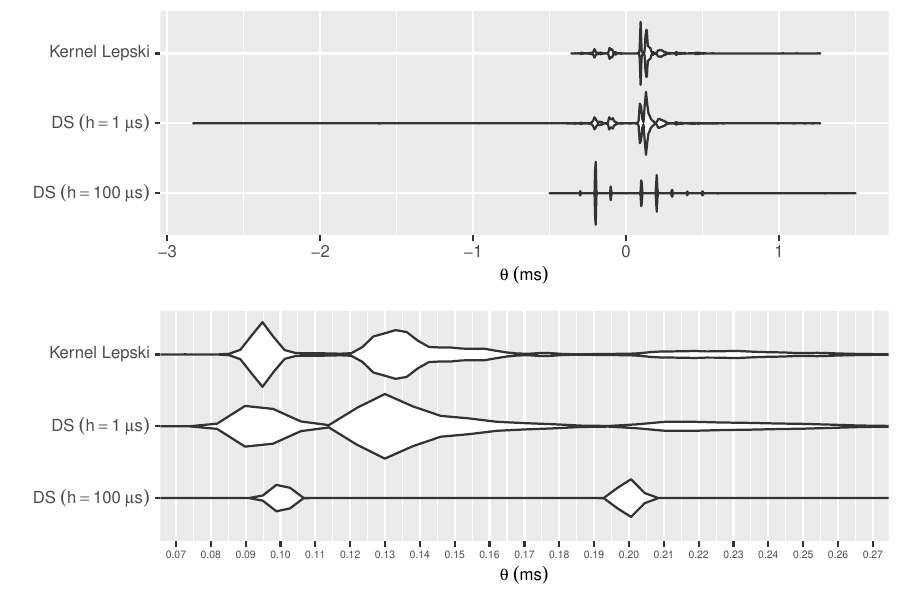}
    \caption{\small
    Violin plots of the daily lead-lag time estimates between quotes on the NASDAQ and BATS exchanges, computed for all the component stocks of the NASDAQ-100 in August 6--31, 2015. 
    The top panel shows the entire plots, while the bottom panel zooms in on the interval from 0.07 ms to 0.27 ms. 
    The smoothing bandwidths for the violin plots are selected by \citet{sheather1991reliable}'s method implemented as the R function \texttt{bw.SJ()}. 
    The horizontal axis is expressed in milliseconds. A positive estimate implies that the NASDAQ leads the BATS and vice versa.
    }
    \label{fig:violin}
\end{figure}

\begin{figure}[ht]
    \centering
    \includegraphics[scale=0.8]{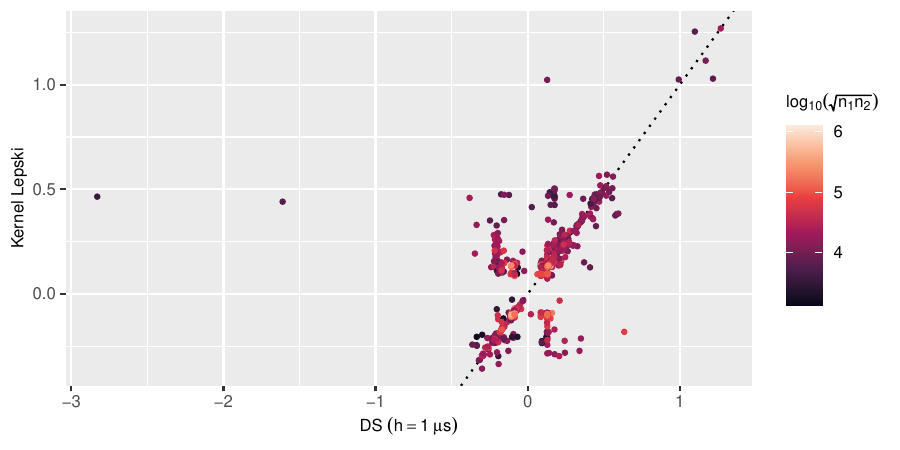}
     \caption{\small
     Scatter plot of the estimates of $\hat\theta^{DS}_h$ with $h=1~\mu\text{s}$ versus $\hat\theta_{\hat h}$, color-coded by the geometric mean of the sample sizes. 
     The values are expressed in milliseconds. 
     The dotted line is the 45-degree line.
    }
    \label{fig:scatter}
\end{figure}




\section{Concluding remarks}\label{sec:conclusion}

In this paper, we have established a theoretical foundation for timestamp-based lead-lag analysis from a point process perspective. 
Within this framework, the method of \citet{dobrev2017high} for analyzing lead-lag relationships can be interpreted as an estimator of the cross-pair correlation function (CPCF) of the bivariate point process generated by two timestamp series. 
Accordingly, the prevailing lead-lag time is naturally defined as the location of the sharpest peak of this CPCF. 
We have proposed estimating this lead-lag time by maximizing a kernel density estimator of the CPCF. 
Theoretically, our estimator nearly attains the optimal convergence rate for estimating the lead-lag time, provided that the bandwidth of the kernel estimator is chosen appropriately. 
To this end, we have introduced a Lepski type bandwidth selection method. 
In practice, our procedure addresses several shortcomings of the Dobrev--Schaumburg estimator that arise from its discrete nature. 
We have demonstrated the superior performance of our estimator through comprehensive simulation studies and an illustrative empirical analysis. 

Finally, we discuss several directions for future research to extend the applicability of our framework.
First, extending the proposed method to non-stationary settings is of practical interest. 
Since financial market activity typically varies over time, the stationarity assumption imposed in this study may be restrictive for certain empirical applications. A promising avenue is to adopt the concept of transition-invariant cross-pair correlation function (see, e.g., \cite{shaw2021globally}), which allows time-varying intensities while preserving a stable lead-lag structure that depends only on the time lag. Developing similar methods for such inhomogeneous settings would allow more accurate estimation of the lead-lag time.
Second, in practical applications, timestamp series may be contaminated by observation noise. 
In such cases, a noise-robust estimator of the lead-lag parameter is required. 
This problem is closely related to deconvolution, which has been studied by \cite{cucala2008intensity} in the context of intensity function estimation for point processes and by \cite{meister2011general} in the context of mode estimation for i.i.d.~data. 
Third, empirical CPCF estimates for high-frequency timestamps often exhibit multiple peaks. 
These naturally arise because information transmission between two markets is bidirectional (as documented in \cite{dobrev2017high}) and because transmission speeds differ across market participants. 
While this paper focuses solely on the sharpest peak, identifying all ``significant'' peaks would also be of interest. 
Similar questions have been studied in the classical literature on mode estimation; see \cite[Section 3]{chacon2020modal} and references therein. 
Fourth, it is also worth investigating how the lead-lag parameter is affected by order types, size, and market conditions such as spread and depth. By treating these attributes as marks, we can analyze the lead-lag relationships between marked point processes. 


\appendix

\section*{Appendix}

\section{Proofs}



Throughout the discussion, we use the following notation and convention. 
For two real numbers $x,y$, we write $x\lesssim y$ if $x\leq Cy$ for some constant $C>0$ depending only on $r,\alpha,\alpha_0,\delta,b,(B_p)_{p\geq1},(B_{p,q})_{p,q\geq1},\eps$ and $\|K\|_\infty$. 
For a real-valued function $f$ defined on $[-r_1,r_1]$, we write $\|f\|_{L^p}=\|f\|_{L^p([-r_1,r_1])}$ for each $p\in[1,\infty]$ for short. 
We set $g_0\equiv1$ whenever \ref{ass:cpcf}\ref{type-I} is assumed. 

\subsection{Preliminaries}

First, we can easily deduce from the definition of the cross-intensity function \eqref{def:cross-intensity} that
\[
\E\sbra{\int_{D\times\mathbb R}\varphi(y-x)N_1(dx)N_2(dy)}
=\leb(D)\lambda_1\lambda_2\int_{\mathbb R}\varphi(u)g(u)du
\]
for any Borel function $\varphi:\mathbb R\to[0,\infty]$ and $D\in\mcl B(\mathbb R)$. 
We refer to this identity as Campbell's formula in the following. 

Next, we prove two auxiliary estimates that play a basic role throughout our discussion. 
The first is a moment inequality for a sum of dependent random variables in terms of the $\alpha$-mixing coefficients \eqref{def:alpha-mixing} and is a simple consequence of \cite[Theorem 2]{doukhan1999new} with a truncation argument.  
\begin{lemma}\label{lemma:mom}
Let $(X_j)_{j=0}^{T-1}$ be a sequence of random variables such that $X_j$ is $\sigma(N\cap(I_j\oplus r_1))$-measurable for all $j$. 
Suppose also that $\max_{0\leq j\leq T-1}\E[X_j^q]<\infty$ for some even integer $q\geq2$. 
Then, there exists a constant $C_q$ depending only on $q$ such that for any $M,\tau>0$,
\bm{
\norm{\sum_{j=0}^{T-1}(X_j-\E[X_j])}_q
\leq C_q\left\{
\bra{T\tau M\max_{0\leq j\leq T-1}\E[|X_j|]+TM^2\sum_{m=\tau}^\infty\alpha^N_{q,q}(m;r_1)}^{1/2}\right.\\
\left.+T^{1/q}M\bra{\sum_{m=0}^\infty(m+1)^{q-2}\alpha^N_{q,q}(m;r_1)}^{1/q}\right\}
+2T\max_{0\leq j\leq T-1}\norm{X_j1_{\{|X_j|>M\}}}_q.
}
\end{lemma}

\begin{proof}
Set $Y_j:=X_j1_{\{|X_j|\leq M\}}$ for $j=0,1,\dots,T-1$. 
By the triangle inequality,
\besn{\label{mom-truncate}
\norm{\sum_{j=0}^{T-1}(X_j-\E[X_j])-\sum_{j=0}^{T-1}(Y_j-\E[Y_j])}_q
&\leq\sum_{j=0}^{T-1}\norm{X_j1_{\{|X_j|>M\}}-\E[X_j1_{\{|X_j|>M\}}]}_q\\
&\leq 2T\max_{0\leq j\leq T-1}\norm{X_j1_{\{|X_j|>M\}}}_q.
}
Meanwhile, for any integers $0\leq j_1\leq \cdots\leq j_p\leq T-1$ such that $j_{k+1}-j_k=m$ for some $0\leq k<p$ and $m\geq0$, we have
\bes{
|\Cov[Y_{j_1}\cdots Y_{j_k},Y_{j_{k+1}}\cdots Y_{j_p}]|
\leq \min\cbra{2M^{p-1}\max_{0\leq j\leq T-1}\E[|X_j|],\,4M^p\alpha^N_{p,p}(m;r_1)},
}
where the second upper bound follows by the covariance inequality under strong mixing (see e.g.~Lemma 3 in \cite[Section 1.2]{Doukhan1994}). 
Hence, we can apply \cite[Theorem 2]{doukhan1999new} to $Y_j$ with $M=M,\gamma=C=1$ and
\[
\theta_m=\min\cbra{M\max_{0\leq j\leq T-1}\E[|X_j|],\,4M^2\alpha^N_{q,q}(m;r_1)}
\]
in their notation. 
Hence, there exists a constant $C_q$ depending only on $q$ such that
\ba{
\E\sbra{\abs{\sum_{j=0}^{T-1}(Y_j-\E[Y_j])}^q}
&\leq C_q\cbra{\bra{T\sum_{m=0}^{T-1}\theta_m}^{q/2}
+M^{q-2}T\sum_{m=0}^{T-1}(m+1)^{q-2}\theta_m}\\
&\leq C_q\left\{
\bra{2T\tau M\max_{0\leq j\leq T-1}\E[|X_j|]+4TM^2\sum_{m=\tau}^\infty\alpha^N_{q,q}(m;r_1)}^{q/2}\right.\\
&\left.\qquad\qquad+4TM^q\sum_{m=0}^\infty(m+1)^{q-2}\alpha^N_{q,q}(m;r_1)\right\}.
}
Combining this with \eqref{mom-truncate} gives the desired result. 
\end{proof}

The second is an estimate for the kernel-smoothed CPCF. 
\begin{lemma}\label{fh-bound}
Suppose that $K$ is bounded and supported on $[-1,1]$. 
For $h\in(0,1]$, define a function $f_h:\mathbb R\to[0,\infty)$ as
\ben{\label{def:fh}
f_h(u)=\int_{\mathbb R}K_h(v-u)g(v)dv=\int_{\mathbb R}K(t)g(u+ht)dt,\qquad u\in\mathbb R.
}
Suppose also that there exist constants $\theta^*\in[-r,r]$, $\tilde\alpha\in(0,1]$, $b>1$ and a function $g_0\in L^{\frac{1}{1-\tilde\alpha}}([-r_1,r_1])$ such that
\ben{\label{g-bound}
g(u)\leq b\cbra{g_0(u)+|u-\theta^*|^{\tilde\alpha-1}}\quad\text{for all }u\in[-r_1,r_1].
}
Then we have
\[
f_h(u)\leq 2^{\tilde\alpha}\|K\|_\infty b\bra{\|g_0\|_{L^{1/(1-\tilde\alpha)}}+\frac{2}{\tilde\alpha}}h^{\tilde\alpha-1}
\quad\text{for all $u\in[-r,r]$}.
\]
\end{lemma}

\begin{proof}
Fix $u\in[-r,r]$. 
\eqref{g-bound} gives 
\ba{
f_h(u)\leq b\bra{\int_{\mathbb R}K_h(v-u)g_0(v)dv+\int_{\mathbb R}K(t)|u-\theta^*+ht|^{\tilde\alpha-1}dt}
=:b(I+II).
}
Since $K_h$ is supported on $[-h,h]\subset[-1,1]$, we can rewrite $I$ as
\[
I=\int_{-r_1}^{r_1}K_h(v-u)g_0(v)dv.
\]
Hence, Young's convolution inequality gives $I\leq\|K_h\|_{L^{1/\tilde\alpha}}\|g_0\|_{L^{1/(1-\tilde\alpha)}}$. 
Since
\[
\|K_h\|_{L^{1/\tilde\alpha}}\leq\bra{2h\|K_h\|_\infty^{1/\tilde\alpha}}^{\tilde\alpha}
\leq2^{\tilde\alpha} h^{\tilde\alpha-1}\|K\|_\infty,
\]
we obtain
\ben{\label{I-bound}
I\leq2^{\tilde\alpha} h^{\tilde\alpha-1}\|K\|_\infty\|g_0\|_{L^{1/(1-\tilde\alpha)}}.
} 
Meanwhile, for any $c\in\mathbb R$, a straightforward computation shows
\ba{
\int_{-1}^1|c+ht|^{\tilde\alpha-1}dt
=\begin{cases}
\displaystyle\frac{\abs{|c+h|^{\tilde\alpha}-|c-h|^{\tilde\alpha}}}{\tilde\alpha h} & \text{if }|c|>h,\\
\displaystyle\frac{|c+h|^{\tilde\alpha}+|c-h|^{\tilde\alpha}}{\tilde\alpha h} & \text{if }|c|\leq h.
\end{cases}
}
Noting that $|x^{\tilde\alpha}-y^{\tilde\alpha}|\leq|x-y|^{\tilde\alpha}$ for any $x,y\geq0$, we obtain 
\ben{\label{power-bound}
\sup_{c\in\mathbb R}\int_{-1}^1|c+ht|^{\tilde\alpha-1}dt\leq\frac{2^{1+\tilde\alpha}}{\tilde\alpha} h^{\tilde\alpha-1}.
}
Therefore, $II\leq\|K\|_\infty\frac{2^{1+\tilde\alpha}}{\tilde\alpha} h^{\tilde\alpha-1}$ since $K$ is supported on $[-1,1]$. 
Combining this with \eqref{I-bound} gives the desired result. 
\end{proof}

\subsection{Proof of Proposition \ref{prop:ds-kernel}}

In the following two lemmas, we deal with the numerator and denominator of $\dsr_h(\ell)$ separately. 
\begin{lemma}\label{lemma:ds-raw}
Assume \ref{ass:pp}. 
Suppose that there exist constants $\theta^*\in[-r,r]$, $\tilde\alpha\in(0,1]$, $b>1$ and a function $g_0\in L^{1/(1-\tilde\alpha)}([-r_1,r_1])$ such that \eqref{g-bound} holds. 
Assume also $h=h_T\asymp T^{-\gamma}$ as $T\to\infty$ for some $\gamma>0$. 
Then, for any $\eps>0$,
\ben{\label{ds-step1}
\E\sbra{\max_{\ell\in\mcl G_h}\abs{\ds_h(\ell)-\frac{T}{h}\E[1_{\{N_1(I_{0}^h)>0,N_2(I_{\ell}^h)>0\}}]}}
=O\bra{1+\sqrt{Th^{\tilde\alpha-\eps}}}.
}
\end{lemma}

\begin{proof}
Set
\[
\dsn_h(\ell):=\sum_{k=0}^{T/h-1}1_{\{N_1(I_{k}^h)>0,\ N_2(I_{k+\ell}^h)>0\}}.
\]
Since
\ba{
|\ds_h(\ell)-\dsn_h(\ell)|
&\leq\sum_{k=0}^{|\ell|-1}1_{\{N_1(I_{k}^h)>0\}}
+\sum_{k=T/h-|\ell|}^{T/h-1}1_{\{N_1(I_{k}^h)>0\}}\\
&\leq N_1((0,|\ell|h])+N_1((T-|\ell|h,T]),
}
we have
\ben{\label{ds-edge}
\E\sbra{\max_{\ell\in\mcl G_h}|\ds_h(\ell)-\dsn_h(\ell)|}
\leq \E[N_1((0,r])]+\E[N_1((T-r,T])]=2\lambda_1r=O(1).
}
Also, $\E[\dsn_h(\ell)]=(T/h)\E[1_{\{N_1(I_{0}^h)>0,N_2(I_{\ell}^h)>0\}}]$ by stationarity. 
Therefore, \eqref{ds-step1} follows once we show
\be{
\E\sbra{\max_{\ell\in\mcl G_h}\abs{\dsn_h(\ell)-\E[\dsn_h(\ell)]}}
=O\bra{\sqrt{Th^{\tilde\alpha-\eps}}}.
}
For any $p>1$, Jensen's inequality gives
\ba{
\E\sbra{\max_{\ell\in\mcl G_h}\abs{\dsn_h(\ell)-\E[\dsn_h(\ell)]}}
&\leq\norm{\max_{\ell\in\mcl G_h}\abs{\dsn_h(\ell)-\E[\dsn_h(\ell)]}}_p\\
&\leq|\mcl G_h|^{1/p}\max_{\ell\in\mcl G_h}\norm{\dsn_h(\ell)-\E[\dsn_h(\ell)]}_p.
}
Let $p$ be an even integer such that $p\geq\frac{4}{\eps(1\wedge\gamma)}$. Then we have $|\mcl G_h|^{1/p}=O(h^{-\eps/4})$. 
Therefore, it suffices to prove
\ben{\label{ds-lp}
\max_{\ell\in\mcl G_h}\norm{\dsn_h(\ell)-\E[\dsn_h(\ell)]}_p=O\bra{\sqrt{Th^{\tilde\alpha-\eps/2}}}.
}
For each $\ell\in\mcl G_h$, set
\[
Y_j(\ell):=\sum_{k=j/h}^{(j+1)/h-1}1_{\{N_1(I_{k}^h)>0,N_2(I_{k+\ell}^h)>0\}},
\]
so that
$\dsn_h(\ell)=\sum_{j=0}^{T-1}Y_j(\ell)$. 
Observe that $Y_j(\ell)$ is $\sigma(N\cap(I_j\oplus r))$-measurable. 
Also, since
\ba{
Y_j(\ell)\leq\sum_{k=j/h}^{(j+1)/h-1}N_1(I_{k}^h)=N_1(I_{j}),
}
\ref{ass:pp}(i) yields 
\ben{\label{y-mom}
\sup_j\|Y_j(\ell)\|_q\leq \|N_1(I_{j})\|_q\lesssim\lambda_1
}
for any $q\geq1$. 
Therefore, applying \cref{lemma:mom} to $(Y_j(\ell))_{j=0}^{T-1}$ with $M=h^{-\eps/4}$ and $\tau=\lfloor h^{-\eps/4}\rfloor$ gives
\bm{
\norm{\dsn_h(\ell)-\E[\dsn_h(\ell)]}_p
\leq C_p\left\{
\bra{Th^{-\eps/2}\E[Y_0(\ell)]+Th^{-\eps/2}\sum_{m=\lfloor h^{-\eps/4}\rfloor}^\infty\alpha^N_{p,p}(m;r_1)}^{1/2}\right.\\
\left.+T^{1/p}h^{-\eps/4}\bra{\sum_{m=0}^\infty(m+1)^{p-2}\alpha^N_{p,p}(m;r_1)}^{1/p}\right\}
+2T\norm{Y_0(\ell)1_{\{Y_0(\ell)>h^{-\eps/4}\}}}_p,
}
where $C_p$ is a constant depending only on $p$ and $\lambda_1$. 
By \ref{ass:pp}(ii),
\ba{
\sum_{m=\lfloor h^{-\eps/4}\rfloor}^\infty\alpha^N_{p,p}(m;r_1)=O(h^{\tilde\alpha})
\quad\text{and}\quad
\sum_{m=0}^\infty(m+1)^{p-2}\alpha^N_{p,p}(m;r_1)=O(1).
}
Also, we have $\|Y_0(\ell)1_{\{Y_0(\ell)>h^{-\eps/4}\}}\|_p\leq h^{p\eps/4}(\E[Y_0(\ell)^{p^2+p}])^{1/p}=O(T^{-1})$ by \eqref{y-mom}. 
Therefore, \eqref{ds-lp} follows once we show
\ben{\label{y-mean}
\max_\ell\E[Y_0(\ell)]=O\bra{h^{\tilde\alpha}}.
}
Observe that
\be{
\E[Y_0(\ell)]
=h^{-1}\E[1_{\{N_1(I_{0}^h)>0,N_2(I_{\ell}^h)>0\}}]
\leq h^{-1}\E[N_1(I_{0}^h)N_2(I_{\ell}^h)]
}
and
\ban{
\E[N_1(I_{0}^h)N_2(I_{\ell}^h)]
&=\lambda_1\lambda_2\int_{\mathbb R^2}1_{I_0^h}(x)1_{I_\ell^h}(x+u)g(u)dxdu
\notag\\
&=\lambda_1\lambda_2h^2\int_{\mathbb R}K^\tri_h(u-\ell h)g(u)du.
\label{ds-kernel}
}
Therefore, \eqref{y-mean} follows from \cref{fh-bound}.
\end{proof}

\begin{lemma}\label{lemma:ds-denom}
Assume \ref{ass:pp}. Then,
\besn{\label{ds-denom1}
&\max_{\ell\in\mcl G_h}\abs{\frac{1}{T}\sum_{k=|\ell|}^{ T/h-1-|\ell|}1_{\{N_1(I_{k}^h)>0\}}-\frac{\pr(N_1(I_{0}^h)>0)}{h}}=O_p\bra{\frac{1}{\sqrt T}},\\
&\max_{\ell\in\mcl G_h}\abs{\frac{1}{T}\sum_{k=|\ell|}^{ T/h-1-|\ell|}1_{\{N_2(I_{k+\ell}^h)>0\}}-\frac{\pr(N_2(I_{0}^h)>0)}{h}}=O_p\bra{\frac{1}{\sqrt T}}
}
and
\ben{\label{ds-denom2}
\max_{\ell\in\mcl G_h}\abs{\frac{1}{T}\min\cbra{\sum_{k=|\ell|}^{T/h-1-|\ell|}1_{\{N_1(I_{k}^h)>0\}},\ \sum_{k=|\ell|}^{T/h-1-|\ell|}1_{\{N_2(I_{k+\ell}^h)>0\}}}-\lambda_1\wedge\lambda_2}\to^p0
}
as $T\to\infty$. 
\end{lemma}

\begin{proof}
\eqref{ds-denom2} follows from \eqref{ds-denom1} and the definition of intensity (cf.~Eq.(3.3.4) of \cite{daley2006introduction}), so it remains to prove \eqref{ds-denom1}.  
We only prove the first equation of \eqref{ds-denom1} because the second can be shown by almost the same argument. 

First, the same argument as in the proof of \eqref{ds-edge} gives 
\ba{
\E\sbra{\max_{\ell\in\mcl G_h}\abs{\sum_{k=|\ell|}^{T/h-1-|\ell|}1_{\{N_1(I_{k}^h)>0\}}-\sum_{k=0}^{T/h-1}1_{\{N_1(I_{k}^h)>0\}}}}=O(1).
}
Hence it suffices to show
\ben{\label{ds-denom-aim}
\frac{1}{T}\sum_{k=0}^{T/h-1}1_{\{N_1(I_{k}^h)>0\}}=\frac{\pr(N_1(I_{0}^h)>0)}{h}+O_p\bra{\frac{1}{\sqrt T}}.
}
We rewrite the left hand side as 
\[
\frac{1}{T}\sum_{k=0}^{T/h-1}1_{\{N_1(I_{k}^h)>0\}}=\frac{1}{T}\sum_{j=0}^{T-1}X_j,
\]
where $X_j:=\sum_{k=j/h}^{(j+1)/h-1}1_{\{N_1(I_{k}^h)>0\}}$. 
Observe that $X_j$ is $\sigma(N\cap I_j)$-measurable. 
Hence,
\ba{
\Var\sbra{\sum_{j=0}^{T-1}X_j}
\leq\sum_{j,m=0}^{T-1}|\Cov[X_j,X_{j+m}]|
\leq8\sum_{j,m=0}^{T-1}\|X_j\|_4\|X_m\|_4\sqrt{\alpha_{1,1}(m;0)}
\lesssim T\|X_0\|_4^2,
}
where the second inequality follows by Theorem 3 in \cite[Section 1.2]{Doukhan1994} and the third by \ref{ass:pp}(ii). 
Since $|X_0|\leq N_1(I_{0})$, we obtain 
\[
\sum_{j=0}^{T-1}X_j=\sum_{j=0}^{T-1}\E[X_j]+O_p(\sqrt T)
=\frac{T}{h}\E[1_{\{N_1(I_{0}^h)>0\}}]+O_p(\sqrt T).
\]
Since $\E[1_{\{N_1(I_{0}^h)>0\}}]=\pr(N_1(I_{0}^h)>0)$, we obtain \eqref{ds-denom-aim}.
\end{proof}

\begin{proof}[Proof of \cref{prop:ds-kernel}]
Since $g$ is bounded, \eqref{g-bound} holds for $b=\sup_{x\in\mathbb R}|g(x)|$, $\tilde\alpha=1$ and $g_0\equiv0$ (with $\theta^*$ arbitrary). 
Also, by assumption, $Th^{\tilde\alpha+\eps}\to\infty$ as $T\to\infty$ for some $\eps>0$. 
Hence, \cref{lemma:ds-raw} gives
\ben{\label{ds-raw-applied}
\max_{\ell\in\mcl G_h}\abs{\frac{\ds_h(\ell)}{Th}-\frac{\E[1_{\{N_1(I_{0}^h)>0,N_2(I_{\ell}^h)>0\}}]}{h^2}}\to^p0.
}
This particularly gives $\max_{\ell\in\mcl G_h}\ds_h(\ell)/Th=O_p(1)$. 
Combining this with \eqref{ds-denom2} gives
\[
\max_{\ell\in\mcl G_h}\abs{\frac{\dsr_h(\ell)}{h}-\frac{1}{\lambda_1\wedge\lambda_2}\frac{\E[1_{\{N_1(I_{0}^h)>0,N_2(I_{\ell}^h)>0\}}]}{h^2}}\to^p0.
\]
Since $\lambda_1\lambda_2/(\lambda_1\wedge\lambda_2)=\lambda_1\vee\lambda_2$, we complete the proof once we show
\ba{
\max_{\ell\in\mcl G_h}\abs{\frac{\E[1_{\{N_1(I_{0}^h)>0,N_2(I_{\ell}^h)>0\}}]}{h^2}-\lambda_1\lambda_2\int_{\mathbb R}K^\tri_h(u-\ell h)g(u)du}
=o(1).
}
Observe that
\ba{
0&\leq\E[N_1(I_{0}^h)N_2(I_{\ell}^h)]-\E[1_{\{N_1(I_{0}^h)>0,N_2(I_{\ell}^h)>0\}}]\\
&=\E[(N_1(I_{0}^h)-1)_+\cdot N_2(I_{\ell}^h)]
+\E[1_{\{N_1(I_{0}^h)>0\}}(N_2(I_{\ell}^h)-1)_+]\\
&\leq\E[N_1(I_0^h)(N_1(I_{0}^h)-1)_+\cdot N_2(I_{\ell}^h)]
+\E[N_1(I_{0}^h)N_2(I_\ell^h)(N_2(I_{\ell}^h)-1)_+]\\
&=\E[N_1(I_0^h)(N_1(I_{0}^h)-1)N_2(I_{\ell}^h)]
+\E[N_1(I_{0}^h)N_2(I_\ell^h)(N_2(I_{\ell}^h)-1)].
}
Since we assume \eqref{ds-extra} for $\varpi=1$, we obtain
\[
\max_{\ell\in\mcl G_h}\abs{\E[N_1(I_{0}^h)N_2(I_{\ell}^h)]-\E[1_{\{N_1(I_{0}^h)>0,N_2(I_{\ell}^h)>0\}}]}=o(h^2).
\]
Combining this with \eqref{ds-kernel} gives the desired result. 
\end{proof}

\subsection{Proof of Theorem \ref{thm:ds}}

The following lemma summarizes identifiability conditions implied by \ref{ass:cpcf}. 
\begin{lemma}\label{lemma:gap}
Assume \ref{kernel} and define the function $f_h$ as in \eqref{def:fh}.
\begin{enumerate}[label=(\alph*)]

\item\label{gap1} Assume \ref{ass:cpcf}\ref{type-I}. Then, for some $\sigma\in\{-1,1\}$, there exist constants $A>1$ and $0<h_0<1$ depending only on $\alpha,b,\delta$ such that
\ben{\label{type-I-gap}
f_{h}(\theta^*+\sigma v)-\sup_{u\in[-r,r]:|u-\theta^*|>Ah}f_h(u)
\geq h^{\alpha-1}
\quad\text{for all $h<h_0$ and $v\in[h,2h]$}.
}

\item\label{gap2} Assume \ref{ass:cpcf}\ref{type-II}. 
Then, for some $\sigma\in\{-1,1\}$, there exist constants $A>1,c>0$ and $0<h_0<1$ depending only on $\alpha,\alpha_0,b,\delta,K$ such that
\ben{\label{type-II-gap}
f_{h}(\theta^*+\sigma v)-\sup_{u\in[-r,r]:|u-\theta^*|>Ah}f_h(u)
\geq ch^{\alpha-1}
\quad\text{for all $h< h_0$ and $v\in[0,h]$.}
} 
    
\end{enumerate}
\end{lemma}

\begin{proof}
\noindent\ref{gap1} By assumption, we have
\ben{\label{type-I-sigma}
\sup_{0<\sigma(u-\theta^*)<\delta}\frac{g(\theta^*)-g(u)}{|u-\theta^*|^{\alpha-1}}\leq b
}
for some $\sigma\in\{-1,1\}$. 
Also, we have
\ben{\label{eq:B2}
g(u)\leq g(\theta^*)-b^{-1}\min\cbra{1,|u-\theta^*|^{\alpha-1}}
}
for all $u\in\mathbb R$. 
Now, for any $A>1$, we have by \ref{kernel} and \eqref{eq:B2}
\ba{
\sup_{u\in[-r,r]:|u-\theta^*|>Ah}f_h(u)
&\leq g(\theta^*)
-b^{-1}\sup_{u\in[-r,r]:|u-\theta^*|>Ah}\int_{-1}^1K(t)\min\cbra{1,|u-\theta^*+ht|^{\alpha-1}}dt\\
&\leq g(\theta^*)-b^{-1}\min\cbra{1,(A-1)^{\alpha-1}h^{\alpha-1}}.
}
Meanwhile, by \eqref{type-I-sigma}, we have for all $h<\delta/3$ and $v\in[h,2h]$,
\ba{
f_{h}(\theta^*+\sigma v)
\geq g(\theta^*)-b\int_{-1}^1K(t)|\sigma v+ht|^{\alpha-1}dt
\geq g(\theta^*)-b(3h)^{\alpha-1}.
}
Combining these estimates gives
\ba{
f_{h}(\theta^*+\sigma v)-\sup_{u\in[-r,r]:|u-\theta^*|>Ah}f_h(u)
\geq b^{-1}\min\cbra{1,(A-1)^{\alpha-1}h^{\alpha-1}}-b(3h)^{\alpha-1}.
}
Therefore, if $A\geq 1+(b+3^{\alpha-1}b^2)^{1/(\alpha-1)}$ and $h<\min\{\delta/3,\,\{b(3^{\alpha-1}b+1)\}^{-1/(\alpha-1)}\}$, we have \eqref{type-I-gap}.
\smallskip

\noindent\ref{gap2} By assumption, we have
\ben{\label{type-II-sigma}
\inf_{0<\sigma(u-\theta^*)<\delta}\frac{g(u)}{|u-\theta^*|^{\alpha-1}}\geq\frac{1}{b}
}
for some $\sigma\in\{-1,1\}$. 
Also, we have $\inf_{t\in[-\delta_0,\delta_0]}K(t)\geq K(0)/2>0$ for some $0<\delta_0<1$ by \ref{kernel}. 
Combining this with \eqref{type-II-sigma}, we have for all $h<\delta/2$ and $v\in[0,h]$
\ba{
f_{h}(\theta^*+\sigma v)\geq \frac{K(0)}{2b}\int_0^{\delta_0} |v+ht|^{\alpha-1}dt 
\geq \frac{K(0)}{2b}h^{\alpha-1}\int_0^{\delta_0} (1+t)^{\alpha-1}dt 
\geq\frac{K(0)\delta_0}{4b}h^{\alpha-1}.
}
Meanwhile, applying \eqref{I-bound} to $\tilde\alpha=\alpha_0$ gives
\[
\int_{\mathbb R}K_h(v-u)g_0(v)dv\leq 2^{\alpha_0}h^{\alpha_0-1}\|K\|_\infty b.
\]
Hence, for any $A>1$, we have by \ref{ass:cpcf} and \ref{kernel}
\ba{
\sup_{u\in[-r,r]:|u-\theta^*|>Ah}f_{h}(u)
&\leq  b\|K\|_\infty\bra{2^{\alpha_0} h^{\alpha_0-1}b
+\sup_{u\in[-r,r]:|u-\theta^*|>Ah}\int_{-1}^1|u-\theta^*+ht|^{\alpha-1}dt}\\
&\leq2b\|K\|_\infty\bra{bh^{\alpha_0-1}+\frac{h^{\alpha-1}}{(A-1)^{1-\alpha}}}.
}
Therefore, if $A$ is sufficiently large such that
\[
\frac{2\|K\|_\infty b}{(A-1)^{1-\alpha}}\leq\frac{K(0)\delta_0}{8b},
\]
we have
\ba{
f_h(\theta^*+\sigma v)-\sup_{u\in[-r,r]:|u-\theta^*|>Ah}f_{h}(u)
\geq\frac{K(0)\delta_0}{8b}h^{\alpha-1}-2b^2\|K\|_\infty h^{\alpha_0-1}.
}
Consequently, \eqref{type-II-gap} holds if 
\[
h_0\leq\min\cbra{\frac{\delta}{2},\ \bra{\frac{K(0)\delta_0}{32\|K\|_\infty b^3}}^{1/(\alpha_0-\alpha)}}
\quad\text{and}\quad 
c\leq \frac{K(0)\delta_0}{16b}.
\]
This completes the proof.
\end{proof}

\begin{proof}[Proof of \cref{thm:ds}]
Observe that \eqref{g-bound} holds for $\tilde\alpha=\alpha\wedge1$ under \ref{ass:cpcf}. 
Hence, \cref{lemma:ds-raw} gives
\[
h^{1-\alpha}\max_{\ell\in\mcl G_h}\abs{\frac{\ds_h(\ell)}{Th}-\frac{E[1_{\{N_1(I_{0}^h)>0,N_2(I_{\ell}^h)>0\}}]}{h^2}}=O\bra{\frac{h^\alpha}{T}+\frac{1}{\sqrt{Th^{\beta_\alpha+\eps}}}}\quad\text{for any }\eps>0,
\]
where we used the identity $2\alpha-\alpha\wedge1=\beta_\alpha$.  
Also, by the proof of \cref{prop:ds-kernel} and \eqref{ds-extra},
\[
\max_{\ell\in\mcl G_h}\frac{\abs{\E[N_1(I_{0}^h)N_2(I_{\ell}^h)]-\E[1_{\{N_1(I_{0}^h)>0,N_2(I_{\ell}^h)>0\}}]}}{h^2}=o(h^{\alpha-1}).
\]
Now, define the function $f_h$ in \eqref{def:fh} with $K=K^\tri$. 
Recall that $h\asymp T^{-\gamma}$ with $0<\gamma<1/\beta_\alpha$. 
Therefore, combining the above two equations with \eqref{ds-kernel}, we obtain
\[
h^{1-\alpha}\max_{\ell\in\mcl G_h}\abs{\frac{\ds_h(\ell)}{Th}-\lambda_1\lambda_2f_h(\ell h)}\to^p0.
\]
Combining this with \cref{fh-bound} particularly gives $\max_{\ell\in\mcl G_h}\ds_h(\ell)/Th=O_p(h^{\tilde\alpha-1})$. 
Hence, by \cref{lemma:ds-denom},
\[
h^{1-\alpha}\max_{\ell\in\mcl G_h}\abs{\frac{\pr(N_1(I_{0}^h)>0)\wedge\pr(N_2(I_{0}^h)>0)}{h^2}\dsr_h(\ell)-\frac{\ds_h(\ell)}{Th}}\to^p0.
\]
Therefore, we have
\ben{\label{ds-convergence}
h^{1-\alpha}\bar\Delta_T\to^p0\qquad\text{as}\quad T\to\infty,
}
where
\[
\bar\Delta_T:=\max_{\ell\in\mcl G_h}\abs{\frac{\pr(N_1(I_{0}^h)>0)\wedge\pr(N_2(I_{0}^h)>0)}{\lambda_1\lambda_2h^2}\dsr_h(\ell)-f_h(\ell h)}.
\]

We turn to the main body of the proof. 
Consider the case $\alpha>1$. 
Then, by \cref{lemma:gap}, for some $\sigma\in\{-1,1\}$, there exist constants $A>1$ and $0<h_0<1$ depending only on $\alpha,b,\delta$ such that \eqref{type-I-gap} holds. 
We can find an integer $\ell^*$ such that $\ell^*h=\theta^*+\sigma v$ for some $v\in[h,2h]$. 
Observe that $\ell^*\in\mcl G_h$ for sufficiently small $h$. 
Then, since $\hat\theta_h^{DS}$ is a maximizer of $\mcl G_h\ni \ell\mapsto\dsr_h(\ell)\in[0,\infty)$, we have
\ba{
\pr\bra{|\hat\theta_h^{DS}-\theta^*|>(A+4)h}
&\leq \pr\bra{|\hat\theta_h^{DS}-\ell^* h|>(A+2)h}\\
&\leq \pr\bra{\dsr_{h}(\ell^*)\leq\max_{\ell\in\mcl G_h:|\ell-\ell^*|>A+2}\dsr_{h}(\ell)}\\
&\leq \pr\bra{f_{h}(\ell^*h)\leq\max_{\ell\in\mcl G_h:|\ell-\ell^*|>A+2}f_h(\ell h)+2\bar\Delta_T}\\
&\leq \pr\bra{f_{h}(\ell^*h)\leq\max_{\ell\in\mcl G_h:|\ell h-\theta^*|>Ah}f_h(\ell h)+2\bar\Delta_T}.
}
Hence, \eqref{type-I-gap} gives $\pr(|\hat\theta_h^{DS}-\theta^*|>(A+4)h)
\leq\pr\bra{2\bar\Delta_T\geq h^{\alpha-1}}$. 
Thus we obtain the desired result by \eqref{ds-convergence}.

Next, consider the case $\alpha<1$. 
Then, by \cref{lemma:gap}, for some $\sigma\in\{-1,1\}$, there exist constants $A>1,c>0$ and $0<h_0<1$ depending only on $\alpha,\alpha_0,b,\delta$ such that \eqref{type-II-gap} holds. 
We can find an integer $\ell^*$ such that $\ell^*h=\theta^*+\sigma v$ for some $v\in[0,h]$. 
Then, a similar argument to the above shows $\pr(|\hat\theta_h^{DS}-\theta^*|>(A+2)h)\to0$ as $T\to\infty$. 
\end{proof}

\subsection{Proof of Theorem \ref{thm:cpcf}}

Set $\tilde\alpha:=\alpha\wedge1$. 
Note that $\tilde\alpha\leq\beta_\alpha$. 
Since the left hand side of \eqref{eq:cpcf} is always bounded by 1, we may assume $Th^{\beta_\alpha+\eps}\geq1$ without loss of generality. 

Let us consider the following statistic:
\[
\tilde g_{h}(u):=\frac{n_1}{T\lambda_1}\frac{n_2}{T\lambda_2}\hat g_h(u)
=\frac{1}{T\lambda_1\lambda_2}\int_{(0,T]^2}K_h(y-x-u)N_1(dx)N_2(dy),\quad u\in\mathbb R.
\]
Observe that $\hat\theta_h$ is also a maximizer of $\tilde g_{h}(u)$ over $u\in[-r,r]$. 
Also, we can rewrite it as $\tilde g_{h}(u)=\sum_{j=0}^{T-1}X^0_j(u)$, where
\[
X^0_j(u)=\frac{1}{T\lambda_1\lambda_2}\int_{I_j\times(0,T]}K_h(y-x-u)N_1(dx)N_2(dy).
\]
We introduce an edge-corrected version of $X_j^0(u)$ as
\[
X_j(u)=\frac{1}{T\lambda_1\lambda_2}\int_{I_j\times\mathbb R}K_h(y-x-u)N_1(dx)N_2(dy).
\]
It is not difficult to see that $(X_j(u))_{j\in\mathbb Z}$ is stationary. 
We first show that replacing $X_j^0(u)$ by $X_j(u)$ does not matter for our argument. 
\begin{lemma}\label{cpcf-edge}
Assume \ref{ass:pp}. 
Assume also that $K$ is bounded and supported on $[-1,1]$. 
Then,
\[
\E\sbra{\sup_{u\in[-r,r]}\abs{\tilde g_{h}(u)-\sum_{j=0}^{T-1}X_j(u)}}
\lesssim\frac{1}{Th}.
\]
\end{lemma}

\begin{proof}
Since $K_h$ is supported on $[-h,h]\subset[-1,1]$, $X_j^0(u)=X_j(u)$ if $r_1< j< T-1-r_1$ for any $u\in[-r,r]$. 
Therefore,
\ba{
\abs{\tilde g_{h}(u)-\sum_{j=0}^{T-1}X_j(u)}
\leq\frac{\|K\|_\infty}{Th\lambda_1\lambda_2}\sum_{0\leq j\leq r_1\text{ or }T-1-r_1\leq j\leq T-1}N_1(I_j)N_2(I_j\oplus r_1).
}
Hence, the desired result follows by \ref{ass:pp}(i).
\end{proof}

Next, set
\[
\Delta_T(u):=\sum_{j=0}^{T-1}\cbra{X_j(u)-\E[X_j(u)]},\qquad u\in[-r,r].
\]
Our next aim is to establish a sufficiently fast convergence of $\sup_{u\in[-r,r]}|\Delta_T(u)|$. 
We first develop pointwise moment bounds.
\begin{lemma}\label{lemma:lp}
Assume \ref{ass:pp} and \ref{ass:cpcf}. 
Assume also that $K$ is bounded and supported on $[-1,1]$. 
If an even integer $p>2$ satisfies $Th^{p\eps/4}\leq 1$, then
\ba{
h^{1-\tilde\alpha}\sup_{u\in[-r,r]}\norm{\Delta_T(u)}_p\lesssim\frac{C_p}{\sqrt{Th^{\tilde\alpha+\eps/2}}},
}
where $C_p$ is a constant depending only on $p$. 
\end{lemma}
%
\begin{proof}
Fix $u\in[-r,r]$. 
Since $K_h$ is supported on $[-h,h]\subset[-1,1]$, we can rewrite $X_j(u)$ as
\ben{\label{eq:xj}
X_j(u)=\frac{1}{T\lambda_1\lambda_2}\int_{I_j\times(I_j\oplus r_1)}K_h(y-x-u)N_1(dx)N_2(dy).
}
Hence, $X_j(u)$ is $\sigma(N\cap (I_j\oplus r_1))$-measurable for every $j$. 
Also, since $\|K_h\|_\infty\leq h^{-1}\|K\|_\infty$, we obtain for any $q>1$
\ben{\label{x-mom}
\max_{0\leq j\leq T-1}\|X_j(u)\|_{q}
=\|X_0(u)\|_{q}\lesssim\frac{1}{Th\lambda_1\lambda_2}\|N_1(I_0)N_2(I_0\oplus r_1)\|_{q}\lesssim\frac{1}{Th}.
}
Therefore, applying \cref{lemma:mom} to $(X_j(u))_{j=0}^{T-1}$ with $M=h^{-\eps/4}/(Th)$ and $\tau=\lfloor h^{-\eps/4}\rfloor$ gives
\bm{
\norm{\Delta_T(u)}_p
\leq C_p\left\{
\bra{\frac{h^{-\eps/2}}{h}\E[|X_0(u)|]+\frac{h^{-\eps/2}}{Th^{2}}\sum_{m=\lfloor h^{-\eps/4}\rfloor}^\infty\alpha^N_{p,p}(m;r_1)}^{1/2}\right.\\
\left.+\frac{T^{1/p}h^{-\eps/4}}{Th}\bra{\sum_{m=0}^\infty(m+1)^{p-2}\alpha^N_{p,p}(m;r_1)}^{1/p}\right\}
+2T\norm{X_0(u)1_{\{X_0(u)>h^{-\eps/4}/(Th)\}}}_p,
}
where $C_p$ is a constant depending only on $p$. 
By \ref{ass:pp}(ii),
\ba{
\sum_{m=\lfloor h^{-\eps/4}\rfloor}^\infty\alpha^N_{p,p}(m;r_1)\lesssim h^{\tilde\alpha}
\quad\text{and}\quad
\sum_{m=0}^\infty(m+1)^{p-2}\alpha^N_{p,p}(m;r_1)\lesssim1.
}
Also, since \eqref{g-bound} holds under \ref{ass:cpcf}, Campbell's formula and \cref{fh-bound} give
\ben{\label{x-campbell}
\E[X_0(u)]=\frac{1}{T}\int_{-1}^{1}K(t)g(u+ht)dt
\lesssim\frac{1}{T}\int_{-1}^1\bra{1+|u-\theta^*+ht|^{\tilde\alpha-1}}dt
\lesssim\frac{1}{Th^{1-\tilde\alpha}}.
}
Moreover, by \eqref{x-mom},
\ba{
\norm{X_0(u)1_{\{X_0(u)>h^{-\eps/4}/(Th)\}}}_p
\leq\bra{\frac{Th}{h^{-\eps/4}}}^p(\E[X_0(u)^{p^2+p}])^{1/p}
\lesssim\frac{h^{p\eps/4}}{Th}.
}
Since $Th^{p\eps/4}\leq 1$, we obtain
\ba{
h^{1-\tilde\alpha}\norm{\Delta_T(u)}_p
\lesssim C_p\bra{
\frac{1}{\sqrt{Th^{\tilde\alpha+\eps/2}}}
+\frac{1}{Th^{\tilde\alpha+\eps/2}}
}.
}
Since we assume $Th^{\tilde\alpha+\eps/2}\geq Th^{\beta_\alpha+\eps}\geq1$, this gives the desired result.  
\end{proof}

To upgrade \cref{lemma:lp} to a moment bound for $\sup_{u\in[-r,r]}|\Delta_T(u)|$, we need the following technical lemma. 
\begin{lemma}\label{lemma:vc}
Let $F$ be a bounded non-decreasing function on $\mathbb R$. 
For any $h,\rho>0$, there exist finite points $-r=u_0\leq u_1\leq\dots\leq u_N=r$ and universal constants $C,d\geq1$ such that $N\leq C\rho^{-d}$ and
\ben{\label{k-covering}
\sup_{u\in[-r,r]}\min_{0\leq a\leq N-1}\int_{-r_1}^{r_1}\cbra{F_h(v-u_{a})-F_h(v-u_{a+1})}g(v)dv\leq2\rho\frac{\|F\|_\infty}{h}\int_{-r_1}^{r_1}g(u)du.
}
\end{lemma}

\begin{proof} 
Without loss of generality, we may assume $\int_{-r_1}^{r_1}g(u)du>0$ since otherwise the asserted claim is trivial. 

By the proof of \cite[Proposition 3.6.12]{gine2016mathematical}, $\mcl G:=\{F_h(\cdot - u):u\in[-r,r]\}$ is a VC subgraph class of functions. 
Also, $\mcl G$ admits an envelope $\|F\|_\infty/h$. 
Therefore, by \cite[Theorem 3.3.9]{gine2016mathematical}, there exist finite points $-r=s_0<s_1<\dots<s_L=r$ and universal constants $C,d\geq1$ such that $L\leq C\rho^{-d}$ and 
\ben{\label{eq:vc}
\sup_{u\in[-r,r]}\min_{0\leq a\leq L}\int_{\mathbb R}\abs{F_h(v-u)-F_h(v-s_a)}Q(dv)\leq\rho\frac{\|F\|_\infty}{h},
}
where $Q$ is a probability measure on $(\mathbb R,\mcl B(\mathbb R))$ defined as
\[
Q(A)=\frac{\int_{A\cap[-r_1,r_1]}g(u)du}{\int_{-r_1}^{r_1}g(u)du},\qquad A\in\mcl B(\mathbb R).
\]
Next, define a function $\psi:\mathbb R\to\mathbb R$ as 
\[
\psi(u)=\int_{-r_1}^{r_1}F_h(v-u)g(v)dv
=\int_{\mathbb R}F(t)g(u+ht)1_{[-r_1,r_1]}(u+ht)dt,\qquad u\in\mathbb R.
\]
Then, \eqref{eq:vc} gives
\[
\sup_{u\in[-r,r]}\min_{0\leq a\leq L}|\psi(u)-\psi(s_a)|\leq\rho\frac{\|F\|_\infty}{h}\int_{-r_1}^{r_1}g(u)du.
\]
Also, since $g$ is non-negative and $F$ is non-decreasing, $\psi$ is non-increasing. 
Moreover, since $g1_{[-r_1,r_1]}\in L^1(\mathbb R)$ and $F$ is bounded, $\psi$ is continuous (see e.g.~\cite[Lemma 1.8.1]{malliavin1995integration}).
Consequently, for every $a=0,\dots,L-1$, there exists a point $t_a\in[s_a,s_{a+1}]$ such that $\psi(t_a)=\{\psi(s_{a})+\psi(s_{a+1})\}/2$ by the intermediate value theorem. 
Observe that $\psi(s_{a})-\psi(t_a)=\psi(t_a)-\psi(s_{a+1})=\{\psi(s_{a})-\psi(s_{a+1})\}/2$. 
Moreover, since $\psi$ is non-increasing,
\[
\min_{0\leq a'\leq L}|\psi(t_a)-\psi(s_{a'})|
=\{\psi(s_a)-\psi(t_{a})\}\wedge\{\psi(t_{a})-\psi(s_{a+1})\}
=\frac{\psi(s_{a})-\psi(s_{a+1})}{2}.
\]
Therefore, we obtain the desired points by setting $u_{2a}=s_a$ and $u_{2a+1}=t_a$ for $a=0,\dots,L-1$ and $u_{2L}=s_{L}$.  
\end{proof}

Combining the previous two lemmas, we can derive the following uniform moment bound for $\Delta_T(u)$: 
\begin{lemma}\label{lemma:ucp}
Assume \ref{ass:pp} and \ref{ass:cpcf}. 
Assume also that $K$ is of bounded variation and supported on $[-1,1]$. 
If $h\leq \min\{T^{-\eta},\,1/2\}$ for some $\eta>0$, then
\ba{
\E\sbra{h^{1-\tilde\alpha}\sup_{u\in[-r,r]}\abs{\Delta_T(u)}}\lesssim\frac{C_\eta}{\sqrt{Th^{\tilde\alpha+\eps}}},
}
where $C_\eta$ depends only on $\eta$. 
\end{lemma}

\begin{proof}
Since $K$ is of bounded variation and supported on $[-1,1]$, there exist two non-decreasing functions $F_1,F_2$ on $\mathbb R$ such that $K=(F_1-F_2)1_{[-1,1]}$ and $|F_1|\vee|F_2|\leq\|K\|_\infty$. 
Therefore, without loss of generality, we may assume that $K$ is of the form $K=F1_{[-1,1]}$ with $F$ a non-decreasing function on $\mathbb R$. 
In the remainder of the proof, we proceed in two steps. 

\paragraph{Step 1.}
For $\rho=h/T^{1/\tilde\alpha}$, let $-r=u_0\leq u_1\leq\dots\leq u_N=r$ and $d\geq1$ be as in \cref{lemma:vc}. 
Inserting the equi-spaced points $-r+k\rho$ $(k=1,\dots,\lfloor 2r/\rho\rfloor)$ into the sequence $(u_a)_{a=0}^N$ if necessary, we may assume $\max_{0\leq a\leq N-1}(u_{a+1}-u_a)\leq\rho$ while \eqref{k-covering} still holds. Note that this operation increases the number of points at most $\lfloor 2r/\rho\rfloor$, so we have $N\lesssim\rho^{-d}$. 

For each $u\in[-r,r]$, set
\[
\wt X_j(u):=\frac{1}{T\lambda_1\lambda_2}\int_{I_j\times\mathbb R}1_{[-h,h]\oplus\rho}(y-x-u)F_h(y-x-u)N_1(dx)N_2(dy),\quad j=0,1,\dots,T-1
\]
and
\[
\wt\Delta_T(u):=\sum_{j=0}^{T-1}\cbra{\wt X_j(u)-\E[\wt X_j(u)]}.
\]
In Step 2, we will show
\ben{\label{mono-arg}
\E\sbra{h^{1-\tilde\alpha}\sup_{u\in[-r,r]}|\Delta_T(u)|}
\lesssim \E\sbra{h^{1-\tilde\alpha}\max_{0\leq a\leq N}|\wt\Delta_T(u_a)|}+\frac{1}{\sqrt{Th^{\tilde\alpha}}}.
}
Given this estimate, we can prove the claim of the lemma as follows. 
Since $N\lesssim \rho^{-d}\leq h^{-d\{1+1/(\tilde\alpha\eta)\}}$, we have $N^{1/p}\lesssim h^{-\eps/4}$ for a sufficiently large even integer $p$ depending only on $\alpha$ and $\eta$. 
Observe that Jensen's inequality gives
\ba{
\E\sbra{h^{1-\tilde\alpha}\max_{0\leq a\leq N}\abs{\wt\Delta_T(u_a)}}
&\leq h^{1-\tilde\alpha}\bra{\E\sbra{\max_{0\leq a\leq N}\abs{\wt\Delta_T(u_a)}^p}}^{1/p}
\leq (N+1)^{1/p}h^{1-\tilde\alpha}\max_{0\leq a\leq N}\norm{\wt\Delta_T(u_a)}_p.
}
Meanwhile, define a function $\wt K:\mathbb R\to\mathbb R$ as $\wt K(u)=2F(2u)1_{[-1,1]\oplus T^{-1/\tilde\alpha}}(2u)$ for $u\in\mathbb R$. Observe that $\wt K$ is supported on $[-1,1]$ and bounded by $2\|K\|_\infty$. 
Moreover, we can rewrite $\wt X_j(u)$ as
\[
\wt X_j(u)=\frac{1}{T\lambda_1\lambda_2}\int_{I_j\times\mathbb R}\tilde K_{2h}(y-x-u)N_1(dx)N_2(dy).
\]
Therefore, we can apply \cref{lemma:lp} to $\wt\Delta_T(u)$ and thus obtain 
\ba{
\E\sbra{h^{1-\tilde\alpha}\max_{0\leq a\leq N}\abs{\wt\Delta_T(u_a)}}
\lesssim\frac{(N+1)^{1/p}}{\sqrt{Th^{\tilde\alpha+\eps/2}}}
\lesssim\frac{1}{\sqrt{Th^{\tilde\alpha+\eps}}}.
}
Combining this with \eqref{mono-arg} gives the claim of the lemma. 

\paragraph{Step 2.}
It remains to prove \eqref{mono-arg}. 
Fix $u\in[-r,r]$. We can find an index $0\leq a<N$ such that $u_a\leq u\leq u_{a+1}$. 
Since $u-u_a\leq u_{a+1}-u_a\leq\rho$ and $F$ is non-decreasing, we have
\ba{
X_j(u)&=\frac{1}{T\lambda_1\lambda_2}\int_{I_j\times\mathbb R}1_{[-h,h]}(y-x-u)F_h(y-x-u)N_1(dx)N_2(dy)\leq\wt X_j(u_a)
}
for all $j$. 
Campbell's formula gives
\ba{
\E[\wt X_j(u_a)]
&=\frac{1}{T}\int_{\mathbb R}1_{[-h,h]\oplus\rho}(v-u_a)K_h(v-u_a)g(v)dv\\
&=\frac{1}{T}\int_{-r_1}^{r_1}1_{[-h,h]\oplus\rho}(v-u_a)K_h(v-u_a)g(v)dv,
}
where the second equality follows from $u_a\in[-r,r]$ and $h+\rho\leq2h\leq1$. 
Observe that $1_{[-h,h]\oplus\rho}(v-u_a)=1_{J_{h,u}}(v-u)$ with $J_{h,u}:=([-h,h]\oplus\rho)+(u_a-u)$. 
Note that $J_{h,u}\supset[-h,h]$ because $u-u_a\leq\rho$. 
Hence, 
\ba{
\abs{\E[\wt X_j(u_a)]-\frac{1}{T}\int_{-r_1}^{r_1}1_{[-h,h]}(v-u)K_h(v-u_a)g(v)dv}
&\leq\frac{\|K\|_\infty}{Th}\int_{-r_1}^{r_1}1_{J_{h,u}\setminus[-h,h]}(v-u)g(v)dv.
}
Noting that $\|g_0\|_{L^{1/(1-\tilde\alpha/2)}}\lesssim1$ by Jensen's inequality (if $\alpha<1$) and $\int_{-r_1}^{r_1}|u-\theta^*|^{\frac{\alpha-1}{1-\tilde\alpha/2}}du\lesssim1$ because $\frac{\alpha-1}{1-\tilde\alpha/2}>-1$, we have $\|g\|_{L^{1/(1-\tilde\alpha/2)}}\lesssim1$ by \ref{ass:cpcf}. 
Then, since $\leb(J_{h,u}\setminus[-h,h])\lesssim\rho$, Young's inequality gives
\ba{
\abs{\E[\wt X_j(u_a)]-\frac{1}{T}\int_{-r_1}^{r_1}1_{[-h,h]}(v-u)K_h(v-u_a)g(v)dv}
\lesssim \frac{\rho^{\tilde\alpha/2}}{Th}.
}
Meanwhile, by Campbell's formula again,
\ba{
\E[X_j(u)]
&=\frac{1}{T}\int_{\mathbb R}1_{[-h,h]}(v-u)K_h(v-u)g(v)dv.
}
Therefore, \eqref{k-covering} gives
\ba{
\abs{\E[X_j(u)]-\frac{1}{T}\int_{-r_1}^{r_1}1_{[-h,h]}(v-u)K_h(v-u_a)g(v)dv}
\lesssim \frac{\rho}{Th}.
}
Consequently, 
\ba{
X_j(u)-\E[X_j(u)]\leq\wt X_j(u_a)-\E[\wt X_j(u_a)]+C_0\frac{\rho^{\tilde\alpha/2}}{Th},
}
where $C_0>0$ is a constant depending only on $r,\alpha,\alpha_0,\delta,b,(B_p)_{p\geq1},(B_{p,q})_{p,q\geq1},\eps$ and $\|K\|_\infty$.
Similarly, we also have
\ba{
X_j(u)-\E[X_j(u)]\geq\wt X_j(u_{a+1})-\E[\wt X_j(u_{a+1})]-C_0\frac{\rho^{\tilde\alpha/2}}{Th}.
}
Thus, we conclude
\[
\sup_{u\in[-r,r]}|\Delta_T(u)|
\leq \max_{0\leq a\leq N}|\wt\Delta_T(u_a)|+C_0\frac{\rho^{\tilde\alpha/2}}{h}.
\]
Therefore, \eqref{mono-arg} follows from the definition of $\rho$. 
\end{proof}

\begin{proof}[Proof of \cref{thm:cpcf}]
Define the function $f_h$ by \eqref{def:fh} and 
set $\bar\Delta_T:=\sup_{u\in[-r,r]}\abs{\tilde g_h(u)-f_h(u)}$. 
Since $\E[X_j(u)]=f_h(u)$ for all $j$, \cref{cpcf-edge} gives
\ben{\label{edge-applied}
\E\sbra{\bar\Delta_T}\lesssim\E\sbra{\sup_{u\in[-r,r]}|\Delta_T(u)|}+\frac{1}{Th}.
}

Now, consider the case $\alpha>1$. 
Then, by \cref{lemma:gap}, for some $\sigma\in\{-1,1\}$, there exist constants $A>1$ and $0<h_0<1$ depending only on $\alpha,b,\delta$ such that \eqref{type-I-gap} holds. 
Since $\hat\theta_h$ is a maximizer of $[-r,r]\ni u\mapsto\tilde g_{h}(u)\in[0,\infty)$, we have
\ba{
\pr\bra{|\hat\theta_h-\theta^*|>Ah}
&\leq \pr\bra{\tilde g_{h}(\theta^*+\sigma h)\leq\sup_{u\in[-r,r]:|u-\theta^*|>Ah}\tilde g_{h}(u)}\\
&\leq \pr\bra{f_{h}(\theta^*+\sigma h)\leq\sup_{u\in[-r,r]:|u-\theta^*|>Ah}f_{h}(u)+2\bar\Delta_T}.
}
Therefore, \eqref{type-I-gap} gives $\pr\bra{|\hat\theta_h-\theta^*|>Ah}
\leq\pr\bra{2\bar\Delta_T\geq h^{\alpha-1}}$. 
Hence, the desired result follows by Markov's inequality, \eqref{edge-applied} and \cref{lemma:ucp}. 

Next, consider the case $\alpha<1$. 
Then, by \cref{lemma:gap}, for some $\sigma\in\{-1,1\}$, there exist constants $A>1,c>0$ and $0<h_0<1$ depending only on $\alpha,\alpha_0,b,\delta,K$ such that \eqref{type-II-gap} holds. 
Since $\hat\theta_h$ is a maximizer of $[-r,r]\ni u\mapsto\tilde g_{h}(u)\in[0,\infty)$, 
\ba{
\pr\bra{|\hat\theta_h-\theta^*|>Ah}
&\leq \pr\bra{\tilde g_{h}(\theta^*)\leq\sup_{u\in[-r,r]:|u-\theta^*|>Ah}\tilde g_{h}(u)}\\
&\leq \pr\bra{f_{h}(\theta^*)\leq\sup_{u\in[-r,r]:|u-\theta^*|>Ah}f_{h}(u)+2\bar\Delta_T}\\
&\leq\pr\bra{2\bar\Delta_T\geq ch^{\alpha-1}},
}
where the last line follows from \eqref{type-II-gap}. 
Therefore, the desired result follows by Markov's inequality, \eqref{edge-applied} and \cref{lemma:ucp} again. 
\end{proof}

\subsection{Proof of Theorem \ref{thm:adaptive}}

Below we assume $T$ is sufficiently large $T$ such that $j_{\min}\leq\log_a(T^\gamma_{\max})$. 
For $0<\gamma\leq\gamma_{\max}$, we write $h^*(\gamma)$ for the largest element $h\in\mcl H_T$ such that $h\leq T^{-\gamma}$. 
Note that $h^*(\gamma)$ is well-defined because $\gamma\leq\gamma_{\max}$. 
Also, $ah^*(\gamma)>T^{-\gamma}$ for sufficiently large $T$ by construction, so $h^*(\gamma)\asymp T^{-\gamma}$ as $T\to\infty$. 
\begin{lemma}\label{lem:adaptive}
Under the assumptions of \cref{thm:adaptive}, for any $0<\gamma<1/\beta_\alpha$, $\pr(\hat h>T^{-\gamma})\to0$ as $T\to\infty$.
\end{lemma}

\begin{proof}
Write $h^*=h^*(\gamma)$ for short. 
Observe that
\ba{
\pr(\hat h>T^{-\gamma})
&\leq \pr(\hat h>h^*)
\leq \pr(\bar d(\mcl M_{h^*},\mcl M_{h'})> A_Th'\text{ for some }h'\in\mcl H_T\text{ with }h'\geq h^*)\\
&\leq\sum_{h'\in\mcl H_T:h'\geq h^*}\pr(\bar d(\mcl M_{h^*},\mcl M_{h'})> A_Th').
}
Since $\bar d(\mcl M_{h^*},\mcl M_{h'})\leq\bar d(\mcl M_{h^*},\{\theta^*\})+\bar d(\mcl M_{h'},\{\theta^*\})$, we have
\ba{
\pr(\hat h>T^{-\gamma})
&\leq\sum_{h'\in\mcl H_T:h'\geq h^*}\cbra{\pr(\bar d(\mcl M_{h^*},\{\theta^*\})> A_Th'/2)+\pr(\bar d(\mcl M_{h'},\{\theta^*\})> A_Th'/2)}\\
&\leq2|\mcl H_T|\max_{h\in\mcl H_T:h\geq h^*}\pr(\bar d(\mcl M_{h},\{\theta^*\})> A_Th/2).
}
For every $h\in\mcl H_T$, we can find a random variable $\tilde\theta_{h}\in\mcl M_{h}$ such that $|\tilde\theta_{h}-\theta^*|> A_Th/2$ on the event $\bar d(\mcl M_{h},\{\theta^*\})> A_Th/2$. Hence,
\ba{
\pr(\hat h>T^{-\gamma})
&\leq2|\mcl H_T|\max_{h\in\mcl H_T:h\geq h^*}\pr(|\tilde\theta_{h}-\theta^*|> A_Th/2).
}
Thus, for any $\eps>0$, \cref{thm:cpcf} gives
\ba{
\pr(\hat h_T>T^{-\gamma})
=O\bra{|\mcl H_T|\max_{h\in\mcl H_T:h\geq h^*}\frac{1}{\sqrt{Th^{\beta_\alpha+\eps}}}}
=O\bra{\frac{|\mcl H_T|}{\sqrt{T(h^*)^{\beta_\alpha+\eps}}}}.
}
Since $\gamma<1/\beta_\alpha$ and $|\mcl H_T|=O(\log T)$, we obtain the desired result by taking $\eps$ so that $\gamma(\beta_\alpha+\eps)<1$. 
\end{proof}

\begin{proof}[Proof of \cref{thm:adaptive}]
Let $\gamma'$ be a constant such that $\gamma<\gamma'<1/\beta_\alpha$. Then, $A_TT^{-\gamma'}=o(T^{-\gamma})$ by assumption. 
Hence it is enough to prove $\pr(|\hat\theta_{\hat h}-\theta^*|>2A_TT^{-\gamma'})\to0$. 
Moreover, thanks to \cref{lem:adaptive}, it suffices to show
\[
\pr(|\hat\theta_{\hat h}-\theta^*|>2A_TT^{-\gamma'},\hat h\leq T^{-\gamma'})\to0.
\]
On the event $\hat h\leq T^{-\gamma'}$, we have $\hat h\leq h^*(\gamma')$, so 
\ba{
|\hat\theta_{\hat h}-\theta^*|\leq|\hat\theta_{\hat h}-\hat\theta_{h^*(\gamma')}|+|\hat\theta_{h^*(\gamma')}-\theta^*|
\leq A_Th^*(\gamma')+|\hat\theta_{h^*(\gamma')}-\theta^*|
\leq A_TT^{-\gamma'}+|\hat\theta_{h^*(\gamma')}-\theta^*|,
}
where the second inequality follows by the definition of $\hat h$. 
Therefore,
\[
\pr(|\hat\theta_{\hat h}-\theta^*|>2A_TT^{-\gamma'},\hat h\leq T^{-\gamma'})
\leq\pr(|\hat\theta_{h^*(\gamma')}-\theta^*|>A_TT^{-\gamma'}).
\]
Since $\pr(|\hat\theta_{h^*(\gamma')}-\theta^*|>A_TT^{-\gamma'})\to0$ by \cref{thm:cpcf}, we obtain the desired result.
\end{proof}

\subsection{Proof of Theorem \ref{thm:minimax}}

The proof of \cref{thm:minimax} relies on Theorem 2.2 in \cite{tsybakov2008nonparametric}, which requires the notion of the Hellinger distance. 
Recall that the Hellinger distance between two probability measures $P$ and $Q$ defined on a common measurable space $(\mcl X,\mcl A)$ is defined as
\[
H(P,Q):=\sqrt{\int_{\mcl X}\bra{\sqrt{\frac{dP}{d\nu}}-\sqrt{\frac{dQ}{d\nu}}}^2d\nu},
\]
where $\nu$ is any $\sigma$-finite measure on $(\mcl X,\mcl A)$ dominating both $P$ and $Q$. 
By Lemmas 2.9 and 2.10(1) in \cite{strasser1985mathematical}, 
\ben{\label{eq:affinity}
H^2(P,Q)=2\bra{1-\int_{\mcl X}\sqrt{\frac{dQ}{dP}}dP},
}
where $dQ/dP:=dQ^a/dP$ with $Q^a$ the absolutely continuous part of $Q$ with respect to $P$. 
Note that \citet{strasser1985mathematical} defines the Hellinger distance as $H(P,Q)/\sqrt 2$ in our notation. 

\begin{proof}[Proof of \cref{thm:minimax}]
Define a point process $\tilde N$ on $\mathbb R^2$ as $\tilde N:=\sum_{i=1}^\infty\delta_{(t_i,t_i+\gamma_i)}$. 
Since $N_1(\cdot)=\tilde N(\cdot\times\mathbb R)$ and $N_2(\cdot)=\tilde N(\mathbb R\times\cdot)$, we have $\sigma(N\cap[0,T])\subset\sigma(\tilde N\cap[0,T]^2)$. 
Also, with $D_T:=[0,T]\times[-1,T+1]$, we evidently have $\sigma(\tilde N\cap[0,T]^2)\subset\sigma(\tilde N\cap D_T)$. 
Therefore, it suffices to show that there exists a constant $b>0$ such that
\ben{\label{minimax-informative}
\liminf_{T\to\infty}\inf_{\tilde\theta_T}\sup_{|\theta|\leq 2\rho_T}\sup_{g\in\mcl G(\theta,\alpha,1/2,b)}\pr_g\bra{|\tilde\theta_T-\theta|\geq \rho_T}>0,
}
where the infimum is taken over all estimators based on $\tilde N\cap D_T$. 
For every probability density $g$ on $\mathbb R$, we denote by $P_{T,g}$ the law of $\tilde N\cap D_T$ induced on $(\mcl N^\#_{D_T},\mcl B(\mcl N^\#_{D_T}))$ under $\pr_g$, where $\mcl N^\#_{D_T}$ denotes the space of all counting measures on $D_T$ equipped with the $w^\#$-topology; see \cite[Appendix A2.6]{daley2006introduction} and \cite[Definition 9.1.II]{daley2007introduction} for details. 
According to Eq.(2.9) and Theorem 2.2 in \cite{tsybakov2008nonparametric}, we obtain \eqref{minimax-informative} once we find $g_T\in\mcl G(2\rho_T,\alpha,1/2,b)$ and $g_0\in\mcl G(0,\alpha,1/2,b)$ such that
\ben{\label{minimax-hellinger}
\limsup_{T\to\infty}H^2(P_{T,g_{T}},P_{T,g_0})<2.
}

Let us compute $H^2(P_{T,g_{T}},P_{T,g_0})$. 
Observe that $\tilde N$ can be viewed as a cluster process on $\mathbb R$ with centre process $N_1$ and component processes $\{\delta_{(t_i,t_i+\gamma_i)}:i\in\mathbb N\}$. 
Hence, by Proposition 6.3.III in \cite{daley2006introduction}, the probability generating functional (p.g.fl) of $\tilde N$ under $\pr_g$ for $g\in\{g_0,g_T\}$ is given by
\ba{
G_g(\varphi)&=\exp\bra{-\int_{\mathbb R}\bra{1-\int_{\mathbb R}\varphi(x,y)g(y-x)dy}dx}\\
&=\exp\bra{\int_{\mathbb R^2}\bra{\varphi(x,y)-1}g(y-x)dxdy}
}
for every measurable function $\varphi:\mathbb R^2\to(0,1]$ such that the support of $1-\varphi$ is bounded. 
Since $g$ is supported on $[-1,1]$, 
\ba{
G_{g}(1-1_{D_T}+\varphi1_{D_T})
&=e^{-T}\exp\bra{\int_{[0,T]\times\mathbb R}\varphi(x,y)g(y-x)dxdy}.
}
Therefore, in view of Eq.(5.5.14) in \cite{daley2006introduction}, the local Janossy densities of $\tilde N$ on $D_T$ under $\pr_g$ are given by
\[
j_{n,g}((x_1,y_1),\dots,(x_n,y_n)\mid D_T)=e^{-T}\prod_{i=1}^ng(y_i-x_i)1_{[0,T]}(x_i)\quad(n=1,2,\dots).
\]
This gives
\[
\frac{dP_{T,g_T}}{dP_{T,g_0}}(\tilde N)=\prod_{i:t_i\in[0,T]}\frac{g_T(\gamma_i)}{g_0(\gamma_i)}\quad\text{$\pr_{g_0}$-a.s.}
\]
Here, recall that $dP_{T,g_T}/dP_{T,g_0}:=dP_{T,g_T}^a/dP_{T,g_0}$ with $P_{T,g_T}^a$ the absolutely continuous part of $P_{T,g_T}$ with respect to $P_{T,g_0}$. 
Thus, by \eqref{eq:affinity},
\ba{
H^2(P_{T,g_T},P_{T,g_0})
=2\bra{1-\E_{g_0}\sbra{\prod_{i:t_i\in[0,T]}\sqrt{\frac{g_T(\gamma_i)}{g_0(\gamma_i)}}}}
=:2(1-a_T),
}
where $\E_{g_0}$ denotes expectation under $\pr_{g_0}$. 
Therefore, \eqref{minimax-hellinger} follows once we show
$
\liminf_{T\to\infty}a_T>0.
$
Recall that under $\pr_{g_0}$, $(\gamma_i)_{i=1}^\infty$ is i.i.d.~with common density $g_0$ and independent of $N_1$. 
Hence, 
\ba{
a_T&=\E_{g_0}\sbra{\bra{\int\sqrt{g_T(x)g_0(x)}dx}^{N_1([0,T])}}.
}
Since $N_1([0,T])$ follows the Poisson distribution with intensity $T$ under $\pr_{g_{0}}$,
\ba{
a_T&=\exp\bra{T\bra{\int\sqrt{g_T(x)g_0(x)}dx-1}}
=\exp\bra{-\frac{T}{2}\int\bra{\sqrt{g_T(x)}-\sqrt{g_0(x)}}^2dx}.
}
Therefore, we complete the proof once we show that there exists a constant $b>0$ such that
\ben{\label{minimax-finally}
\int\bra{\sqrt{g_T(x)}-\sqrt{g_0(x)}}^2dx=O(T^{-1})
}
for some $g_T\in\mcl G(2\rho_T,\alpha,1/2,b)$ and $g_0\in\mcl G(0,\alpha,1/2,b)$. 

\paragraph{Case 1: $0<\alpha<1$.} 
For every $\theta\in[0,1]$, define a function $f_\theta:\mathbb R\to[0,\infty)$ as 
\[
f_\theta(x)=\alpha|x-\theta|^{\alpha-1}1_{[-1,0)}(x-\theta)\quad(x\in\mathbb R).
\]
By construction, we evidently have $f_\theta\in\mcl G(\theta,\alpha,1/2,b)$ for some constant $b>0$ depending only on $\alpha$. 
Moreover, since $f_0$ satisfies Eq.(1.9) of \cite[Chapter VI]{ibragimov2013statistical} in a neighborhood of $z=0$ with $\alpha=\alpha-1$, $p\equiv1$ and $q\equiv0$ in their notation, $f_0$ has one singularity of order $\alpha-1$ located at 0 in the sense of Definition 1.1 of \cite[Chapter VI]{ibragimov2013statistical}. 
Therefore, Theorem 1.1 in \cite[Chapter VI]{ibragimov2013statistical} gives \eqref{minimax-finally} for $g_T=f_{2\rho_T}$ and $g_0=f_0$. 

\paragraph{Case 2: $\alpha>1$.} 
We employ a minor variant of the construction used in the proof of \cite[Theorem 2]{arias2022estimation}. 
For every $\theta\in[0,1/2)$, define functions $\psi_\theta:\mathbb R\to\mathbb R$ and $f_\theta:\mathbb R\to\mathbb R$ as
\ba{
\psi_{\theta}(x)=(|x|^{\alpha-1}-(2\theta)^{\alpha-1})1_{(-2\theta,0]}(x)
+\bra{|x|^{\alpha-1}+(2\theta)^{\alpha-1}-2^\alpha|x-\theta|^{\alpha-1}}1_{(0,2\theta)}(x),\quad x\in\mathbb R
}
and
\ba{
f_\theta(x)=\frac{\alpha}{2(\alpha-1)}\bra{1-|x|^{\alpha-1}+\psi_\theta(x)}1_{[-1,1]}(x),\quad x\in\mathbb R.
}
Observe that $f_\theta\geq0$ and
\ba{
\int_{-\infty}^\infty f_{\theta}(x)dx
=\frac{\alpha}{2\alpha-1}\bra{\int_{-1}^1(1-|x|^{\alpha-1})dx+\int_{-2\theta}^{2\theta}|x|^{\alpha-1}dx-2^\alpha\int_0^{2\theta}|x-\theta|^{\alpha-1}dx}
=1.
}
Hence $f_\theta$ is a probability density function on $\mathbb R$. 
Moreover, we have for all $x\in[-1,1]$
\ben{\label{typeI-minimax-key}
2^{2-\alpha}|x-\theta|^{\alpha-1}
\leq\frac{2(\alpha-1)}{\alpha}\bra{f_\theta(\theta)-f_\theta(x)}
\leq2^\alpha|x-\theta|^{\alpha-1}.
}
In fact, a straightforward computation shows
\ba{
\frac{2(\alpha-1)}{\alpha}\bra{f_\theta(\theta)-f_\theta(x)}
&=\begin{cases}
    2(2\theta)^{\alpha-1} & \text{if }-2\theta<x\leq0,\\
    2^\alpha|x-\theta|^{\alpha-1} & \text{if }0<x<2\theta,\\
    |x|^{\alpha-1}+(2\theta)^{\alpha-1} & \text{otherwise}.
\end{cases}
}
Hence \eqref{typeI-minimax-key} is evident if $0<x<2\theta$. 
If $-2\theta<x\leq0$, then $\theta\leq|x-\theta|\leq3\theta$, so 
$
2(2\theta)^{\alpha-1}\leq2^\alpha|x-\theta|^{\alpha-1}
$
and
$
2(2\theta)^{\alpha-1}\geq2(2|x-\theta|/3)^{\alpha-1}\geq2^{2-\alpha}|x-\theta|^{\alpha-1}
$. Hence \eqref{typeI-minimax-key} holds. 
Also, Jensen's inequality gives
$
|x-\theta|^{\alpha-1}\leq2^{\alpha-2}(|x|^{\alpha-1}+\theta^{\alpha-1})\leq2^{\alpha-2}(|x|^{\alpha-1}+(2\theta)^{\alpha-1}).
$ 
Hence the lower bound of \eqref{typeI-minimax-key} holds if $|x|\geq2\theta$. 
Moreover, if $x\leq-2\theta$, then $|x|\geq2\theta$ and $|x-\theta|=\theta-x\geq-x=|x|$, so
\ba{
|x|^{\alpha-1}+(2\theta)^{\alpha-1}
\leq2|x|^{\alpha-1}\leq2|x-\theta|^{\alpha-1}.
}
If $x\geq2\theta$, then $x-\theta\geq\theta$ and thus 
\ba{
|x|^{\alpha-1}+(2\theta)^{\alpha-1}
\leq2^{\alpha-2}|x-\theta|^{\alpha-1}+3\cdot2^{\alpha-2}\theta^{\alpha-1}
\leq2^\alpha|x-\theta|^{\alpha-1},
}
where the first inequality is by Jensen's inequality. 
Therefore, the upper bound of \eqref{typeI-minimax-key} also holds if $|x|\geq2\theta$. 
Consequently, $f_\theta\in\mcl G(\theta,\alpha,1/2,b)$ for some constant $b>0$ depending only on $\alpha$. 

Now, Eq.(2.27) of \cite{tsybakov2008nonparametric} gives
\ba{
\int\bra{\sqrt{f_\theta(x)}-\sqrt{f_0(x)}}^2dx
&\leq\int_{-1}^1\bra{\frac{f_\theta(x)}{f_0(x)}-1}^2f_0(x)dx\\
&=\frac{\alpha}{2(\alpha-1)}\int_{-2\theta}^{2\theta}\frac{\psi_\theta(x)^2}{1-|x|^{\alpha-1}}dx
\leq\frac{\alpha}{2(\alpha-1)}\frac{3(2\theta)^{2\alpha-1}}{1-(2\theta)^{\alpha-1}}.
}
Hence \eqref{minimax-finally} holds for $g_T=f_{2\rho_T}$ and $g_0=f_0$. 
\end{proof}

\subsection{Proof of Assumption \ref{ass:cpcf}(ii) for LBHPG with gamma kernels}
\label{app:lb_hawkes_gamma_a2}
In this section, we show that $\mathrm{LBHPG}$ with common rate parameters $\beta_{ij}\equiv\beta>0, i, j=1, 2$ satisfies  \ref{ass:cpcf}(ii) with $\alpha=\min\{D_{12}, D_{21}\}$ when $\min\{D_{12}, D_{21}\} < 1$ and the stationary assumption (spectral radius $\rho(\boldsymbol{\alpha})$ of the matrix $\boldsymbol{\alpha}$ is smaller than $1$) holds.
Let
\[
h_a(t):=\frac{\beta^{a}}{\Gamma(a)}t^{a-1}e^{-\beta t}\mathbf 1_{(0,\infty)}(t),\qquad t\in\mathbb{R}, a>0,
\] $\boldsymbol{H} = (h_{D_{ij}})_{1\leq i, j\leq 2}$ , $\alpha_*=\min\{D_{12}, D_{21}\}(<1)$, and $D_* = \min\{D_{11}, D_{12}, D_{21}, D_{22}\}$.
Then, we have $\Phi = \boldsymbol{\alpha} \odot \boldsymbol{H}$ by definition, where $\odot$ is the Hadamard product. 
Recall $\Psi = \sum_{m\geq 1} \Phi^{(*m)}$ converges in $L^1(\mathbb{R})$ componentwise.

By \eqref{eq:hawkes-cov-density} and  \eqref{eq:cpcf_LBHPG} , it is sufficient to show the following proposition: 
\begin{proposition}\label{prop:shape_hawkes_cov_density}
Under the assumptions above, there exist $ C>0$ and $\delta>0$ such that 
\begin{align}
\alpha_{ij} h_{D_{ij}}(u)
  &\le \Psi_{ij}(u),
  && u>0,\ (i,j)\in\{(1,2),(2,1)\},
  \label{eq:Psi_lower_bound}\\
  \max\{\Psi_{12}(u),\Psi_{21}(u)\}
  &\le Cu^{\alpha_*-1},
  && 0<u<\delta,
  \label{eq:Psi_max_bound}\\
0 \le \int_{\mathbb R}\Psi_{i1}(s)\Psi_{i2}(s+u)\,ds
  &\le C|u|^{\alpha_*-1},
  && 0<|u|<\delta, i\in\{1, 2\}.
  \label{eq:Psi_convolution_bound}
\end{align}
\end{proposition}

In the following, we first provide several lemmas and then use them to establish
Proposition \ref{prop:shape_hawkes_cov_density}.
Before proceeding, we decompose $\Psi$ as
\begin{equation}\label{eq:psi_decompose_gamma}
\Psi=\Psi^s+\Psi^b,
\end{equation}
where
\[
\Psi^s := \sum_{m=1}^{M-1}\Phi^{(*m)},\qquad
\Psi^b := \sum_{m=M}^{\infty}\Phi^{(*m)},\qquad
M=\lceil 1/D_*\rceil+1.
\]

\begin{lemma}\label{lem:phi_m_convolution_path_expansion}
For $m\ge1$, $i,j\in\{1,2\}$ and $t>0$,
\[
(\Phi^{(*m)})_{ij}(t)
=\sum_{
    \substack{(i_0,\dots,i_m)\in\{1,2\}^{m+1},\\ i_0=j,\; i_m=i}
}
\Bigl(\prod_{\ell=1}^m \alpha_{i_\ell i_{\ell-1}}\Bigr)\,
h_{\sum_{\ell=1}^m D_{i_\ell i_{\ell-1}}}(t).
\]
\end{lemma}
\begin{proof}
The result follows from the gamma distribution's reproducibility.
\end{proof}

\begin{lemma}\label{lem:gamma_sup_a}
For every $a\ge1$, $\|h_a\|_\infty \le \beta$.
\end{lemma}
\begin{proof}
Fix $a\ge1$. 
Since the gamma density $h_a$ is log-concave on $[0, \infty)$, we have
\[
\|h_a\|_\infty \le \frac{1}{\sigma_a}
\]
by \cite[Eq.~(5.8)]{saumard2014log}, where $\sigma_a$ is the standard deviation of $h_a$.
Since the standard deviation of $h_a$ is $\sigma_a=\sqrt{a}/\beta$, we conclude that
\[
\|h_a\|_\infty \le \frac{\beta}{\sqrt{a}} \leq \beta.
\]
\end{proof}

\begin{lemma}\label{lem:psi_b_boundedness}
Each component of $\Psi^b$ is bounded on $\mathbb{R}$.
\end{lemma}
\begin{proof}
By Lemma~\ref{lem:phi_m_convolution_path_expansion}, for $m\ge M$ every component of $\Phi^{(*m)}$ is a gamma density
with shape parameter at least $mD_* > 1$. Hence Lemma~\ref{lem:gamma_sup_a} yields
\[
\|(\Phi^{(*m)})_{ij}\|_\infty
\le \beta\sum_{
    \substack{(i_0,\dots,i_m)\in\{1,2\}^{m+1},\\ i_0=j,\; i_m=i}
}\prod_{\ell=1}^m \alpha_{i_\ell i_{\ell-1}}
=\beta(\boldsymbol{\alpha}^m)_{ij}.
\]
Since $\rho(\boldsymbol{\alpha})<1$, the series $\sum_{m\ge1}\boldsymbol{\alpha}^m$ converges entrywise, hence
$\sum_{m\ge M}(\boldsymbol{\alpha}^m)_{ij}<\infty$ and therefore
\[
\|\Psi^b_{ij}\|_\infty
\le \sum_{m=M}^\infty \|(\Phi^{(*m)})_{ij}\|_\infty
\le \beta\sum_{m\ge M}(\boldsymbol{\alpha}^m)_{ij}<\infty.
\]
\end{proof}

\begin{proof}[Proof of Proposition \ref{prop:shape_hawkes_cov_density}]
Throughout the proof, $C>0$ and $\delta>0$ denote generic constants.

\medskip
\noindent\textbf{Proof of \eqref{eq:Psi_lower_bound}.}
Since $\Psi=\sum_{m\ge1}\Phi^{(*m)}$ and each term is nonnegative, we have $\Psi_{ij}\ge \Phi_{ij}=\alpha_{ij}h_{D_{ij}}$
on $(0,\infty)$, which yields \eqref{eq:Psi_lower_bound}.

\medskip
\noindent\textbf{Proof of \eqref{eq:Psi_max_bound}.}
Fix $(i,j)\in\{(1,2),(2,1)\}$. By \eqref{eq:psi_decompose_gamma},
\[
\Psi_{ij}(u)=\Psi^s_{ij}(u)+\Psi^b_{ij}(u),\qquad u>0.
\]
By Lemma~\ref{lem:phi_m_convolution_path_expansion} and the definition of $\Psi^s$ as a finite sum,
$\Psi^s_{ij}$ is a finite nonnegative linear combination of gamma densities $h_a$ (with common rate $\beta$) whose shape parameters satisfy $a\ge D_{ij}\ge \alpha_*$. 
Hence, there exists $C_s>0$ such that $\Psi^s_{ij}(u)\le C_s u^{\alpha_*-1}$ for all $u\in(0,1)$.
Moreover, Lemma~\ref{lem:psi_b_boundedness} gives $\|\Psi^b_{ij}\|_\infty<\infty$.
Choose $\delta\in(0,1)$ so that $u^{\alpha_*-1}\ge1$ on $(0,\delta)$. Then for $u\in(0,\delta)$,
\[
\Psi_{ij}(u)\le C_su^{\alpha_*-1}+\|\Psi^b_{ij}\|_\infty
\le (C_s+\|\Psi^b_{ij}\|_\infty)\,u^{\alpha_*-1}.
\]
Taking the maximum over $(i,j)=(1,2),(2,1)$ yields \eqref{eq:Psi_max_bound}.

\medskip
\noindent\textbf{Proof of \eqref{eq:Psi_convolution_bound}.}
Fix $i\in\{1,2\}$ and set
\[
I_i(u):=\int_{\mathbb R}\Psi_{i1}(s)\Psi_{i2}(s+u)\,ds,\qquad u\in\mathbb R.
\]
Non-negativity implies $I_i(u)\ge0$. We will prove the upper bound.

Using \eqref{eq:psi_decompose_gamma} and $\int_{\mathbb{R}}f(s)g(s+u)\,ds\le \|f\|_1\|g\|_\infty$ for nonnegative $f,g$, we obtain
\begin{align*}
I_i(u)
&=\int_{\mathbb{R}}(\Psi^s_{i1}+\Psi^b_{i1})(s)\,(\Psi^s_{i2}+\Psi^b_{i2})(s+u)\,ds\\
&\le \int_{\mathbb{R}}\Psi^s_{i1}(s)\Psi^s_{i2}(s+u)\,ds
    + \|\Psi_{i1}\|_1\|\Psi^b_{i2}\|_\infty
    + \|\Psi^b_{i1}\|_\infty\|\Psi_{i2}\|_1,
\end{align*}
where we used $\Psi^s_{ik}\le \Psi_{ik}$ and $\Psi^b_{ik}\le \Psi_{ik}$.
Since $\|\Phi_{ij}\|_1=\alpha_{ij}$ and $\|(\Phi^{(*m)})_{ij}\|_1=(\boldsymbol{\alpha}^m)_{ij}$, the assumption
$\rho(\boldsymbol{\alpha})<1$ implies $\|\Psi_{ik}\|_1=\sum_{m\ge1}(\boldsymbol{\alpha}^m)_{ik}<\infty$.
By Lemma~\ref{lem:psi_b_boundedness}, $\|\Psi^b_{ik}\|_\infty<\infty$. Hence, the last two terms are finite constants
independent of $u$, and since $\alpha_*<1$ we may shrink $\delta\in(0,1)$ so that these constants are absorbed by
$C|u|^{\alpha_*-1}$ on $0<|u|<\delta$ (using $|u|^{\alpha_*-1}\ge1$ there).
Therefore, it suffices to show that
\begin{equation}\label{eq:ss_bound}
\int_{\mathbb R}\Psi^s_{i1}(s)\Psi^s_{i2}(s+u)\,ds \le C|u|^{\alpha_*-1},
\qquad 0<|u|<\delta.
\end{equation}

For $a,b>0$, define
\[
f^{\mathrm{BG}}_{a, b}(u):=\int_{\mathbb R} h_a(s)\,h_b(s+u)\,ds,\qquad u\in\mathbb R.
\]
Then, $f^{\mathrm{BG}}_{a, b}$ is the probability density function of a bilateral gamma distribution \cite{kuchler2008shapes} with parameters \\$(\alpha_+,\lambda_+,\alpha_-,\lambda_-)=(b,\beta,a,\beta)$.
By \citet[Theorem 6.1]{kuchler2008shapes}, as $u\to0$ the density satisfies
$f^{\mathrm{BG}}_{a, b}(u)=O(|u|^{a+b-1})$ if $a+b<1$, and $f^{\mathrm{BG}}_{a, b}(u)=O(M(|u|))$ if $a+b=1$, where $M$ is slowly varying at $0$. If $a+b>1$, then $f^{\mathrm{BG}}_{a, b}$ is bounded in a neighborhood of $0$.
Consequently, for any $a,b>0$ with $a+b>\alpha_*$, there exists $C_{a,b}>0$ and
$\delta_{a,b}\in(0,1)$ such that
\begin{equation}\label{eq:gab_alpha_star}
f^{\mathrm{BG}}_{a, b}(u)\le C_{a,b}|u|^{\alpha_*-1},\qquad 0<|u|<\delta_{a,b}
\end{equation}
since $\alpha^*<1$ by the assumption.
Next, by Lemma~\ref{lem:phi_m_convolution_path_expansion},
each $\Psi^s_{ik}$ is a finite nonnegative linear combination of gamma densities $h_a$. Hence, the left-hand side of
\eqref{eq:ss_bound} is a finite nonnegative linear combination of $f^{\mathrm{BG}}_{a, b}(u)$.
If $i=1$, then $\Psi^s_{12}$ only involves shapes $b\ge D_{12}\ge \alpha_*$, while shapes $a$ in $\Psi^s_{11}$ are strictly positive; thus $a+b>\alpha_*$.
If $i=2$, then $\Psi^s_{21}$ only involves shapes $a\ge D_{21}\ge \alpha_*$, while shapes $b$ in $\Psi^s_{22}$ are strictly
positive; thus again $a+b>\alpha_*$.
Therefore, \eqref{eq:gab_alpha_star} applies to all pairs $(a,b)$ appearing in the linear combination, we can
take a common $\delta>0$ and $C>0$, which yields \eqref{eq:ss_bound}. This proves \eqref{eq:Psi_convolution_bound}.
\end{proof}

\section{Implementation and computational complexity}\label{app:computation}
In this section, we discuss the efficient computation of the kernel density estimator $\hat g_h(u)$ and the search strategy for its maximizer $\hat\theta_h$.
In our implementation, the expected time complexity for computing $\hat\theta_h$ from observations on $[0,T]$ with bandwidth $h$ scales as
\[
\begin{cases}
O\!\left(T\log T + T^2 h\right), & \text{under Assumption~\ref{ass:cpcf}(i)},\\
O\!\left(T\log T + T^2 h^\alpha\right), & \text{under Assumption~\ref{ass:cpcf}(ii) with } 0<\alpha<1.
\end{cases}
\]
This is better than a naive $O(T^2)$ approach that evaluates $\hat g_h(u)$ at each candidate $u$ by summing over all pairs, especially when the bandwidth $h$ is small.

\subsection{Algorithm for computing $\hat g_h$ on a grid}
Directly evaluating $\hat g_h$ on a grid $\{u_1, \dots, u_M\}\subset\mathbb{R}$ may be computationally expensive, roughly scaling with the product of the grid size and the number of data pairs.
To reduce computational cost, we employ an algorithm that iterates over relevant timestamp pairs and distributes their kernel weights onto nearby grid points.
This approach is particularly efficient when the bandwidth $h$ is small relative to the grid's range.

Let $N_1$ and $N_2$ be the underlying point processes, observed over the window $[0,T]$, and let
\(
\mathcal{T}_i := \{t_{i,1}<\cdots<t_{i,n_i}\}\subset[0,T],\ n_i := N_i([0,T]),\ i=1,2
\)
denote the corresponding observed event times.
Let $\mathcal{U} = \{u_1<\cdots<u_M\}$ be a sorted grid where we wish to evaluate the estimator, and define
$u_{\min}:=u_1$ and $u_{\max}:=u_M$.
For any $a<b$, define the set of relevant pairs in the lag window $[a,b]$ and the corresponding set of differences by
\[
\mathcal{P}(a,b)
:=
\{(x,y)\in \mathcal{T}_1\times \mathcal{T}_2:\ y-x\in[a,b]\},
\qquad
N_{\mathrm{pairs}}(a,b):=|\mathcal{P}(a,b)|,
\]
\[
\mathcal{D}(a,b)
:=
\{y-x:\ (x,y)\in\mathcal{P}(a,b)\}\subset[a,b].
\]

Algorithm~\ref{alg:cpcf_grid} outlines the procedure.
Instead of fixing $u$ and summing over all pairs, we iterate through each observed time $x \in \mathcal{T}_1$.
Using binary search in $\mathcal{T}_2$, we identify the range of $y \in \mathcal{T}_2$ for which the difference
$d = y - x$ lies within the lag window $[u_{\min}-h,\,u_{\max}+h]$, i.e., the set of lags that can influence at least one grid point through a bandwidth-$h$ kernel.
For each such difference $d$, we find the subset of grid points in $\mathcal{U} \cap [d-h, d+h]$ (again via binary search) and accumulate the kernel contribution at those grid points.
Throughout this section, we assume that $K$ is supported on $[-1,1]$.

\begin{algorithm}[t]
\caption{Computation of $\hat g_h(u)$ on a grid}\label{alg:cpcf_grid}
\begin{algorithmic}
\Require Sorted event times $\mathcal{T}_1$, $\mathcal{T}_2$, sorted grid $\mathcal{U}=\{u_1, \dots, u_M\}$, bandwidth $h$, \\ kernel $K$ supported on $[-1, 1]$.
\Ensure Values $G=(G_1,\dots,G_M)$ corresponding to $(\hat g_h(u_1),\dots,\hat g_h(u_M))$.
\State Initialize $G \gets (0, \dots, 0)$
\State $u_{\min} \gets u_1,\quad u_{\max} \gets u_M$
\For{$x \in \mathcal{T}_1$}
    \State Identify indices $[j_{\text{start}}, j_{\text{end}}]$ for $\mathcal{T}_2 \cap [x + u_{\min} - h,\ x + u_{\max} + h]$
    \For{$j \gets j_{\text{start}}$ \textbf{to} $j_{\text{end}}$}
        \State $y \gets \mathcal{T}_2[j]$
	        \State $d \gets y - x$
	        \State Identify indices $[k_{\text{start}}, k_{\text{end}}]$ for $\mathcal{U} \cap [d - h,\ d + h]$
	        \For{$k \gets k_{\text{start}}$ \textbf{to} $k_{\text{end}}$}
	            \State $G_k \gets G_k + \frac{1}{h} K\!\left(\frac{d - u_k}{h}\right)$
	        \EndFor
	    \EndFor
\EndFor
\State Scale $G$ by $\frac{T}{n_1 n_2}$, i.e., $G \gets \frac{T}{n_1 n_2}G$
\State \Return $G$
\end{algorithmic}
\end{algorithm}

\paragraph{Computational complexity.}
Only pairs in the lag window $[u_{\min}-h,\ u_{\max}+h]$ contribute to $\hat g_h$ evaluated on $\mathcal{U}$, and the number of such pairs is
$N_{\mathrm{pairs}}(u_{\min}-h,\ u_{\max}+h)$.
For each $x\in\mathcal{T}_1$, we locate the index range of
$\mathcal{T}_2\cap[x+u_{\min}-h,\ x+u_{\max}+h]$ by binary search in $\mathcal{T}_2$, which costs $O(\log n_2)$ per $x$
(and thus $O(n_1\log n_2)$ in total).
For each relevant pair $(x,y)\in\mathcal{P}(u_{\min}-h,\ u_{\max}+h)$ with $d=y-x$, we perform a binary search on $\mathcal{U}$ to locate
$\mathcal{U}\cap[d-h,d+h]$, which costs $O(\log M)$, and then update all grid points in that subarray.
Let $M_h := \sup_{u\in\mathbb R}|\mathcal{U}\cap[u-h,u+h]|$ denote the maximum local grid occupancy at scale $h$.
Thus, the total complexity is bounded by
\[
O\!\left(n_1 \log n_2 + N_{\mathrm{pairs}}(u_{\min}-h,\ u_{\max}+h)\bigl(\log M + M_h\bigr)\right).
\]
This is significantly faster than the naive $O\!\left(M\cdot N_{\mathrm{pairs}}(u_{\min}-h,\ u_{\max}+h)\right)$ approach when $M_h \ll M$,
which typically occurs when the bandwidth $h$ is small.
As suggested by Theorem~\ref{thm:cpcf}, small bandwidths are desirable in practice, especially when $g$ exhibits a sharp peak,
where the proposed implementation yields a substantial computational gain.

\subsection{Finding the maximizer of $\hat g_h$}
To compute the estimator $\hat{\theta}_h$, we need to find the global maximizer of $\hat g_h(u)$ within the range $[-r, r]$.
The objective function $\hat g_h$ may have many local optima, making it difficult for standard numerical optimization to converge to the global maximum.
However, when $K$ is piecewise linear, $\hat g_h(u)$ is also piecewise linear.
Hence, a maximizer over $[-r,r]$ can be found by evaluating $\hat g_h$ only at the kink locations induced by the differences $d=y-x$,
$x\in\mathcal{T}_1$, $y\in\mathcal{T}_2$.

Let $\mathcal{D}_{r,h}:=\mathcal{D}(-r-h,\ r+h)$ denote the set of lag differences in the window $[-r-h,r+h]$.
In particular, for the triangular kernel $K^\tri(x)=(1-|x|)\mathbf{1}_{[-1, 1]}(x)$, the function may change slope only at
\(
u=d-h,\ u=d,\ u=d+h
\) for each lag difference $d\in \mathcal{D}_{r,h}$.
Therefore, to obtain the maximizer $\hat{\theta}_h$, it suffices to evaluate $\hat g_h$ on the candidate set
\[
\mathcal{U}_{\mathrm{kink}}
:=
\Bigl(\mathcal{D}_{r,h}\ \cup\ \bigl(\mathcal{D}_{r,h}+h\bigr)\ \cup\ \bigl(\mathcal{D}_{r,h}-h\bigr)
\cup\ \{-r,r\}
\Bigr )\cap[-r,r].
\]
In an implementation, it is recommended to remove duplicates in $\mathcal{U}_{\mathrm{kink}}$
before sorting; otherwise the grid may contain repeated points and incur unnecessary work.

\paragraph{Computational complexity.}
The set $\mathcal{D}_{r,h}$ can be constructed while enumerating the relevant pairs $\mathcal{P}(-r-h,r+h)$, which takes
$O\!\left(N_{\mathrm{pairs}}(-r-h,\ r+h)\right)$ time up to constant-factor overhead.
Building $\mathcal{U}_{\mathrm{kink}}$ from $\mathcal{D}_{r,h}$ is linear in $|\mathcal{U}_{\mathrm{kink}}|$, and we may sort
$\mathcal{U}_{\mathrm{kink}}$ once to apply Algorithm~\ref{alg:cpcf_grid}.
Evaluating $\hat g_h$ on $\mathcal{U}_{\mathrm{kink}}$ using Algorithm~\ref{alg:cpcf_grid} has the same form as above, with $M$ replaced by
$|\mathcal{U}_{\mathrm{kink}}|$ and with $M_h$ defined as above but computed on the grid $\mathcal{U}_{\mathrm{kink}}$:
\[
O\!\left(n_1 \log n_2 + N_{\mathrm{pairs}}(-r-h,\ r+h)\bigl(\log |\mathcal{U}_{\mathrm{kink}}| + M_h\bigr)\right).
\]
Finally, the maximizer is obtained by a single pass over the evaluated values, which costs $O(|\mathcal{U}_{\mathrm{kink}}|)$.

\paragraph{Scaling in $T$ and $h$.}
For the simple stationary bivariate point process $N=(N_1, N_2)$ with intensities $\lambda_1,\lambda_2$ and CPCF $g$, we have $\E[n_i]=\lambda_i T$ for $i=1,2$.
Moreover, by the definition of the cross-intensity function, we have
\[
\E\!\left[N_{\mathrm{pairs}}(a,b)\right]
=\int_{(0,T]^2}\mathbf{1}\{y-x\in[a,b]\}\,\lambda_{12}(y-x)\,dx\,dy
\approx \lambda_1\lambda_2\,T\int_a^b g(u)\,du.
\]
If $T\gg r+h$ and $[a,b]\subset[-r-h,r+h]$, then $\E[N_{\mathrm{pairs}}(a,b)] \asymp \lambda_1\lambda_2\,T\int_a^b g(u)\,du$.
In particular,
\[
\E\!\left[N_{\mathrm{pairs}}(-r-h,r+h)\right]=O(T),
\qquad
\E[n_1\log n_2]=O(T\log T),
\]
where the hidden constant depends on $g$ through its mass on $[-r-h, r+h]$.

Next, note that $|\mathcal{U}_{\mathrm{kink}}|\le 3|\mathcal{D}_{r,h}|+2\le 3N_{\mathrm{pairs}}(-r-h,r+h)+2$,
so $|\mathcal{U}_{\mathrm{kink}}|=O(T)$ in expectation, and the one-time sorting cost is $O(T\log T)$.
For the local update cost on the kink grid, observe that for any $u\in[-r,r]$,
\[
|\mathcal{U}_{\mathrm{kink}}\cap[u-h,u+h]|
\;\le\;
3\,|\mathcal{D}_{r,h}\cap[u-2h,u+2h]|+2
\;\le\;
3\,N_{\mathrm{pairs}}(u-2h,u+2h)+2.
\]
Taking expectations yields
\[
\E\!\left[|\mathcal{U}_{\mathrm{kink}}\cap[u-h,u+h]|\right]
= O\!\left(T\int_{u-2h}^{u+2h}g(v)\,dv\right).
\]
Taking the supremum over $u\in[-r,r]$ gives the worst-case bound
\[
M_h = O\!\left(T\sup_{u\in[-r,r]}\int_{u-2h}^{u+2h} g(v)\,dv\right).
\]
In particular, if $g$ is bounded (e.g., under \ref{ass:cpcf}(i)), then
$\int_{u-2h}^{u+2h} g(v)\,dv = O(h)$ uniformly in $u$, so $M_h=O(Th)$ and the expected time bound
reduces to $O(T\log T + T^2 h)$.
If $g$ is unbounded at $\theta^\ast$ as in \ref{ass:cpcf}(ii), then
$\int_{\theta^\ast-2h}^{\theta^\ast+2h} g(v)\,dv = O(h^\alpha)$. 
Thus, we may have $M_h=O(Th^\alpha)$, leading to an expected time bound of $O(T\log T + T^2 h^\alpha)$.
On the other hand, the naive implementation costs $O(N_{\mathrm{pairs}}\times |\mathcal{U}_{\mathrm{kink}}|) \approx O(T^2)$.

\section{Bandwidth selection by cross-validation}\label{sec:cv}
In this section, we introduce additional bandwidth-selection methods for the estimator $\hat\theta_h$ based on cross-validation and assess their performance in simulation studies.

In the context of modal regression, \citet{chen2016nonparametric} propose choosing the bandwidth by minimizing the size of prediction sets associated with the modal regression function, while \citet{zhou2019bandwidth} introduce a cross-validation criterion (CVM) that penalizes the squared distance between the responses and the estimated modal set and includes an explicit penalty for the number of modes.
Motivated by these developments, we design a cross-validation scheme that directly evaluates how well the maximizer set obtained from the training part of the point process predicts the empirical lag differences observed in the test part.  

Fix an integer $K_{\mathrm{cv}}\ge2$ and split the observation window $[0,T]$ into $K_{\mathrm{cv}}$ disjoint subintervals
$I_1,\dots,I_{K_{\mathrm{cv}}}$ of equal length. Here $K_{\mathrm{cv}}$ denotes the number of folds. For the $j$th fold, we regard $I_j$ as a test interval and
$[0,T]\setminus I_j$ as the corresponding training interval.
On the training interval we compute the kernel estimator $\hat g_h^{(-j)}$ based on
$N\cap([0,T]\setminus I_j)$ and its (possibly set-valued) maximizer set
\[
  M_h^{(-j)}
  := \argmax_{u\in [-r, r]} \hat g^{(-j)}_h(u),
  \qquad h\in \mathcal{H}_T,
\]
where $\mathcal{H}_T$ is the finite bandwidth grid.

For a set
$M\subset[-r, r]$ and $z\in\mathbb R$ we write
\[
  d(z,M):=\inf_{u\in M}|z-u|
\]
for the distance from $z$ to $M$.  Given a test interval $I_j$, we define the set of observed lag
differences restricted to $[-r, r]$ by
\[
  \Delta_j(r)
  := \bigl\{ y-x ;
              x\in \mathcal{T}_{1,j},\ y\in \mathcal{T}_{2,j},\ -r\le y-x\le r\bigr\},
\]
where $\mathcal{T}_{i,j}$ denotes the (finite) set of event times of $N_i$ that fall in $I_j$
for $i\in\{1,2\}$. We interpret $\Delta_j(r)$ as a multiset, i.e., lag differences are counted
with multiplicity, and write $n_j := |\Delta_j(r)|$ for the number of elements.
In practice, the fold length should be chosen large relative to $r$ so that each $I_j$ contains
enough pairs with lag in $[-r,r]$.
When $n_j\ge1$, we enumerate the elements of $\Delta_j(r)$ as $(d_{j,1},\dots,d_{j,n_j})$ (in an arbitrary order) and introduce the
distances
\[
  \delta_{j,\ell}(M) := d(d_{j,\ell},M), \qquad \ell=1,\dots,n_j,
\]
with order statistics
$\delta_{j,(1)}(M)\le\cdots\le\delta_{j,(n_j)}(M)$.
Fix a trimming parameter $\tau\in(0,1]$ and a minimum count $n_{\min}\in\mathbb N$, and set
\[
  k_j := \max\{\lceil \tau n_j\rceil,\ n_{\min}\},\qquad
  \varepsilon_j(M) :=
  \begin{cases}
    \delta_{j,(k_j)}(M), & n_j\ge k_j,\\[2pt]
    +\infty, & n_j<k_j.
  \end{cases}
\]
If $n_j<k_j$, we set $L_{\mathrm{nearest}}(M;I_j)=+\infty$, regardless of $M$, effectively discarding bandwidths for which the test fold contains too few lag differences.
Under this notation, we consider the following loss functions on $I_j$ for a finite candidate maximizer set $M\subset[-r, r]$ (so $|M|$ is its cardinality):
\begin{itemize}
  \item \textbf{MSE-type loss}:
        \[
          L_{\mathrm{mse}}(M;I_j)
          := \frac{|M|^2}{n_j} \sum_{\ell=1}^{n_j} \delta_{j,\ell}(M)^2,
        \]
        with the convention that folds with $n_j=0$ are excluded from the CV average.
        The factor $|M|^2$ penalizes bandwidths that produce many maximizers, playing a stabilizing role in the spirit of \citet{zhou2019bandwidth}.
	  \item \textbf{Nearest-range loss}:
	        \[
	          L_{\mathrm{nearest}}(M;I_j)
	          := \mathrm{Leb}\bigl(\{x\in\mathbb R : d(x,M)\le\varepsilon_j(M)\}\bigr).
	        \] The set
	        $\{x\in\mathbb R:d(x,M)\le\varepsilon\}$ is the $\varepsilon$-neighborhood of $M$, and
	        $\mathrm{Leb}(\cdot)$ is its total length. By construction, $\varepsilon_j(M)$ is chosen
	        so that at least $k_j=\max\{\lceil\tau n_j\rceil,n_{\min}\}$ of the test lag differences
	        lie within distance $\varepsilon_j(M)$ of $M$, so $L_{\mathrm{nearest}}(M;I_j)$ is the
	        length of the smallest neighborhood of $M$ covering that trimmed fraction. The minimum
	        count $n_{\min}$ improves numerical stability. This loss function is based on the
	        prediction set approach of \citet{chen2016nonparametric}.
\end{itemize}
Let $J:=\{j\in\{1,\dots,K_{\mathrm{cv}}\}: n_j\ge1\}$ denote the set of folds with at least one lag
difference in $[-r, r]$.
Given a choice of loss function $L\in\{L_{\mathrm{mse}},L_{\mathrm{nearest}}\}$, we define the
$K_{\mathrm{cv}}$-fold CV score for $h\in \mathcal{H}_T$ by
\[
  \mathrm{CV}(h)
  := \frac{1}{|J|} \sum_{j\in J} L\bigl(M_h^{(-j)}; I_j\bigr).
\]
Our cross-validated bandwidth is then chosen as
\[
  \hat h_{\mathrm{CV}} \in \arg\min_{h\in \mathcal{H}_T} \mathrm{CV}(h)
\]
If the minimizer is not unique, we select the smallest $h$ in $\mathcal{H}_T$.
We finally obtain the adaptive estimator $\hat{\theta}_{\hat{h}_{\mathrm{CV}}}$.

To compare the performance of different loss functions and their accompanying tuning parameters, we conduct a series of simulation experiments under the following common design. The candidate bandwidths are $h \in \{10^{-1},10^{-2},10^{-3},10^{-4},10^{-5},10^{-6}\}$, $r=1$, and the observation window is $[0,T]$. The observation horizon takes the values $T \in \{1000,2000,4000,8000\}$; for each combination of model, estimator, and $T$, we generate $5000$ Monte Carlo replicates.
The cross-validation criteria considered are the nearest-range loss $L_{\mathrm{nearest}}$ with trimming levels $\tau \in \{0.01,0.025,0.05\}$ and the MSE-type loss $L_{\mathrm{mse}}$. 
For comparison, we also include the Lepski selector with $A_T = \log\log T$. 
In the plots, the nearest-range CV curves for different $\tau$ are distinguished by a red color gradient; the MSE-based CV curves are shown in blue; and the curve of the Lepski estimator is shown in green. 
We set $K_{\mathrm{cv}}=5$ and $n_{\min}=5$. For reference, each panel also shows the theoretical rate $T^{-1/\beta_\alpha}$ as a black dashed line.

\begin{figure}[t]
    \centering
    \includegraphics[width=1\linewidth]{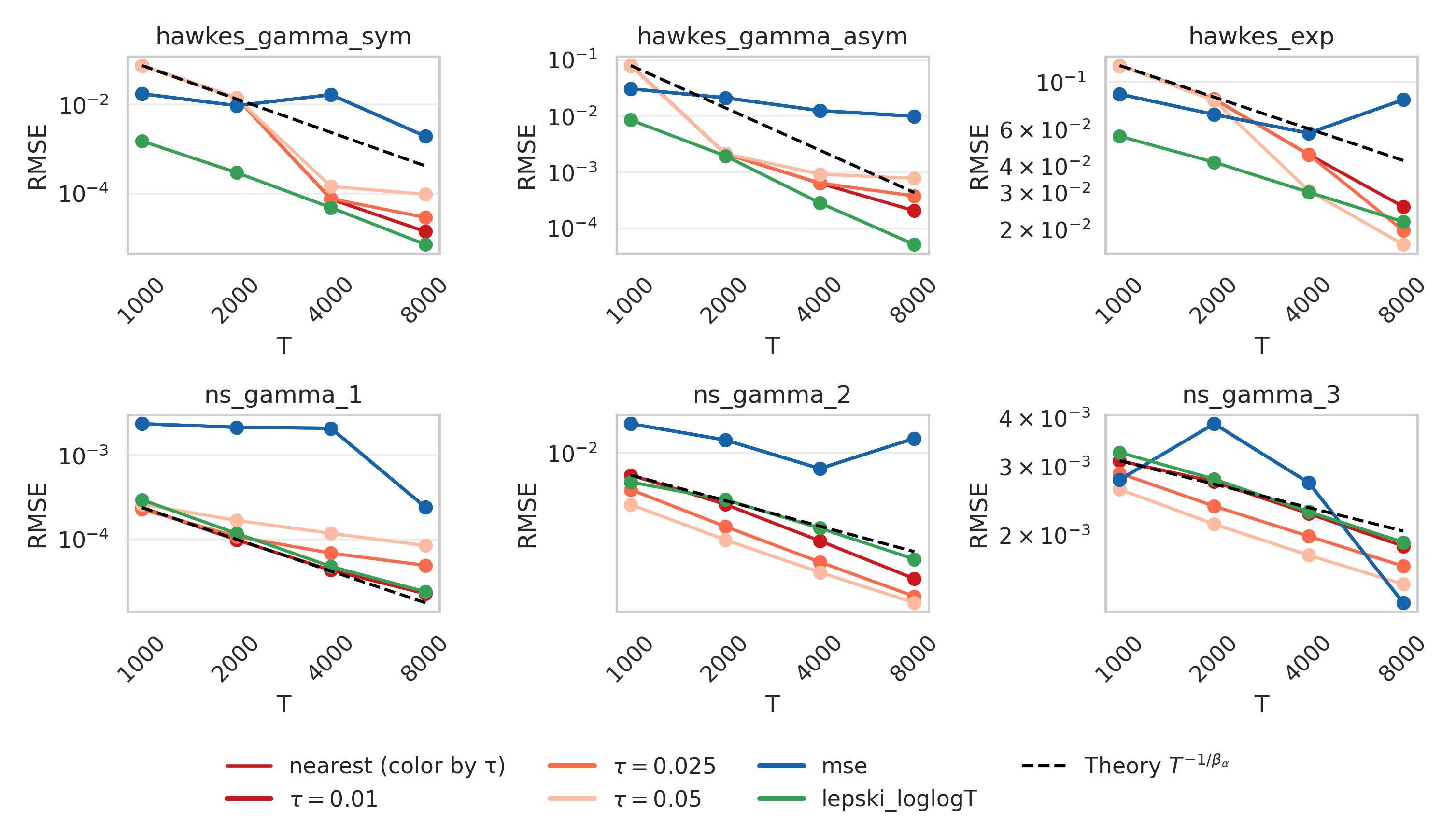}
    \caption{Performance comparisons between the bandwidth-selection methods across scenarios: RMSE versus $T$ on log--log axes for CV using $L_{\mathrm{nearest}}$ (red gradient by $\tau$), CV using $L_{\mathrm{mse}}$ (blue), and Lepski's method with $A_T=\log\log T$ (green); the dashed black line shows the theoretical rate $T^{-1/\beta_\alpha}$.}
    \label{fig:cv}
\end{figure}

Figure \ref{fig:cv} compares the finite-sample performance of the different bandwidth selection schemes across the six data-generating models described in Section \ref{sec:sim-models}. Several systematic patterns emerge.

First, the cross-validation criterion based on the MSE-type loss is somewhat unstable.
In five out of the six models, the RMSE of the MSE-CV estimator decreases only slowly, and sometimes not at all, as $T$ increases, compared to the theoretical slope. This suggests that the MSE-type loss is not well-suited to selecting the bandwidth for our cases.

The nearest-range loss $L_{\mathrm{nearest}}$ behaves more favorably, but its performance depends on the trimming parameter $\tau$. For models with sharper CPCFs, i.e., with $\alpha$ smaller than $1$, smaller values of $\tau$ tend to work better.
In contrast, for smoother models with $\alpha>1$, larger values of $\tau$ become competitive or even preferable.
In other words, the optimal choice of $\tau$ appears to be $\alpha$-dependent: aggressive trimming is beneficial when $g$ has a very sharp peak, whereas milder trimming is adequate when $g$ is flatter around its maximum.
A notable exception is the asymmetric Hawkes model (\texttt{hawkes\_gamma\_asym}). 
In this case, the nearest-range CV estimator improves more slowly and less regularly with $T$
across the $\tau$-values considered. One possible contributing factor is that
$L_{\mathrm{nearest}}$ is built from symmetric neighborhoods of the estimated maximizer set,
while the within-fold lag differences in this asymmetric setting may be skewed in finite samples;
in such cases, symmetric neighborhoods may be less informative for selecting $h$.

When compared with the Lepski-type procedure, the nearest-range CV estimator is inferior in the small-$\alpha$ models. Indeed, for \texttt{hawkes\_gamma\_sym}, \texttt{hawkes\_gamma\_asym}, and \texttt{ns\_gamma\_1}, the Lepski estimator shows lower RMSE. 
For smoother models with $\alpha>1$, however, the best-tuned nearest-range CV estimator becomes competitive. 
Overall, these experiments indicate that cross-validation based on $L_{\mathrm{nearest}}$ is a promising alternative, but it may require a delicate choice of $\tau$.

Taken together, our findings suggest the following practical recommendation. 
Among the bandwidth selectors we have examined, the Lepski method stands out as the most robust option. It requires only the slowly diverging threshold $A_T$, shows stable behavior across all models, and nearly attains the minimax rate $T^{-1/\beta_\alpha}$ both theoretically and in numerical experiments. 
Nearest-range cross-validation can be competitive, especially for smoother CPCFs, but it requires
choosing the trimming level $\tau$ in addition to the bandwidth, and its best choice is not yet
well understood theoretically. 
For this reason, we currently recommend the Lepski-type bandwidth choice as a default method for estimating the lead-lag time.

\paragraph{Acknowledgements}

We thank participants at the ``Big Data and Artificial Intelligence in Econometrics, Finance, and Statistics'' workshop at University of Chicago, October 2-4, 2025, the quantitative finance seminar at National University of Singapore, October 24, 2025, the KAKENHI symposium at Tsukuba University, October 30-31, 2025, Nakanoshima Workshop at Osaka University, December 5-6, 2025, and CFE-CMStatistics 2025 at University of London, December 13-15, 2025, for insightful comments and constructive suggestions on this work. 
Takaaki Shiotani's work was partly supported by Grant-in-Aid for JSPS Fellows (25KJ0933) and World-leading Innovative Graduate Study for Frontiers of Mathematical Sciences and Physics. 
Yuta Koike's work was partly supported by JST CREST Grant Number JPMJCR2115 and JSPS KAKENHI Grant Numbers JP22H00834, JP22H01139. 

\bibliography{ts}

\end{document}